\documentclass[11pt]{article}

\usepackage{verbatim,latexsym,amsfonts,amsmath,amssymb,graphicx,fancyhdr,hyperref,asymptote,enumitem}
\usepackage{appendix,latexsym,amsfonts,amsmath,amssymb,graphicx,hyperref,amsthm,soul,verbatim,authblk,enumitem}
\usepackage[framemethod=tikz]{mdframed}

\newcommand{\EE}{\mathbb{E}}
\newcommand{\PP}{\mathbb{P}}

\newcommand{\R}{\mathbb{R}}
\newcommand{\C}{\mathbb{C}}

\newcommand{\HH}{\mathbb{H}}
\newcommand{\N}{\mathbb{N}}
\newcommand{\D}{\mathbb{D}}
\newcommand{\Z}{\mathbb{Z}}

\newcommand{\TT}{\mathbb{T}}

\newcommand{\pa}{\partial}

\newcommand{\K}{{\cal K}}

\newcommand{\F}{{\cal F}}

\def\eps{\varepsilon}
\def\til{\widetilde}
\def\ha{\widehat}
\def\sem{\setminus}
\def\lin{\overline}

\def\M{{\cal M}}
\def\conf{\stackrel{\rm Conf}{\twoheadrightarrow}}
\def\rt{\rightleftharpoons}

\DeclareMathOperator{\ccap}{cap} \DeclareMathOperator{\Cont}{Cont}
 \DeclareMathOperator{\tip}{tip}
 
 \DeclareMathOperator{\diam}{diam}
\DeclareMathOperator{\dist}{dist} 
\DeclareMathOperator{\hcap}{hcap} \DeclareMathOperator{\id}{id}
\DeclareMathOperator{\Imm}{Im } \DeclareMathOperator{\Ree}{Re }

\DeclareMathOperator{\mA}{m} 
 
 \DeclareMathOperator{\cc}{c}
 \DeclareMathOperator{\doub}{doub}
\DeclareMathOperator{\lloop}{lp} \DeclareMathOperator{\bub}{bb}

\theoremstyle{plain}
\newtheorem{Theorem}{Theorem}[section]
\newtheorem{Lemma}[Theorem]{Lemma}
\newtheorem{Corollary}[Theorem]{Corollary}
\newtheorem{Proposition}[Theorem]{Proposition}
\theoremstyle{definition}
\newtheorem{Definition}[Theorem]{Definition}
\newtheorem{Remark}[Theorem]{Remark}
\numberwithin{equation}{section}
\newcommand{\BGE}{\begin{equation}}
\newcommand{\BGEN}{\begin{equation*}}
\newcommand{\EDE}{\end{equation}}
\newcommand{\EDEN}{\end{equation*}}

\begin{document}
\title{SLE Loop Measures}
\author{Dapeng Zhan\thanks{Research partially supported by NSF grant  DMS-1056840 and Simons Foundation grant \#396973.}}
\affil{Michigan State University}
\maketitle

\begin{abstract}
We use Minkowski content (i.e., natural parametrization) of SLE to construct several types of SLE$_\kappa$ loop measures for $\kappa\in(0,8)$. First, we construct rooted SLE$_\kappa$ loop measures in the Riemann sphere $\ha \C$, which satisfy M\"obius covariance, conformal Markov property, reversibility, and space-time homogeneity, when the loop is parametrized by its $(1+\frac \kappa 8)$-dimensional Minkowski content. Second, by integrating rooted SLE$_\kappa$ loop measures, we construct the unrooted SLE$_\kappa$ loop measure in $\ha\C$, which satisfies M\"obius invariance and reversibility. Third, we extend the SLE$_\kappa$ loop measures from $\ha\C$ to subdomains of $\ha\C$ and to two types of Riemann surfaces using Brownian loop measures, and obtain conformal invariance or covariance of these measures. Finally, using a similar approach, we construct SLE$_\kappa$ bubble measures in simply/multiply connected domains rooted at a boundary point. The SLE$_\kappa$ loop measures for $\kappa\in(0,4]$ give examples of Malliavin-Kontsevich-Suhov loop measures for all $\cc\le 1$. The space-time homogeneity of rooted SLE$_\kappa$ loop measures in $\ha C$ answers a question raised by Greg Lawler.
\end{abstract}

\tableofcontents

\section{Introduction}
\subsection{Overview}
The Schramm-Loewner evolution (SLE), introduced by Oded Schramm in 1999 (\cite{Sch}), is a one-parameter ($\kappa\in(0,\infty)$) family of probability measures on non-self-crossing curves, which has received a lot of attention since then.
It has been shown that, modulo time parametrization, the interface of several discrete lattice models at criticality 
have SLE$_\kappa$ with different parameters $\kappa$ as their scaling limits. The reader may refer to \cite{Law1,RS} for basic properties of SLE.

There are several versions of SLE$_\kappa$ curves in the literature. For most of them, the initial point and the terminal point of the SLE$_\kappa$ curve are different. Motivated by the Brownian loop measure  constructed in \cite{loop}, people have been considering the construction of a new version of SLE called  SLE$_\kappa$  loops, which locally looks like an ordinary SLE$_\kappa$ curve, starts and ends at the same point, and satisfies some prerequired properties.

In this paper we focus on the SLE with parameter $\kappa\in(0,8)$, which has Hausdorff dimension
 $d:=1+\frac\kappa 8\in(1,2)$  (cf.\ \cite{Bf}), and possesses natural parametrization (cf.\ \cite{LS,LZ}) that agrees with its $d$-dimensional Minkowski content (cf.\ \cite{LR}). Lawler and Sheffield introduced the natural  parametrization of SLE in \cite{LS} in order to describe the scaling limits of discrete random paths with their natural length. So far the convergence of loop-erased random walk  to SLE$_2$ with  natural parametrization has been established (cf.\ \cite{LERW-NP2}).

Besides conformal invariance or covariance, an SLE$_\kappa$ loop is expected to satisfy the space-time homogeneity when it is parametrized by its natural parametrization, i.e., Minkowski content. The existence of such SLE$_\kappa$ loops was conjectured by Greg Lawler.

Similar to the Brownian loop, the ``law'' of an SLE$_\kappa$ loop can not be a probability measure or a finite measure. Instead, it should be a $\sigma$-finite infinite measure. We will call it an SLE$_\kappa$ loop measure to emphasize this fact.

In \cite{Wer-loop} Werner used the Brownian loop measure to construct an essentially unique  measure on the space of simple loops in any Riemann surface, which satisfies conformal invariance and the restriction property, and has a close relation with SLE$_{8/3}$.

Inspired by Malliavin's work \cite{Malliavin} and SLE theory, Kontsevich and  Suhov conjectured in \cite{KS} that for every  $\cc\le 1$, there exists a unique locally conformally covariant measure on simple loops in a Riemann surface with values in a certain determinant bundle. Furthermore, they proposed a reduction of this problem, to construct a scalar measure on simple loops in $\C$ surrounding the origin, satisfying a restriction covariance property. The parameter $\cc$ in their conjecture is the central charge from conformal field theory (CFT). It is related to the parameter $\kappa$ for SLE by the formula:
\BGE \cc=\frac{(6-\kappa)(3\kappa-8)}{2\kappa}.\label{cc}\EDE
For $\cc=0$ (i.e., $\kappa=8/3$), their measure is Werner's measure. For other $\cc\le 1$, their measure should correspond to the SLE$_\kappa$ loop measure for some $\kappa\le 4$.

A loop version of SLE called conformal loop ensemble (CLE$_\kappa$) was constructed for $\kappa\in(8/3,8)$ by Sheffield and Werner (cf.\ \cite{CLE})
in order to describe the scaling limit of a full collection of interfaces of critical lattice models. A CLE is a random collection of non-crossing loops in a simply connected domain. Every loop in a CLE$_\kappa$ looks locally like an SLE$_\kappa$ curve.  CLE is different from the SLE loop here because the latter object is a single loop.

Kassel and Kenyon constructed in \cite{KK} natural probability measures on cycle-rooted spanning trees (CRSTs). A CRST on a graph $G$ is a connected subgraph, which contains a unique cycle called unicycle. They proved that, if $G$ approximates a conformal annulus $\Sigma$, as the mesh size tends to $0$, the law of the unicycle of a uniform CRST on $G$, conditional on the event that the unicycle  separates the two boundary components of $\Sigma$,  converges weakly to a probability measure on simple loops in $\Sigma$ separating the two boundary components of $\Sigma$. They proposed a question whether this limit measure can be constructed  via a stochastic differential equation, like a variant of SLE$_2$ defined on Riemann surfaces. The limit  measure was later studied in \cite{Dubedat-loop} using a different approach, and it was explained there that this gives an example of a Malliavin-Kontsevich-Suhov loop measure for $\cc=-2$, i.e., $\kappa=2$.

Kemppainen and Werner defined (\cite{KW}) unrooted SLE$_\kappa$ loop measure in $\ha\C$ for $\kappa\in(8/3,4]$ as the intensity measure of a nested whole-plane CLE$_\kappa$, and proved that this measure satisfies M\"obius invariance and is the only invariant measure under various Markov kernels defined using CLE. 
They used the loop measure to prove the M\"obius invariance of nested CLE on $\ha\C$. They also defined a rooted SLE$_\kappa$ loop measure as a suitable scaling limit of their unrooted loop measure restricted to the event that the curve passes through a small disc centered at a marked point, and claimed that the limit converges\footnote{Werner told the author privately that they were able to prove that the rooted loop measure is well defined and satisfies the conformal Markov property (CMP) as described in the current paper (Theorem \ref{Thm-loop-measure} (ii)). Given this fact, using the uniqueness statement  (Theorem \ref{Thm-loop-measure} (vii)), we see that the loop measures constructed in the current paper for $\kappa\in(8/3,4]$ agree with Kemppainen-Werner's measures.}.

Another natural object is the SLE$_\kappa$ bubble measure, which is similar to the Brownian bubble measure constructed in \cite{LSW-8/3}. In the same paper, an SLE$_{8/3}$ bubble measure was constructed. Later in \cite{CLE}, SLE$_\kappa$ bubble measures for $\kappa\in(8/3,4]$ were constructed by conditioning a CLE loop to touch a boundary point.

Field and Lawler have also been working on the construction of SLE loops (\cite{Talk}). They have constructed SLE loops rooted at an interior point in the whole plane and in simply connected domains, and are able to verify that the measures are conformally covariant. Benoist and Dub\'edat (\cite{Benoist-Dubedat}) have been working on the construction of SLE loops using flow lines of Gaussian free field, a natural object from Imaginary Geometry (\cite{MS1,MS4}).

\subsection{Main results}
In this paper, we construct several types of SLE$_\kappa$ loop measures for all $\kappa\in(0,8)$.
 Below is a rough version of the theorem about rooted SLE$_\kappa$ loop measures in $\ha\C$ (for complete and rigorous statements, see Theorem \ref{Thm-loop-measure} for details).

\begin{Theorem}
	Let $\kappa\in(0,8)$ and $d=1+\frac{\kappa}{8}$. There is a $\sigma$-finite measure $\mu^1_0$ on the space of (oriented) nondegenerate loops rooted at $0$ such that, if $\gamma$ follows the ``law'' of $\mu^1_0$, then the following hold.
	\begin{enumerate} [label=(\roman*)]
		\item ({\bf Conformal Markov property}) For any stopping time $\tau$ that does not happen at the initial time, conditional on the part of $\gamma$ before $\tau$ and the event that $\tau$ happens before the loop returns to $0$, the rest part of $\gamma$ is a chordal SLE$_\kappa$ curve.
		\item ({\bf Space-time homogeneity}) We may parametrize $\gamma$ periodically with period $p$ equal to the ($d$-dimensional) Minkowski content of $\gamma$, such that $\gamma(0)=0$, and for any $a<b\le a+p$, the Minkowski content of  $\gamma([a,b])$ equals $b-a$. Suppose $\gamma$ has this parametrization. For any deterministic number $a\in\R$, if we reroot the loop at $\gamma(a)$, which means that we define a new loop ${\cal T}_a(\gamma)$ by ${\cal T}_a(\gamma)(t)=\gamma(a+t)-\gamma(a)$, then the ``law'' of ${\cal T}_a(\gamma)$ is also $\mu^1_0$.
		\item ({\bf Reversibility}) The reversal of $\gamma$ also has the ``law'' $\mu^1_0$.
		\item ({\bf M\"obius covariance}) For every M\"obius transformation $F$ that fixes $0$, we have $F(\mu^1_0)=|F'(0)|^{2-d}\mu^1_0$.
		\item ({\bf Finiteness of big loops}) For any $r>0$, (a) the $\mu^1_0$ measure of loops with diameter  $>r$ is finite; (b) the $\mu^1_0$ measure of loops with Minkowski content  $>r$ is finite.
		\item ({\bf Uniqueness})  The measure $\mu^1_0$ is determined by (i) and (v.a) up to a constant factor.
	\end{enumerate} \label{short}
\end{Theorem}

Here we remark that the conformal Markov property (CMP) is an essential property that characterizes SLE. The CMP of the rooted SLE$_\kappa$ loop measure justifies its name, and allows us to apply the SLE-based results and arguments to study SLE$_\kappa$ loop measures. The space-time homogeneity gives a positive answer to Lawler's conjecture.

The construction of rooted SLE$_\kappa$ loop measure uses two-sided whole-plane SLE$_\kappa$. A two-sided whole-plane SLE$_\kappa$ is a random loop in $\ha\C$ passing through two distinct marked points, which is characterized by the property that, conditional on any arc on the loop connecting the two marked points, the other arc is a chordal SLE$_\kappa$ curve. Although this is also an SLE$_\kappa$ loop, it does not satisfy the space-time homogeneity that we want.

The measure $\mu^1_0$ in Theorem \ref{short} is constructed by integrating the laws of two-sided whole-plane SLE$_\kappa$ curves with marked points being $0$ and  $z\in\C\sem \{0\}$ against the function $|z|^{2(d-2)}$, and then unweighting the measure of loop by the Minkowski content of the loop. The proof of the theorem makes use of the reversibility of two-sided whole-plane SLE$_\kappa$ curves (\cite{MS3,MS4,reversibility}) and the decomposition of chordal SLE$_\kappa$ in terms of two-sided radial SLE$_\kappa$ (\cite{Fie,decomposition}).

A corollary of this theorem (Corollary \ref{sssi}) is that if a two-sided whole-plane SLE$_\kappa$ curve $\gamma$ from $\infty$ to $\infty$ passing through $0$ is parametrized by $d$-dimensional Minkowski content with $\gamma(0)=0$, then it becomes a self-similar process of index $\frac 1d$ with stationary increments. This result was later used in \cite{Holder} to study the H\"older regularity and dimension property of SLE with natural parametrization. 

After obtaining rooted SLE$_\kappa$ loop measures, we construct the unrooted SLE$_\kappa$ loop measure $\mu^0$ in $\ha\C$ by integrating SLE$_\kappa$ loop measures rooted at different $z\in\C$ against the Lebesgue measure, and then unweighting the measure by the Minkowski content of the loop.  The unrooted SLE$_\kappa$ loop measure satisfies M\"obius invariance and reversibility.

After constructing SLE loops in $\ha\C$, we turned to the construction of SLE loops in subdomains of $\ha\C$. We follow Lawler's approach in \cite{Law-mult} about defining SLE in multiply connected domains using Brownian loop measures. At first, we tried to define rooted/unrooted SLE$_\kappa$ loop measures in a subdomain $D$ of $\ha\C$ by
$$\mu^1_{D;z}={\bf 1}_{\{\cdot\subset D\}}e^{\cc\mu^{\lloop}({\cal L}(\cdot, D^c))}\cdot \mu^1_z,\quad
\mu^0_D={\bf 1}_{\{\cdot\subset D\}}e^{\cc\mu^{\lloop}({\cal L}(\cdot,D^c))}\cdot \mu^0,$$
where $\mu^{\lloop}$ is the Brownian loop measure in $\ha\C$ defined in \cite{loop}, ${\cal L}(\gamma,D^c)$ is the family of loops in $\ha\C$ that intersect both $\gamma$ and $D^c$, and $\cc$ is the central charge given by (\ref{cc}). However, as pointed out by Laurie Field, the quantity $\mu^{\lloop}({\cal L}(\gamma,D^c))$ is not finite for any curve $\gamma$ in $D$, and the correct alternative is the normalized Brownian loop measure introduced in \cite{normalize}.

The normalized Brownian loop measure introduced in \cite{normalize} is the following limit:
\BGE \Lambda^*(V_1,V_2):=\lim_{r\downarrow 0} [\mu^{\lloop}_{\{|z-z_0|>r\}}({\cal L}(V_1,V_2))-\log\log(1/r)],\label{normalized-Brownian}\EDE
where $\mu^{\lloop}_{\{|z-z_0|>r\}}$ is the Brownian loop measure in $\{|z-z_0|>r\}$, and $z_0\in\C$. It was proved in \cite{normalize} that the limit converges to a finite number if $V_1$ and $V_2$ are disjoint compact subsets of $\ha\C$; and the value does not depend on the choice of $z_0$, and satisfies M\"obius invariance. Thus, the correct way to define SLE$_\kappa$ loop measures in subdomains of $\ha\C$ is using:
$$\mu^1_{D;z}={\bf 1}_{\{\cdot\subset D\}}e^{\cc\Lambda^*(\cdot, D^c)}\cdot \mu^1_z,\quad
\mu^0_D={\bf 1}_{\{\cdot\subset D\}}e^{\cc\Lambda^*(\cdot,D^c)}\cdot \mu^0.$$

Combining the generalized restriction property of chordal SLE with the CMP of rooted SLE$_\kappa$ loop measure in $\ha\C$, we are able to prove that the rooted and unrooted SLE$_\kappa$ loop measures in the subdomains of $\C$ satisfy conformally covariance and invariant, respectively.

By definition, the SLE$_\kappa$ loop measures in subdomains of $\ha\C$ satisfy the generalized restriction property. Especially, when $\kappa=8/3$, i.e, $\cc=0$, they satisfy the strong restriction property, and so agree with Werner's measure.
When $\kappa=2$ and $D$ is a conformal annulus, if we restrict $\mu^0_D$ to the family of curves that separate the two boundary components of $D$, then we get a finite measure, which is expected to agree with Kassel-Kenyon's probability measure after normalization. For $\kappa\in(8/3,4]$, the SLE$_\kappa$ loop measures and bubble measures should agree with the Kemppainen-Werner's loop measures and Sheffield-Werner's bubble measures up to a multiplicative constant depending on $\kappa$.
Our study of SLE$_\kappa$ loop measures will provide better understanding of these known measures.
Moreover, the SLE$_\kappa$ loop measures for $\kappa\in(0,4]$ give examples of   Malliavin-Kontsevich-Suhov loop measures for all $\cc\le 1$.

Later, we extend unrooted SLE$_\kappa$ loop measures to two types of Riemann surfaces $S$ using the Brownian loop measure on $S$. A Riemann surface $S$ of the first type satisfies that, if for any two disjoints subsets $V_1,V_2$ of $S$ such that $V_1$ is compact and $V_2$ is closed, we have
\BGE \mu^{\lloop}_S({\cal L}(V_1,V_2))<\infty,\label{finite}\EDE where $\mu^{\lloop}_S$ denotes the Brownain loop measure on $S$. For a Riemann surface $S$ of the second type, the above quantity is infinite, but the normalization method in \cite{normalize} works. This means that: first, if $K$ is a nonpolar closed subset of $S$, i.e., $K$ is accessible by a Brownian motion on $S$, then $S\sem K$ is of the first type; second, for any two disjoint closed subsets $V_1,V_2$ of $S$, one of which is compact, and any $z_0\in S$, the limit
\BGE \Lambda^*_S(V_1,V_2):=\lim_{r\downarrow 0} [\mu^{\lloop}_{S\sem \lin B(z_0,r)}({\cal L}(V_1,V_2))-\log\log(1/r)]\label{normalized-Brownian-S}\EDE
converges to a finite number, which does not depend on the choice of $z_0\in S$. Here $\lin B(z_0,r)$ is a closed disc centered at $z_0$ w.r.t.\ some chart surrounding $z_0$. The limit should also not depend on the choice of the chart. The quantity $\mu^{\lloop}_{S\sem \lin B(z_0,r)}({\cal L}(V_1,V_2))$ is finite because $\lin B(z_0,r)$ is a nonpolar set. We believe that (\cite{future}) any compact Riemann surface is of the second type, and any compact Riemann surface minus a nonpolar set is of the first type.

In contrast to the SLE$_\kappa$ defined in multiply connected domains and Riemann surfaces in \cite{BF-multiply,Thesis,Law-mult}, the definition of (unrooted) SLE$_\kappa$ loop measure in a Riemann surface does not require that the surface has a boundary, and does not need a marked point to start the curve. This makes the SLE$_\kappa$ loop measure a more natural object in some sense.

At the end of the paper, we use a similar method to construct an SLE$_\kappa$ bubble measure $\mu^1_{\HH;x}$ in the upper half plane $\HH$ rooted at a boundary point $x$. We obtain a theorem for $\mu^1_{\HH;x}$, which is similar to Theorem \ref{short}, except that now the space-time homogeneity (ii) does not make sense, and the covariance exponent $2-d$ in (iv) should be replaced by $\frac 8\kappa -1$ (and $F$ maps $\HH$ onto $\HH$). Using the Brownian loop measure, we then extend the SLE$_\kappa$ bubble measures to multiply connected domains.

The paper is organized as follows. In Section \ref{Preliminaries}, we fix symbols and recall some fundamental results about SLE. In Section \ref{whole-kappa-rho}, we describe how a whole-plane SLE$_\kappa(2)$ curve is distorted by a conformal map that fixes $0$. In Section \ref{Section-C}, we construct the rooted and unrooted SLE$_\kappa$ loop measures in $\ha\C$. In Section \ref{Section-S}, we construct SLE$_\kappa$ loop measures in subdomains of $\ha\C$ and in general Riemann surfaces. In Section \ref{Section-B}, we construct SLE$_\kappa$ bubble measures. In the appendix, we extend the generalized restriction property for chordal SLE$_\kappa$ from $\kappa\in(0,4]$ to $\kappa\in(0,8)$.

\section*{Acknowledgments}
The author would like to thank Greg Lawler and Wendelin Werner for inspiring discussions,  thank Laurie Field for correcting a mistake in an earlier draft, and thank Wei Wu and Yiling Wang  for some useful comments.

The author acknowledges the support from the National Science Foundation (DMS-1056840) and from the Simons Foundation (\#396973). The author  also thanks the  Institut Mittag-Leffler and Columbia University, where part of this work was carried out during workshops held there.

\section{Preliminaries} \label{Preliminaries}
\subsection{Symbols and notation} \label{Symbols}
Throughout, we fix $\kappa\in(0,8)$. Let $d=1+\frac\kappa 8\in(1,2)$  and $\cc$ be given by (\ref{cc}). Let $\HH=\{z\in\C:\Imm z>0\}$; $\D=\{z\in\C:|z|<1\}$; $\D^*=\{z\in\C:|z|>1\}\cup\{\infty\}$; $\TT=\pa\D=\pa\D^*$.
 For $z_0\in\C$ and $r>0$, let $B(z_0;r)=\{z\in\C:|z-z_0|<r\}$. For a set $S\subset\C$ and $r>0$, let $B(S;r)=\bigcup_{z\in S} B(z;r)$. Let $e^i$ denote the map $z\mapsto e^{iz}$. We will use the functions $\sin_2=\sin(\cdot /2)$, $\cos_2=\cos(\cdot/2)$, and $\cot_2=\cot(\cdot/2)$.

We use $\mA$ and $\mA^2$ to denote the $1$-dimensional and $2$-dimensional Lebesgue measures, respectively. Given a measure $\mu$, a nonnegative measurable function $f$, and a measurable set $E$ on a measurable space $\Omega$, we use $f\cdot \mu$ to denote the measure on $\Omega$ that satisfies $(f\cdot \mu)(A)=\int_A fd\mu$ for any measurable set $A$ in $\Omega$, and use $\mu|_E$ to denote the measure ${\bf 1}_E\cdot \mu=\mu(\cdot\cap E)$. If $h:\Omega\to\Omega'$ is a measurable map, then we use $h(\mu)$ to denote the pushforward measure $\mu\circ h^{-1}$ on $\Omega'$.

The Brownian loop measure  in $\ha\C$ is a sigma-finite measure on unrooted loops in $\ha\C$, which locally look like planar Brownian motions.
We use $\mu^{\lloop}$ to denote the Brownian loop measure in $\ha\C$.  Let  ${\cal L}_D(A,B)$ (resp.\ ${\cal L}_D(A)$) denote the sets of loops in $D$ that intersect both $A$ and $B$ (resp.\ $A$). We omit the subscript $D$ when $D=\ha\C$. We need the following fact\ (\cite[Corollary 4.20]{normalize}): if $D$ is a nonpolar domain, i.e., $\pa D$ can be visited by a Brownian motion, then (\ref{finite}) holds with $S=D$ and disjoint closed subsets $V_1,V_2$ of $D$, one of which is compact.
If $D=\ha\C$, $\mu^{\lloop}({\cal L}_D(V_1,V_2))$ is not finite. Instead, we should use the normalized quantity $\Lambda^*(V_1,V_2)$ in the formula (\ref{normalized-Brownian}) as introduced in \cite{normalize}. Suppose $D_1\subset D_2$ are two nonpolar subdomains of $\ha\C$, and $K$ is a compact subset of $D_1$. Using the fact that ${\cal L}(K, D_1^c)$ is the disjoint union of ${\cal L}(K, D_2^c)$ and ${\cal L}_{D_2}(K,D_2\sem D_1)$ and the formula (\ref{normalized-Brownian}), we get the equality:
\BGE \Lambda^*(K,D_1^c)=\Lambda^*(K,D_2^c)+\mu^{\lloop}({\cal L}_{D_2}(K,D_2\sem D_1)).\label{normalized-Brownian-equality}\EDE

We will use an important notion of modern probability: kernel (cf.\ \cite{Kal}). Suppose $(U,{\cal U})$ and $(V,{\cal V})$ are two measurable spaces. A kernel from $(U,{\cal U})$ to $(V,\cal V)$ is a map $\nu:U\times {\cal V}\to[0,\infty]$ such that (i) for every $u\in U$, $\nu(u,\cdot)$ is a measure on $\cal V$, and (ii) for every $F\in\cal V$, $\nu(\cdot,F)$ is $\cal U$-measurable. The kernel is said to be finite if for every $u\in U$, $\nu(u,V)<\infty$; and is said to be $\sigma$-finite if there is a sequence $F_n\in\cal V$, $n\in\N$, with $V=\bigcup F_n$ such that for any $n\in\N$ and $u\in U$, $\nu(u,F_n)<\infty$.
Let $\mu$ be a $\sigma$-finite measure on $(U,\cal U)$.
Let $\nu$ be a $\sigma$-finite $\mu$-kernel from $(U,{\cal U})$ to $(V,\cal V)$. Then we may define a measure $\mu\otimes \nu$ on ${\cal U}\times{\cal V}$ such that
$$\mu\otimes \nu(E\times F)=\int_E \nu(u,F)d\mu(u),\quad E\in\cal U,\quad F\in \cal V.$$
Sometimes, we write $\mu\otimes \nu$ as $\mu(du)\otimes \nu(u,dv)$ when the meaning of $\mu\otimes\nu$ is clearer with the variable $u,v$ explicitly stated.

If $\nu$ is a $\sigma$-finite measure on $(V,\cal V)$, and $\mu$ is a $\sigma$-finite kernel from $(V,\cal V)$ to $(U,\cal U)$, then we use $\mu\overleftarrow{\otimes}\nu$ or $\mu(v,du)\overleftarrow{\otimes}\nu(dv)$ to denote the measure on ${\cal U}\times{\cal V}$, which is the pushforward of $\nu\otimes \mu$ under the map $(v,u)\mapsto (u,v)$.

We may describe the sampling of $(X,Y)$ according to the measure $\mu\otimes\nu$  in two steps. First, ``sample'' $X$ according to the measure $\mu$. Second, ``sample'' $Y$ according to the kernel $\nu$ and the value of $X$. After the second step, the marginal measure of $X$ is changed unless $\nu$ is $\mu$-a.s.\ a probability kernel, i.e., $\nu(u,V)=1$ for $\mu$-a.s.\ every $u\in U$. The new marginal measure of $X$ after sampling $Y$ is absolutely continuous w.r.t.\ $\mu$. If $\nu$ is finite, then the new marginal measure of $X$ is $\sigma$-finite, and its Radon-Nikodym derivative w.r.t.\ $\mu$ is $\nu(\cdot,V)$; otherwise, the new marginal measure of $X$ is not $\sigma$-finite, and the Radon-Nikodym theorem does not apply.

By a simply connected domain, we mean a domain that is conformally equivalent to $\D$.
Prime ends (cf.\ \cite{Ahl}) of simply connected domains are needed to rigorously describe the initial point and terminal point of a chordal SLE or two-sided radial SLE curve. For a simply connected domain $D$, a boundary point $z_0\in \pa D$, and a prime end $p$ of $D$, if for any sequence $(z_n)$ in $D$, $z_n\to z_0$ if and only if $z_n\to p$， then we do not distinguish $z_0$ from $p$. For example, if $D$ is a Jordan domain, then there is a one-to-one correspondence between boundary points of $D$ and prime ends of $D$. If  $\gamma$ is a simple curve that starts from a boundary point of a simply connected domain $D$, stays in $D$ otherwise, and ends at an interior point of $D$, then the tip of $\gamma$ determines a prime end of $D\sem \gamma$, while every other point of $\gamma$ does not determine a prime end of $D\sem \gamma$. Instead, each of them corresponds to two prime ends. In this paper, when we say that a curve lies in a simply connected domain $D$, it often means that the curve is contained in the conformal closure of $D$, i.e., the union of $D$ and all of its prime ends.

By $f:D\conf E$, we mean that $f$ maps a domain $D$ conformally onto a domain $E$. If, $f$ also maps interior points or prime ends $z_1,\dots,z_n$ of $D$ to interior points or prime ends $w_1,\dots,w_n$ of $E$, then we write $f:(D;z_1,\dots,z_n)\conf (E;w_1,\dots,w_n)$.

For a simply connected domain $D$ with two distinct prime ends $a$ and $b$, and $z_0\in D$, we use $\mu^\#_{D;a\to b}$ and $\nu^\#_{D;a\to z_0\to b}$ to denote the laws of a chordal SLE$_\kappa$ curve in $D$ from $a$ to $b$ and a two-sided radial SLE$_\kappa$ curve in $D$ from $a$ to $b$ through $z_0$, respectively, modulo a time change.  For $z_0\ne w_0$, we use $\nu^\#_{z_0\to w_0}$ and $\nu^\#_{z_0\rt w_0}$ to denote the laws of a whole-plane SLE$_\kappa(2)$ curve in $\ha\C$ from $z_0$ to $w_0$ and a two-sided whole-plane SLE$_\kappa$ curve in $\ha\C$ from $z_0$ to $z_0$ passing through $w_0$, respectively, modulo a time change. The superscript $\#$ is used to emphasize that the measure is a probability measure.

We use $G_{D;a\to b}$ to denote the Green's function for the chordal SLE$_\kappa$: $\mu^\#_{D;a\to b}$.  We have a close-form formula for $G_{\HH;0\to\infty}$ (cf.\ \cite{LR}):
\BGE G_{\HH;0\to\infty}(z)=\ha c |z|^{1-\frac8\kappa} (\Imm z)^{\frac\kappa8+\frac8\kappa-2},\quad z\in\HH,\label{Green-H}\EDE
where  $\ha c>0$ is a constant depending only on $\kappa$. For general $(D;a,b)$, we may recover $G_{D;a\to b}$ using (\ref{Green-H}) and the conformal covariance property:
\BGE G_{D;a\to b}(z)=|g'(z)|^{2-d} G_{E;c\to d}(g(z)),\quad \mbox{if }g:(D;a,b)\conf(E;c,d).\label{Green}\EDE

A stopping time $\tau$ for a curve is called nontrivial if it does not happen at the initial time. This is an assumption used in Theorem \ref{short} (i).
For a curve $\gamma$ and a (stopping) time $\tau$, we use $\K_\tau(\gamma)$ to denote the part of $\gamma$ from its initial time till the time $\tau$.

For two curves $\beta$ and $\gamma$ such that the terminal point of $\beta$ agrees with the initial point $\gamma$, we use $\beta\oplus\gamma$ to denote the concatenation of $\beta$ and $\gamma$ (modulo a time change). For a measure $\mu$ and a kernel $\nu$ on the space of curves, if $\mu\otimes \nu$ is supported by the pairs $(\beta,\gamma)$ such that $\beta\oplus\gamma$ is well defined, we then use $\mu\oplus\nu$ to denote the pushforward measure of $\mu\otimes \nu$ under the concatenation map $(\beta,\gamma)\mapsto \beta\oplus\gamma$.

For a simply connected domain $D$ with two distinct prime ends $a$ and $b$, let $\Gamma(D;a, b)$ denote the family of curves $\gamma$ in $\lin D$ (modulo a time change) started from $a$ such that $\gamma$ does not intersect a neighborhood of $b$ in $D$, and the unique connected component of $D\sem \gamma$ that has $b$ as its prime end, denoted by $D(\gamma;b)$, has a prime end  determined by the tip of $\gamma$, denoted by $\gamma_{\tip}$. For $\gamma\in \Gamma(D;a, b)$, the chordal SLE$_\kappa$ measure $\mu^\#_{D(\gamma;b);\gamma_{\tip}\to b}$ is well defined. Moreover, $\gamma\mapsto \mu^\#_{D(\gamma;b);\gamma_{\tip}\to b}$ is a kernel from $\Gamma(D;a,b)$ to the space of curves.

For $z\in\ha\C$, let $\Gamma(\ha\C;z)$ denote the set of curves $\gamma$ in $\ha\C$ (modulo a time change)   from $z$ to another point $\gamma_{\tip}\in\ha\C$, such that there is a unique connected component of $\ha\C\sem \gamma$ whose boundary contains $z$ and has two prime ends determined by  $z$ and $\gamma_{\tip}$, respectively. Let $\ha\C(\gamma;z)$ denote this connected component.
For $z\ne w\in\ha\C $, let $\Gamma(\ha\C;z;w)$ denote the set of $\gamma\in\Gamma(\ha\C;z)$ such that $w\in \ha\C(\gamma;z)$.
For $\gamma\in \Gamma(\ha\C;z;w)$,  the two-sided radial SLE$_\kappa$ measure $\nu^\#_{\ha\C(\gamma;z);\gamma_{\tip}\to w\to z}$ is well defined, and the map from $\gamma\in \Gamma(\ha\C;z;w)$ to this measure is a kernel.

\subsection{SLE processes and their conformal Markov properties} \label{CMP-section}
In this subsection, we briefly review several types of SLE processes that are needed in this paper, and describe their conformal Markov properties (CMP).

A chordal SLE$_\kappa$ curve is a random curve running in a simply connected domain $D$ from one prime end to another prime end. It is first defined in the upper half-plane $\HH$ from $0$ to $\infty$ using chordal Loewner equation, and then extended to other domains by conformal maps. Chordal SLE is characterized by its CMP, i.e., if $\tau$ is a stopping time for a chordal SLE$_\kappa$ curve $\gamma$ in $D$ from $a$ to $b$, then conditional on the part of $\gamma$ before $\tau$ and the event that $\tau<T_b$ (the hitting time at $b$), the rest part of $\gamma$ is a chordal SLE$_\kappa$ curve from $\gamma(\tau)$ to $b$ in the remaining domain.  From (\cite{reversibility,MS3}) we know that chordal SLE$_\kappa$ satisfies reversibility, i.e., the reversal of a chordal SLE$_\kappa$ curve in $D$ from $a$ to $b$ has the same law (modulo a time change) as a chordal SLE$_\kappa$ curve in $D$ from $b$ to $a$.

A two-sided radial SLE$_\kappa$ curve is a random curve running in a simply connected domain $D$ from one prime end $a$ to another prime end $b$ through an interior point $z_0$. It is defined by first running a radial SLE$_\kappa(2)$ curve in $D$  from $a$ to $z_0$ with force point at $b$, and then continuing it with a chordal SLE$_\kappa$ curve from $z_0$ to $b$ in the remaining domain.
Two-sided radial SLE also satisfies CMP: if $\tau$ is a stopping time for the above two-sided radial SLE$_\kappa$ curve $\gamma$, then conditional on the part of $\gamma$ before $\tau$ and the event that $\tau<T_{z_0}$ (the hitting time at $z_0$), the rest part of $\gamma$ is a two-sided radial SLE$_\kappa$ curve from $\gamma(\tau)$ to $b$ though $z_0$ in the remaining domain.
Intuitively, one may view a two-sided radial SLE$_\kappa$ curve as a chordal SLE$_\kappa$ curve   conditioned to pass through an interior point.

Using the results and arguments in \cite{reversibility,MS3}, one can show that the two-sided radial SLE$_\kappa$ curve also satisfies reversibility, i.e., the reversal of a two-sided radial SLE$_\kappa$ curve in $D$ from $a$ to $b$ through $z_0$ has the same law (modulo a time change) as a two-sided radial SLE$_\kappa$ curve in $D$ from $b$ to $a$ though $z_0$. In particular, we see that the two arms of a two-sided radial SLE$_\kappa$ curve satisfies the resampling property: conditional on any one arm, the other arm is a chordal SLE$_\kappa$ curve in the remaining domain.

A two-sided whole-plane SLE$_\kappa$ curve from $a$ to $a$ through $b$ is a random loop in the Riemann sphere $\ha\C$ that starts from $a\in\ha\C$, passes through $b\in\ha\C$, and ends at $a$. The first arm of the curve  is a whole-plane SLE$_\kappa(2 )$ curve  from $a$ to $b$. Given the first arm of the curve, the second arm of the curve  is a chordal SLE$_\kappa$ curve from $b$ to $a$ in the remaining domain.
Two-sided whole-plane SLE$_\kappa$ is related to two-sided radial SLE$_\kappa$  by the following CMP: If $\tau$ is a nontrivial stopping time for a two-sided whole-plane SLE$_\kappa$ curve $\gamma$ from $a$ to $a$ through $b$, then conditional on the part of $\gamma$ before $\tau$ and the event that $\tau<T_b$, the rest part of $\gamma$ is a two-sided radial SLE$_\kappa$ curve from $\gamma(\tau)$ to $a$ though $b$ in the remaining domain. If the event is replaced by $T_b\le \tau<T_a$, where $T_a$ is the returning time at $a$, then the rest part of $\gamma$ is a chordal SLE$_\kappa$ curve.

From the resampling property of two-sided radial SLE$_\kappa$, and the reversibility of whole-plane SLE$_\kappa(2)$ and chordal SLE$_\kappa$ (\cite{MS3,MS4,reversibility}) we know that two-sided whole-plane SLE satisfies the following two types of reversibility properties. Suppose $\gamma$ is a whole-plane SLE$_\kappa$ curve from $a$ to $a$ through $b$. Then (i) the reversal of $\gamma$ has the same law (modulo a time change) as $\gamma$; and (ii) the closed curve obtained by traveling along any arm from $b$ to $a$ and continuing with the other arm from $a$ to $b$ has the same law (modulo a time-change) as a whole-plane SLE$_\kappa$ curve from $b$ to $b$ through $a$.

The CMP of chordal SLE may be stated in terms of kernels by the following formula. Let $T_b$ be the hitting time at $b$. If $\tau$ is a stopping time, then
\BGE \K_\tau(\mu^\#_{D;a\to b}|_{\{\tau<T_b\}})(d\gamma_\tau)\oplus \mu^\#_{D(\gamma_\tau;b);(\gamma_\tau)_{\tip}\to b}(d\gamma^\tau)=\mu^\#_{D;a\to b}|_{\{\tau<T_b\}},\label{CMP-chordal}\EDE
where implicitly stated in (\ref{CMP-chordal}) is that $\K_\tau(\mu^\#_{D;a\to b}|_{\{\tau<T_b\}})$ is supported by $\Gamma(D;a,b)$.

The CMP of the two-sided whole-plane SLE may  be stated in terms of kernels by the following formula. Let $T_w$ be the hitting time at $w$. If $\tau$ is a nontrivial stopping time, then
\BGE \K_\tau(\nu^\#_{z\to w}|_{\{\tau<T_{w}\}})(d\gamma_\tau)\oplus \nu^\#_{\ha\C( \gamma_\tau;z);(\gamma_\tau)_{\tip}\to w\to z}(d\gamma^\tau)=\nu^\#_{z\rt w}|_{\{\tau<T_{w}\}}.\label{CMP-whole}\EDE
where implicitly stated in (\ref{CMP-whole}) is that
$\K_\tau(\nu^\#_{z\to w}|_{\{\tau<T_{w}\}})$ is supported by $\Gamma(\ha\C;z;w)$, and the $\nu^\#_{z\to w}$ on the left may be replaced by $\nu^\#_{z\rt w}$.

\subsection{Minkowski content measure} \label{Minkowski-Section}
Now we review the Minkowski content. Since we have fixed $d=1+\frac{\kappa}8\in(1,2)$, we will omit the word ``$d$-dimensional''.
Let $S\subset\C$ be a closed set. The   Minkowski content of $S$ is defined to be
\BGE \Cont(S)=\lim_{r\downarrow 0} r^{d-2}{ \mA^2(B(S;r))},\label{Minkowski}\EDE
provided that the limit exists. Similarly, we define the   upper (resp.\ lower) Minkowski content of $S$: $\lin{\Cont}_d(S)$ (resp.\ $\underline{\Cont}_d(S)$) using (\ref{Minkowski}) with $\limsup$ (resp.\ $\liminf$) in place of $\lim$, which always exists.

Here are some basic facts. We always have $\underline{\Cont}_d(S)\le \lin{\Cont}_d(S)$, and the equality holds iff $\Cont(S)$ exists, which equals the common value. If $S\subset T$, then $\underline{\Cont}_d(S)\le \underline{\Cont}_d(T)$ and $\overline{\Cont}_d(S)\le \overline{\Cont}_d(T)$. Moreover, if $S=\bigcup_{n=0}^\infty S_n$, then
\BGE \lin \Cont(S)\le \sum_{n=0}^\infty \lin\Cont(S_n); \label{Cont1}\EDE
\BGE \underline{\Cont}_d(S)\le \underline{\Cont}_d(S_0)+\sum_{n=1}^\infty \lin\Cont(S_n).\label{Cont2}\EDE

\begin{Definition}
	Let $S,U\subset\C$. Suppose $\M$ is a measure supported by $S\cap U$ such that
	for every compact set $K\subset U$, $\Cont(K\cap S)= \M (K)<\infty$.
	Then we say that $\M$ is the  Minkowski content measure on $S$ in $U$, or $S$ possesses  Minkowski content measure in $U$. If $U=\C$, we may omit the phrase ``in $U$''.
\end{Definition}

\begin{Remark}
If $S$ possesses  Minkowski content measure in $U$, then the measure is determined by $S$ and $U$. We will use $\M_{S;U}$  to denote this measure. In the case $U=\C$, we may also omit the subscript $U$. If in addition,  $U'\subset U$, then for any closed set $F\subset\C$,  $S':=S\cap F$ also possesses Minkowski content measure in $U'$, and $\M_{S';U}=\M_{S;U}|_{S'\cap U'}$.
\end{Remark}

\begin{Definition}
	Let $\mu$ be a measure on $\ha\C$. Let $\gamma:I\to \ha\C$ be a continuous curve, where $I$ is a real interval. We say that $\gamma$ can be parametrized by $\mu$, or $\mu$ is a parametrizable measure for $\gamma$ if there is a continuous and strictly increasing function $\theta$ defined on $I$ such that for any $a\le b\in I$, $\theta(b)-\theta(a)=\mu(\gamma([a,b]))$.
\end{Definition}

\begin{Remark}
	Suppose a  parametrizable measure $\mu$ for $\gamma$ exists. Then we may reparametrize $\gamma$ such that for any $a\le b$ in the definition domain, $\mu(\gamma([a,b]))=b-a$. In this case, we say that $\gamma$ is parametrized by $\mu$. 
	Consider the equality $\mu(\gamma(A))=\mA(A)$ for such $\gamma$. By definition, it holds for any interval $A\subset I$, where $I$ is the definition interval of $\gamma$. By subadditivity and monotone convergence of measures,  the equality also holds for any finite or countable union of subintervals of $I$; and if $A$ and $B$ are disjoint intervals, then $\mu(\gamma(A)\cap \gamma(B)=0$. Thus, $\gamma$ induces an isomorphism modulo zero between the measure spaces $(I,\mA|_I)$ and $(\gamma,\mu)$, i.e., there exist $A\subset I$ and $B\subset \gamma$ such that $\mA(I\sem A)=\mu(\gamma\sem B)=0$, and $\gamma $ is an injective measurable map from $A$ onto $B$ such that $\gamma(\mA|_A)=\mu|_B$.
	
	If in addition, $\gamma$ is a non-degenerate closed curve, and we extend $\gamma$ periodically to $\R$, then for any $a,b\in\R$ with $b-a\in[0,\mu(\gamma)]$, we have $\mu(\gamma([a,b]))=b-a$. In this case, we say that $\gamma$ is periodically parametrized by $\mu$. \label{Remark-param}
\end{Remark}

\begin{Lemma}
	A chordal SLE$_\kappa$ curve $\gamma$ in $\HH$ from $0$ to $\infty$ a.s.\ possesses  Minkowski content measure, which is supported by $\HH$ and parametrizable for $\gamma$.
	\label{Minkowski-SLE}
\end{Lemma}
\begin{proof}
	Let $\theta_t$ be the natural parametrization for $\gamma$ (\cite{LS,LZ}).
	From \cite{LR} we know that $\theta_t$ is a.s.\ a strictly increasing continuous adapted process with $\theta_0=0$ such that  for any $0\le t_1\le t_2$, $\Cont(\gamma[t_1,t_2])=\theta_{t_2}-\theta_{t_1}$. We
	claim that $\gamma(d\theta)$ is the  ($d$-dimensional) Minkowski content measure on $\gamma$. To see this,
	we need to prove that for any compact subset $K$ of $\gamma$, $\Cont(K)=\gamma(d\theta)(K)$.   Since $\lim_{t\to\infty}\gamma(t)=\infty$ (\cite{RS}), $\gamma^{-1}(K)$ is a compact subset of $[0,\infty)$. So it suffices to prove that for any compact set $J\subset [0,\infty)$. $\Cont(\gamma(J))=d\theta (J)$. We already know that this is true for  $J=[t_1,t_2]$ for any $0\le t_1\le t_2$. Suppose $J=\bigcup_{j=1}^n [a_j,b_j]$, where  $0\le a_1<b_1<a_2<b_2<\cdots<a_n<b_n$. From (\ref{Cont1}), we get $$\lin\Cont(\gamma(J))\le \sum_{j=1}^n \Cont(\gamma[a_j,b_j])=\sum_{j=1}^n \theta_{b_j}-\theta_{a_j}=(d\theta)(J).$$
	
	Let $J$ be any compact subset of $[0,\infty)$. We may find a decreasing sequence $(J_m)$ such that each $J_m$ is of the form $\bigcup_{j=1}^n [a_j,b_j]$, and $J=\bigcap_{m=1}^\infty J_m$. From this, we see that
	$$\lin\Cont(\gamma(J))\le \lim_{m\to \infty}\lin\Cont(\gamma(J_m))\le \lim_{m\to \infty} (d\theta)(J_m)=(d\theta)(J).$$
	Let $R=\max J+1$. Then we may express $(0,R)$ as the disjoint union of $J$ and  finitely or countably many open intervals $(a_n,b_n)$. Using (\ref{Cont2}) we get $$\Cont(\gamma[0,R])\le \underline{\Cont}_d(\gamma(J))+\sum  {\Cont}_d(\gamma[a_n,b_n]).$$
	Since $\Cont(\gamma[0,R])=(d\theta)([0,R])$ and $\Cont(\gamma[a_n,b_n])=(d\theta)([a_n,b_n])=(d\theta)((a_n,b_n))$, we get $$\underline{\Cont}_d(\gamma(J)) \ge (d\theta)([0,R])-\sum (d\theta)((a_n,b_n))=(d\theta)(J).$$
	Combining this with $\lin{\Cont}_d(\gamma(J))\le (d\theta)(J)$, we get $\Cont(\gamma(J))=(d\theta)(J)$, as desired.
	
	Since $d>1$, for any $n\in\N$, $\gamma(d\theta)([-n,n])=\Cont(\gamma\cap[-n,n])\le \Cont([-n,n])=0$. So we get $\gamma(d\theta)(\R)=0$. Thus, $\gamma(d\theta)$ is supported by $\lin\HH\sem\R=\HH$. Finally, since
	$$\theta(b)-\theta(a)=\Cont(\gamma([a,b]))={\cal M}_d(\gamma([a,b]))=\gamma(d\theta)([a,b]),\quad \forall 0\le a\le b,$$
	and $\theta$ is continuous and strictly increasing, $\gamma(d\theta)$ is parametrizable for $\gamma$.
\end{proof}

\begin{Lemma}
	 Suppose that $S$ possesses  Minkowski content measure $\M_{S;U}$ in an open set $U\subset\C$. Suppose $f$ is a conformal map defined on $U$ such that $f(U)\subset\C$. Then for any compact set $K\subset U$,
	\BGE \Cont(f(K\cap S))=\int_K |f'(z)|^{d}d\M_{S;U}(z).\label{f(K)}\EDE
	From this we see that the Minkowski content measure of $f(S\cap U)$ in $f(U)$ exists, which is absolutely continuous w.r.t.\  $f(\M_{S;U})$, and the Radon-Nikodym derivative is $|f'(f^{-1}(\cdot))|^{d}$. \label{conformal-content}
\end{Lemma}
\begin{proof}
	It suffices to prove (\ref{f(K)}). Let $R>0$ be such that $ B(f(K);R)\subset f(U)$.
	Fix $\delta\in(0,R/6)$ to be determined. Define the squares
	$$Q_{m,n}= [m\delta, (m+1)\delta]\times [n\delta,(n+1)\delta],\quad m,n\in\Z.$$
	Label the finite set $I=\{\iota\in\Z^2: f(S)\cap Q_\iota\ne\emptyset\}$ as $\{\iota_1,\dots,\iota_n\}$. Then $Q_{\iota_j}\subset f(U)$ for $1\le j\le n$. Let $K_j=K\cap f^{-1}(Q_{\iota_j})$, $1\le j\le n$. Then $K=\bigcup_{j=1}^n K_j$. Since $d>1$, and for any $1\le j<k\le n$, $K_j\cap K_k$ is either empty or contained in a straight line,  we have $\M_S(K_j\cap K_K)=\Cont(K_j\cap K_k)=0$. Thus, $\M_S(K)=\sum_{j=1}^n\M_S(K_j)$. Fix $\eps>0$. We may choose $\delta$ small enough such that with $L_j:= B(K_j;\delta)$, $1\le j\le n$, we have
	$$\eps+\int_K |f'(z)|^d d\M_S(z)\ge \sum_{j=1}^n \M_S(K_j)\frac{\sup_{z\in L_j}|f'(z)|^2}{\inf_{z\in L_j}|f'(z)|^{2-d}}$$
	\BGE  \ge \sum_{j=1}^n \M_S(K_j)\frac{\inf_{z\in L_j}|f'(z)|^2}{\sup_{z\in L_j}|f'(z)|^{2-d}}\ge \int_K |f'(z)|^d d\M_S(z)-\eps,\label{MSj}\EDE
	Let $r\in(0,\delta)$. By Koebe's distortion theorem, we have
	$$f(B(z;\frac{r/|f'(z)|}{(1+\frac r R)^2}))\subset B(f(z);r)\subset f(B(z;\frac{r/|f'(z)|}{(1-\frac r R)^2})).$$
	Thus, for any $1\le j\le n$,
	\BGE f(  B(K_j;\frac{r/\sup_{z\in K_j} |f'(z)|}{(1+\frac r R)^2}))\subset   B(f(K_j);r)\subset f( B(f(K_j);\frac{r/\inf_{z\in K_j} |f'(z)|}{(1-\frac r R)^2})).\label{incl}\EDE
	Using the second inclusion in (\ref{incl}), we get
	$$\mA^2( B(f(K);r))\le \sum_{j=1}^n \sup_{z\in L_j} |f'(z)|^2 \mA^2(  B(K_j;\frac{r/\inf_{z\in K_j} |f'(z)|}{(1-\frac r R)^2})).$$
	This together with $\Cont(K_j)=\M_S(K_j)$ and formula (\ref{MSj}) implies that
	\BGE \lin\Cont(f(K))\le \sum_{j=1}^n \frac{\sup_{z\in L_j} |f'(z)|^2}{\inf_{z\in L_j} |f'(z)|^{2-d}} \M_S(K_j)\le \int_K |f'(z)|^d d\M_S(z)+\eps .\label{+eps}\EDE
	Using the first inclusion in (\ref{incl}) and that $f(K_j)\subset Q_{\iota_j}$, we get
	$$\mA^2(B(f(K);r))\ge \sum_{j=1}^n  \inf_{z\in L_j} |f'(z)|^2 \mA^2(B(K_j;\frac{r/\sup_{z\in K_j} |f'(z)|}{(1+\frac r R)^2}))$$
	$$-\sum_{1\le j<k\le n} \mA^2( B(Q_{\iota_j}\cap Q_{\iota_k};r)).$$
	This together with $\Cont(K_j)=\M_S(K_j)$, formula (\ref{MSj}), and that $\Cont( Q_{\iota_j}\cap Q_{\iota_k})=0$ (as $d>1$) implies that
	\BGE \underline\Cont(f(K))\ge \sum_{j=1}^n \frac{\inf_{z\in L_j} |f'(z)|^2}{\sup_{z\in L_j} |f'(z)|^{2-d}} \M_S(K_j)\ge \int_K |f'(z)|^d d\M_S(z)-\eps .\label{-eps}\EDE
	Since (\ref{+eps}) and (\ref{-eps}) both hold for any $\eps>0$, we get (\ref{f(K)}).
\end{proof}

\begin{Remark}
From the above two lemmas, we see that,  if $\beta$ is a chordal SLE$_\kappa$ curve in a simply connected domain $D\subset\C$ from $a$ to $b$, then $\beta$ possesses  Minkowski content measure in $D$, which is parametrizable for any subarc of $\beta$ (strictly) contained in $D$. If there exists $W:(\HH;\infty)\conf (D;b)$, which extends conformally across $\R$, then  the whole $\beta$ without $b$ possesses  Minkowski content measure in $\C$, which is parametrizable for $\beta\sem\{b\}$. If $D$ is  an analytic Jordan domain, then the previous statement holds for the entire  $\beta$ including $b$. Here we use the reversibility of chordal SLE$_\kappa$ to exclude the bad behavior of $\beta$ near $b$.
\end{Remark}

\subsection{Decomposition of chordal SLE}
Field proved in \cite{Fie} that, for $\kappa\in(0,4]$,  if one integrates the laws of two-sided radial SLE$_\kappa$ curves in a suitable simply connected domain $D$ passing through different interior points (with the two ends fixed) against the Green's function for the chordal SLE$_\kappa$ curve, then one gets the law of a chordal SLE$_\kappa$ curve biased by the Minkowski content of the whole curve. This is analogous  to a simple fact of discrete random paths: if one integrates the laws of the path conditioned to pass through different fixed vertices against the probability that the path passes through each fixed vertex, one should get a measure on paths, which is absolutely continuous w.r.t.\ the law of the original discrete random path, and the Radon-Nikodym derivative is the total number of vertices on the path, which is due to the repetition of counting.

Later in \cite{decomposition}, the author extended Field's result to all $\kappa\in(0,8)$. Now
we review a proposition from \cite{decomposition}. It is expressed in terms of measures on the space of curve-point pairs.

\begin{Proposition} Let $D$ be a simply connected domain with two distinct prime ends $a$ and $b$. Then
	$$ \mu^\#_{D;a\to b}(d\gamma)\otimes {\cal M}_{\gamma;D}(dz)=\nu^\#_{D;a\to z\to b}(d\gamma)\overleftarrow{\otimes} (  G_{D;a\to b}\cdot \mA^2)(dz).$$
	\label{decomposition-Thm}
\end{Proposition}
\begin{proof}
	The statement in the special case $(D;a,b)=(\HH;0,\infty)$  follows from \cite[Theorem 4.1]{decomposition} and Lemma \ref{Minkowski-SLE}. The statement in the general case follows from that in the special case together with Lemma \ref{conformal-content} and (\ref{Green}).
\end{proof}

\begin{Remark}
	This proposition is very important for this paper. It has a richer structure than Field's result because it concerns both curve and point, which makes it more convenient for applications.
	If   $\int_D G_{D;a\to b}(z)\mA^2(dz)<\infty$ (this holds if, e.g., $D$ is a bounded analytic domain as assumed in \cite{Fie}), then the measure in the statement is finite. So the Minkowski content of the entire chordal SLE$_\kappa$ curve is a.s.\ finite. By looking at the margin of the restricted measure on the space of curves, we then recover Field's result. For a general domain $D$, we may still restrict the measure to a compact subset of $D$, and get some useful equality.
\end{Remark}

We now use this proposition to show that two-sided radial SLE$_\kappa$ curves and two-sided whole-plane SLE$_\kappa$ curves also possess Minkowski content measures.

\begin{Lemma}
	For every $\theta\in(0,2\pi)$, $\nu^\#_{\D;e^{i\theta}\to 0\to 1}$-a.s.,  $\gamma$ (including its two end points) possesses Minkowski content measure, which is supported by $\D$ and parametrizable for $\gamma$. \label{Minkowski-radial}
\end{Lemma}
\begin{proof}
	From Proposition \ref{decomposition-Thm}, we know that if we integrate the laws $\nu^\#_{\HH;0\to z\to \infty}$ for different $z$ against the measure ${\bf 1}_K G_{\HH;0\to\infty}\cdot \mA^2$ for any compact set $K\subset\HH$, then we get a measure, which is absolutely continuous w.r.t.\ $\mu^\#_{\HH;0\to \infty}$. From Lemma \ref{Minkowski-SLE} and Fubini Theorem,  we conclude that, for (Lebesgue) almost every $z\in\HH$,   $\nu^\#_{\HH;0\to z\to \infty}$-a.s.\ $\gamma$ possesses Minkowski content measure, which is parametrizable for $\gamma$.
	
	Using Lemma \ref{conformal-content} and conformal invariance of two-sided radial SLE, we then conclude that, for almost every $\theta\in(0,2\pi)$, $\nu^\#_{\D;e^{i\theta}\to 0\to 1}$-a.s.,  $\gamma$ including the initial point $e^{i\theta}$ but excluding the terminal point $1$ possesses Minkowski content measure, which is parametrizable for $\gamma$. Using the reversibility of two-sided radial SLE$_\kappa$ curves, we find that the above statement holds for the entire $\gamma$ including its both end points. We need to replace ``for almost every $\theta\in(0,2\pi)$'' with ``for every $\theta\in(0,2\pi)$''. For this purpose, we fix $\theta_0\in(0,2\pi)$, and let $\gamma$ be a two-sided radial SLE$_\kappa$ curve in $\D$ from   $e^{i\theta_0}$ to $1$ through $0$. Recall that $\gamma$ up to $T_0$, the hitting time at $0$, is a radial SLE$_\kappa(2)$ curve in $\D$ started from $e^{i\theta_0}$ with force point at $1$.
	
	For $t<T_0$, let $g_t:(\D(\gamma([0,t]);1);0,1)\conf (\D;0,1)$, and let $u(t)=|g_t'(0)|$. Then $u$ is continuous  and strictly increasing, and maps $[0,T_0)$ onto $[0,\infty)$. Suppose $\gamma$ is parametrized such that $u(t)=t$ for $0\le t\le 1$. For $0\le t\le 1$, let $X_t\in (0,2\pi)$ be such that $e^{iX_t}=g_t(\gamma(t))$. From the CMP of two-sided radial SLE$_\kappa$ curve and the definition of radial SLE$_\kappa(2)$ curve, we know that, for any fixed $t\in(0,1]$, the $g_t$-image of the part of $\gamma$ after the time $t$ is a two-sided radial SLE$_\kappa$ curve in $\D$ from $e^{iX_t}$ to $1$ through $0$; and	
	$X_t$ satisfies the SDE
	$$dX_t=\sqrt{\kappa	}dB_t+2\cot_2(X_t)dt, \quad 0\le t\le 1,$$
	for some Brownian motion $B_t$, with initial value $X_0=\theta_0$. After rescaling, $(X_t)$ may be transformed into a radial Bessel process of dimension $1+\frac 8\kappa$. From \cite[Appendix B]{tip}, we know that for any $t\in(0,1]$, the law of $X_t$ is absolutely continuous w.r.t.\ ${\bf 1}_{(0,2\pi)}\cdot \mA$.
	
	Fix a $t_0\in(0,1)$.
	Let $t_1$ be the last time after $t_0$ that $\gamma$ visits $\gamma([0,t_0])$. Then the part of $\gamma$ strictly after $t_1$ stays in a domain on which $g_{t_0}$ is conformal. Here we note that $g_{t_0}$ extends conformally across $\TT\sem \gamma([0,t_0])$.
	 From Lemma \ref{conformal-content} and the above two paragraphs, we can conclude that almost surely the part of $\gamma$ from $t_1^+$ (not including $t_1$) up to and including the terminal point $1$ possesses Minkowski content measure (in $\C$), which is parametrizable for this part of $\gamma$. Since we may choose $t_0,t_1$ arbitrarily small, the above statement holds with $0^+$ in place of $t_1^+$. Using the reversibility, we then conclude that the statement holds for the entire $\gamma$ including its both end points. Finally, to see that the Minkowski content measure is supported by $\D$, we use the fact that $\Cont(\pa\D)=0$ because $d>1$.
\end{proof}

\begin{Remark}
From Lemmas \ref{conformal-content} and \ref{Minkowski-radial}, we see that, if $\beta$ is a two-sided radial SLE$_\kappa$ curve in a simply connected domain $D\subset\C$ from $a$ to $b$ through some $z\in D$, then $\beta$ possesses  Minkowski content measure in $D$, which is parametrizable for any subarc of $\beta$ (strictly) contained in $D$. If a conformal map from $\HH$ onto $D$ that takes $\infty$ to $b$ extends analytically across $\R$, then $\beta$ without $b$ possesses  Minkowski content measure (in $\C$), which is parametrizable for $\beta\sem\{b\}$. If $D$ is bounded by an analytic Jordan domain, then the previous statement holds for the entire curve $\beta$ including both $a$ and $b$.
\end{Remark}

\begin{Lemma}
	Let $z_1\ne z_2\in\C$. Let $\gamma$ be a two-sided whole-plane SLE$_\kappa$ curve from $z_1$ to $z_1$ through $z_2$. Then $\gamma$ almost surely possesses   Minkowski content measure, which is parametrizable for (the entire) $\gamma$. In particular, $\Cont(\gamma)$ almost surely exists and lies in $(0,\infty)$. \label{Mink-whole}
\end{Lemma}
\begin{proof} Fix $r\in(0,|z_1-z_2|)$. Let $\tau_r$ be the first time that $\gamma$ reaches $\{|z-z_1|=r\}$. Let $K_{\tau_r}$ be the hull generated by the part of $\gamma$ before $\tau_r$, and let $D_{\tau_r}=\ha\C\sem K_{\tau_r}$. By CMP of two-sided whole-plane SLE, conditional on the part of $\gamma$ before $\tau_r$, the rest part of $\gamma$ is a two-sided radial SLE$_\kappa$ curve in $D_{\tau_r}$. Let $a_r$ be the last time that $\gamma$ visits $K_{\tau_r}$ before it reaches $z_2$; and let $b_r$ be the first time that $\gamma$ visits $K_{\tau_r}$ after it reaches $z_2$. By Lemmas \ref{conformal-content} and \ref{Minkowski-radial} and the fact that $\gamma$ a.s.\ does not pass through $\infty$, we see that the part of $\gamma$ strictly between $a_r$ and $b_r$ a.s.\ possesses Minkowski content measure, which is parametrizable for this part of $\gamma$. By letting $r\to 0$, we then conclude that $\gamma\sem\{z_1\}$ possesses Minkowski content measure, which is parametrizable for $\gamma\sem\{z_1\}$. By reversibility of two-sided whole-plane SLE, the above statement holds with $z_2$ in place of $z_1$. The two Minkowski content measures must agree, and so the entire $\gamma$ (including $z_1$ and $z_2$) possesses  Minkowski content measure, which is parametrizable for $\gamma$. Finally, since $\gamma$ is compact and not a single point, the total mass is finite and strictly positive.
\end{proof}

\section{Whole-plane SLE Under Conformal Distortion} \label{whole-kappa-rho}
We need a lemma, which describes how a whole-plane SLE$_\kappa(2)$ curve from $0$ to $\infty$ is modified under a conformal map $W$, which fixes $0$. To state the lemma, we need to review the definition of whole-plane SLE$_\kappa(\rho)$ processes.

We start with the definition of interior hulls in $\C$. A connected compact set $K\subset \C$ is called an interior hull if $\ha\C\sem K$ is connected, and is called non-degenerate if $\diam(K)>0$. For a non-degenerate interior hull $K$, there is a unique  $g_K$ such that $g_K:(\ha\C\sem K;\infty)\conf (\D^*;\infty)$, and $g_K'(\infty):=\lim_{z\to\infty} z/g_K(z)>0$. The value $\ccap(K):=\log(g_K'(\infty))$ is called the whole-plane capacity of $K$. By Koebe's $1/4$ theorem, we see that, for any $z_0\in K$, $\max\{|z-z_0|:z\in K\}$ lies between
$e^{\ccap(K)}$ and $4e^{\ccap(K)}$.

Next, we review the whole-plane Loewner equation.
Let $\lambda\in C((-\infty,T),\R)$ for some $T\in(-\infty,\infty]$. The whole-plane Loewner equation driven by $e^{i\lambda}$ is the ODE:
$$\pa_t g_t(z)=g_t(z)\frac{e^{i\lambda_t}+g_t(z)}{e^{i\lambda_t}-g_t(z)}, $$
with asymptotic initial value $\lim_{t\to-\infty} e^t g_t(z)=z$.
The covering whole-plane Loewner equation driven by $e^{i\lambda}$ is the ODE:
\BGE \pa_t\til g_t(z)=\cot_2(\til g_t(z)-\lambda_t).\label{covering'}\EDE
with asymptotic initial value $\lim_{t\to -\infty} \til g_t(z)-it =z$.
It is known that the solutions $g_t$ and $\til g_t$ exist uniquely for $-\infty<t<T$, and satisfy $e^i\circ \til g_t=g_t\circ e^i$ for every $t$; and there exists an increasing family of non-degenerate interior hulls $K_t$, ${-\infty<t<T}$, such that $\bigcap_t K_t=\{0\}$, and for each $t\in(-\infty,T)$, $\ccap(K_t)=t$ and $g_{K_t}=g_t$. So  $g_t:\ha\C\sem K_t\conf\D^*$. Let $\til K_t=(e^i)^{-1}(K_t)$. Then $\til g_t:\C\sem \til K_t\conf -\HH$. We call $g_t$ and $K_t$, $-\infty<t<T$, the whole-plane Loewner maps and hulls, respectively, driven by $e^{i\lambda}$; and call $\til g_t$ and $\til K_t$ the covering whole-plane Loewner maps and hulls, respectively, driven by $e^{i\lambda}$.

If for every $t\in(-\infty,T)$, $g_t^{-1}$ extends continuously to $\lin{\D^*}$, and $\gamma_t:=g_t^{-1}(e^i(\lambda_t))$, $-\infty<t<T$, is a continuous curve, which extends continuously to $[-\infty,T)$ with $\gamma_{-\infty}=0$, then we call $\gamma$ the whole-plane Loewner curve driven by $e^{i\lambda}$. If such $\gamma$ exists, then for any $t\in(-\infty,T)$, $\ha\C\sem K_t$ is the connected component of $\ha\C\sem \gamma([-\infty,t])$ that contains $\infty$.
Since $\ccap(K_t)=t$ for each $t\in(-\infty,T)$, we say that $\gamma$ is parametrized by whole-plane capacity.

Now we review the definition of whole-plane SLE$_\kappa(\rho)$ processes.
Let $\kappa>0$ and $\rho\ge \frac\kappa 2-2$. Let $(\lambda_t )_{t\in\R}$ and $(q_t )_{t\in\R}$ be two continuous real valued processes such that $X_t:=\lambda_t-q_t\in(0,2\pi)$ for all $t\in\R$. Let $(\F_t)_{t\in\R}$ be the filtration generated by $(e^{i\lambda_t};e^{iq_t})_{t\in\R}$.
We say that the $\TT\times \TT$-valued process $(e^{i\lambda_t};e^{iq_t})_{t\in\R}$ is a whole-plane SLE$_\kappa(\rho)$ driving process if for any finite $(\F_t)$-stopping time $\tau$,  $\lambda_{\tau+t}-\lambda_{\tau}$ and $q_{\tau+t}-q_{\tau}$, $t\ge 0$, satisfy the $(\F_{\tau+t})_{t\ge 0}$-adapted SDE:
\begin{align}
  d(\lambda_{\tau+t}-\lambda_{\tau})=&\sqrt\kappa dB^{\tau}_{t}+\frac\rho 2 \cot_2(X_{\tau+t})dt,\label{dlambda''} \\
  d(q_{\tau+t}-q_{\tau})=&-\cot_2(X_{\tau+t})dt, \label{dqt''}
\end{align}
on $[0,\infty)$, where $(B^{\tau}_t)_{t\ge 0}$ is an $(\F_{\tau+t})_{t\ge 0}$-Brownian motion.
Here we note that $(\lambda_t)$ and $(q_t)$ are in general not $(\F_t)$-adapted, but $(X_t)$ is $(\F_t)$-adapted.

Given a whole-plane SLE$_\kappa(\rho)$ driving process  $(e^{i\lambda_t};e^{iq_t})_{t\in\R}$, the whole-plane Loewner curve $\gamma$ driven by $e^{i\lambda}$, which exists by Girsanov's Theorem, is called a whole-plane SLE$_\kappa(\rho)$ curve from $0$ to $\infty$. Each $g_t^{-1}$ extends continuously to $\lin{\D^*}$; and the extended $g_t^{-1}$ maps $e^{i\lambda_t}$ to $\gamma(t)$, and maps $e^{iq_t}$ to $0$.

If $F$ is a M\"obius transformation, then the $F$-image of a whole-plane SLE$_\kappa(\rho)$ curve from $0$ to $\infty$ is called a whole-plane SLE$_\kappa(\rho)$ curve from $F(0)$ to $F(\infty)$. As mentioned before, each arm of a two-sided whole-plane SLE$_\kappa$ curve  is a whole-plane SLE$_\kappa(2)$ curve.

Let $\gamma(t)$, $-\infty\le t<\infty$, be a  whole-plane SLE$_\kappa(2)$ curve from $0$ to $\infty$ with driving process $(e^{i\lambda_t};e^{iq_t})$, $t\in\R$. Let $g_t$ and $K_t$ (resp.\ $\til g_t$ and $\til K_t$)  be the whole-plane Loewner maps and hulls (covering whole-plane Loewner maps and hulls), respectively, driven by $e^{i\lambda}$. Let $(\F_t)_{t\in\R}$ be the filtration generated by $(e^{i\lambda_t};e^{iq_t})$. Let $\tau$ be any $(\F_t)$-stopping time as in the definition of whole-plane SLE$_\kappa(\rho)$ process.
Then we have the $(\F_{\tau+t})_{t\ge 0}$-adapted SDE (\ref{dlambda''},\ref{dqt''}) with $\rho=2$.
To avoid many occurrences of $\tau+t$, we rewrite them as
\begin{align}
  d(\lambda_{t}-\lambda_{\tau})=&\sqrt\kappa dB^{\tau}_{t-\tau}+  \cot_2(X_{t})dt,\quad \tau\le t<\infty;\label{dlambda'} \\
  d(q_{t}-q_{\tau})=&-\cot_2(X_{t})dt, \quad \tau\le t<\infty.\label{dqt'}
\end{align}
Combining the above two equations, we get an SDE for $(X_t)$:
\BGE dX_{t}=\sqrt\kappa dB^\tau_{t-\tau}+2\cot_2(X_{t})dt,\quad \tau\le t<\infty.\label{dXtau'}\EDE

Let $U$ and $V$ be sub-domains of $\ha\C$ that contain $0$. Suppose that $W:(U;0) \conf(V;0)$. We will show that the law of  $\gamma$ stopped at certain time is absolutely continuous w.r.t.\  the law of  the $W(\gamma)$ stopped at certain time, and describe the Radon-Nikodym derivative.
We are going to use a standard argument that originated in \cite{LSW-8/3}. A similar argument involving chordal Loewner equations can be found in the proof of Proposition \ref{prop-mult}.

Let $\til U=(e^i)^{-1}(U)$ and $\til V=(e^i)^{-1}(V)$.
There exists $\til W:\til U\conf \til V $ such that $W\circ e^i=e^i\circ \til W$.
Let $\tau_U$ be the largest time such that $K_t\subset U\sem \{W^{-1}(\infty)\}$ for $-\infty<t<\tau_U$. If $\tau_U<\infty$, then either $\gamma$ exits $U\sem \{W^{-1}(\infty)\}$ at $\tau_U$, or separates some part of $U\sem \{W^{-1}(\infty)\}$ from $\infty$ at $\tau_U$.  For $-\infty<t<\tau_U$, $W(K_t)$ is an interior hull in $\C$, and we let $u(t)=\ccap(W(K_t))$. Then $u$ is continuous and strictly increasing, and maps $(-\infty,\tau_U)$ onto $(-\infty,S)$ for some $S\in(-\infty,\infty]$. Moreover, by Koebe's distortion theorem, we have
\BGE \lim_{t\to -\infty} e^{u(t)-t}=|W'(0)|.\label{u-infty'}\EDE

Let $L_s:=W(K_{u^{-1}(s)})$ and $\beta(s):=W(\gamma(u^{-1}(s)))$, $-\infty\le s<S$. Then $\beta$ is a whole-plane Loewner curve, and $L_s$ are the hulls generated by $\beta$. Let $e^{i\sigma_s}$ denote the driving function, and let $h_s$ and $\til h_s$ be the corresponding whole-plane Loewner maps and covering whole-plane Loewner maps, respectively.
For $-\infty<t<\tau_U$, define $W_t=h_{u(t)}\circ W\circ g_t^{-1}$, $\til W_t=\til h_{u(t)}\circ \til W\circ \til g_t^{-1}$, $U_t=g_t(U\sem K_t)$, $V_{u(t)}=h_{u(t)}(V\sem L_{u(t)})$, $\til U_t=(e^i)^{-1}(U_t)$, $\til V_{u(t)}=(e^i)^{-1}(V_{u(t)})$. Then $W_t:U_t\conf V_{u(t)}$, and $\til W_t:\til U_t\conf\til V_t $. From $g_t\circ e^i=e^i\circ \til g_t$, $h_s\circ e^i=e^i\circ \til h_s$, and $W\circ e^i=e^i\circ \til W$, we get $W_t\circ e^i=e^i\circ \til W_t$. Note that $U_t$ and $V_t$ are subdomains of $\D^*$ that contain neighborhoods of $\TT$ in $\D^*$, and as $z\in U_t$ tends to a point on $\TT$, $W_t(z)$ tends to $\TT$ as well. By Schwarz reflection principle, $W_t$ extends conformally across $\TT$, and maps $\TT$ onto $\TT$. Similarly, $\til W_t$ extends conformally across $\R$, and maps $\R$ onto $\R$. By the continuity of $\til g_t$ and $\til h_{u(t)}$ in $t$ and the maximal principle, we know that the extended  $\til W_t$ is continuous in $t$ (and $z$). Since $g_t(\gamma(t))=e^{i\lambda_t}$ and $h_{u(t)}(\beta(u(t)))=e^{i\sigma_{u(t)}}$, we get $e^{i\sigma_{u(t)}}=W_t(e^{i\lambda_t})$. By adding an integer multiple of $2\pi$ to $\sigma_s$, we may assume that
\BGE \sigma_{u(t)}=\til W_t(\lambda_t).
\label{W(lambda)'}\EDE
Fix $t\in(-\infty,\tau_U)$. Let $\eps\in(0,\tau_U-t)$. Then $g_t(K_{t+\eps}\sem K_t)$ is a hull in $\D^*$ with radial capacity  w.r.t.\ $\infty$ (c.f.\ \cite{Law1}) being $\eps$; and $h_{u(t)}(L_{u(t+\eps)}\sem L_{u(t)}))$ is a hull in $\D^*$ with  radial capacity w.r.t.\ $\infty$ being $u(t+\eps)-u(t)$. Since $W_t$ maps the former hull to the latter hull, and when $\eps\to 0^+$, the two hulls shrink to $e^{i\lambda_t}$ and $e^{i\sigma_{u(t)}}$, respectively, using a radial version of \cite[Lemma 2.8]{LSW1}, we obtain $u_+'(t)= |W_t'(e^{i\lambda_t})|^2=\til W_t'(\lambda_t)^2$. Using the continuity of $\til W_t$ in $t$, we get
\BGE u'(t)=\til W_t'(\lambda_t)^2.\label{u''}\EDE
Thus, $\til h_{u(t)}$ satisfies the equation
\BGE \pa_t \til h_{u(t)}(z)=\til W_t'(\lambda_t)^2 \cot_2(\til h_{u(t)}(z)-\sigma_{u(t)}).\label{pahut'}\EDE
Combining (\ref{u-infty'},\ref{u''}), we get
\BGE \exp\Big( \int_{-\infty}^t (\til W_s'(\lambda_s)^2-1)ds\Big)=|W'(0)|^{-1}e^{u(t)-t}. \label{W2-1'}\EDE

From the definition of $\til W_t$, we get the equality
\BGE \til W_t\circ \til g_t(z)=\til h_{u(t)}\circ \til W(z), \quad z\in (e^i)^{-1}(U\sem K_t).\label{circ'}\EDE
Differentiating this equality w.r.t.\ $t$ and using (\ref{covering'},\ref{pahut'}), we get
$$\pa_t\til W_t(\til g_t(z))+\til W_t'(\til g_t(z))\cot_2(\til g_t(z)-\lambda_t)=\til W_t'(\lambda_t)^2 \cot_2(\til h_{u(t)}\circ \til W(z)-\sigma_{u(t)}).$$
Combining this formula with (\ref{W(lambda)'},\ref{circ'}) and replacing $\til g_t(z)$ with $w$, we get
\BGE \pa_t\til W_t(w)=\til W_t'(\lambda_t)^2 \cot_2(\til W_t(w)-\til W_t(\lambda_t))-\til W_t'(w)\cot_2(w-\lambda_t),\quad w\in \til U_t.
\label{patWt'}\EDE
Letting $\til U_t\ni w\to \lambda_t$ in (\ref{patWt'}), we get
\BGE \pa_t\til W_t(\lambda_t)=-3\til W_t''(\lambda_t).\label{-3'}\EDE
Differentiating (\ref{patWt'}) w.r.t.\ $w$ and letting  $\til U_t\ni w\to \lambda_t$, we get
\BGE \frac{\pa_t\til W_t'(\lambda_t)}{\til W_t'(\lambda_t)}=\frac 12\Big(\frac{\til W_t''(\lambda_t)}{\til W_t'(\lambda_t)}\Big)^2-\frac 43\frac{\til W_t'''(\lambda_t)}{\til W_t'(\lambda_t)}-\frac 16 (\til W_t'(\lambda_t)^2-1).\label{1243'}\EDE

Define $p_s$ such that
\BGE p_{u(t)}=\til W_t(q_t),\quad -\infty<t<\tau_U.\label{W(q)'}\EDE
Since $g_t(0)=e^{iq_t}$, we get $e^{i p_{u(t)}}=h_{u(t)}(0)$. Since $\til W_t(z+2\pi)=\til W_t(z)+2\pi$, from $\lambda_t-q_t=X_t\in(0,2\pi)$ and (\ref{W(lambda)'},\ref{W(q)'}) we get $Y_{u(t)}:=\sigma_{u(t)}-p_{u(t)}\in(0,2\pi)$. Using
(\ref{covering'},\ref{dqt'},\ref{pahut'},\ref{circ'}) we get
\BGE d p_{u(t)}=-\til W_t'(\lambda_t)^2  \cot_2(Y_{u(t)} ) .\label{dpt'}\EDE
Differentiating (\ref{patWt'}) w.r.t.\ $w$, letting  $\til U_t\ni w\to q_t$, and using (\ref{dqt'}), we get
\BGE \frac{d\til W_t'(q_t)}{\til W_t'(q_t)}=\til W_t'(\lambda_t)^2 \cot_2'(Y_{u(t)} )dt-\cot_2'(X_t)dt.\label{paWt'qt'}\EDE

Suppose the $(\F_t)$-stopping time $\tau$ is less than $\tau_U$.
From now on till right before Lemma \ref{whole-conformal'}, the ranges of $t$ in all equations are $[\tau,\tau_U)$.
Combining (\ref{W(lambda)'},\ref{-3'},\ref{dpt'},\ref{dlambda'}), and using It\^o's formula, we see that $Y_{u(t)}=\sigma_{u(t)}-p_{u(t)}$ satisfies the SDE
\begin{align}
dY_{u(t)}=  \til W_t'(\lambda_t)\sqrt\kappa dB^\tau_{t-\tau}+\til W_t'(\lambda_t)\cot_2(X_{t})dt
+\Big(\frac\kappa 2-3\Big) \til W_t''(\lambda_t)dt
+\til W_t'(\lambda_t)^2  \cot_2(Y_{u(t)} )dt.\label{dYtau'}
\end{align}
Combining (\ref{1243'},\ref{dlambda'}) with $\rho=2$ and using It\^o's formula, we get
\begin{align}
\frac{d\til W_t'(\lambda_t)}{\til W_t'(\lambda_t)}=&\frac{\til W_t''(\lambda_t)}{\til W_t'(\lambda_t)}\sqrt\kappa dB^\tau_{t-\tau}+ \frac{\til W_t''(\lambda_t)}{\til W_t'(\lambda_t)}\cot_2(X_t)dt-\frac 16 (\til W_t'(\lambda_t)^2-1)dt \nonumber \\
&+\frac 12\Big(\frac{\til W_t''(\lambda_t)}{\til W_t'(\lambda_t)}\Big)^2dt+\Big(\frac \kappa 2-\frac 43\Big)\frac{\til W_t'''(\lambda_t)}{\til W_t'(\lambda_t)}dt . \label{dW''}
\end{align}

Let $(Sf)(z)=\frac{f'''(z)}{f'(z)}-\frac 32(\frac{f''(z)}{f'(z)})^2$ be the Schwarzian derivative of $f$. Let  $\cc$ be the central charge for SLE$_\kappa$ as defined by (\ref{cc}). From (\ref{dXtau'},\ref{paWt'qt'},\ref{dYtau'},\ref{dW''}) and It\^o's formula, we see that
\begin{align}
\frac{d \sin_2(X_t)^{-2/\kappa}}{\sin_2(X_t)^{-2/\kappa}}=&-\cot_2(X_t)\frac{dB^\tau_{t-\tau}}{\sqrt\kappa} +\frac{\kappa-6}{4\kappa}\cot_2(X_t)^2dt+\frac 14 dt;\label{sinX'} \\
\frac{d \sin_2(Y_{u(t)})^{2/\kappa}}{\sin_2(Y_{u(t)})^{2/\kappa}}=&\til W_t'(\lambda_t)\cot_2(Y_{u(t)})\frac{dB^\tau_{t-\tau}}{\sqrt\kappa}+\frac 1\kappa \til W_t'(\lambda_t)\cot_2(X_t)\cot_2(Y_{u(t)})dt-\frac14 \til W_t'(\lambda_t)^2dt \nonumber\\
&+\frac{6-\kappa}{4\kappa}\til W_t'(\lambda_t)^2\cot_2(Y_{u(t)})^2dt+\frac {\kappa-6}{2\kappa}\til W_t''(\lambda_t)\cot_2(Y_{u(t)})dt; \label{sinY'}\\
\frac{d \til W_t'(\lambda_t)^{\frac{6-\kappa}{2\kappa}}}{  \til W_t'(\lambda_t)^{\frac{6-\kappa}{2\kappa}}}=&\frac{6-\kappa}{2 } \frac{\til W_t''(\lambda_t)}{\til W_t'(\lambda_t)} \frac{dB^\tau_{t-\tau}}{\sqrt\kappa} +\frac{6-\kappa}{2\kappa} \frac{\til W_t''(\lambda_t)}{\til W_t'(\lambda_t)}\cot_2(X_t)dt\nonumber \\
& -\frac{6-\kappa}{12\kappa}(\til W_t'(\lambda_t)^2-1)dt+\frac{\cc}6 S(\til W_t)(\lambda_t)dt;\label{W'lambda'}\\
\frac{d\til W_t'(q_t)^{\frac{6-\kappa}{2\kappa}}}{\til W_t'(q_t)^{\frac{6-\kappa}{2\kappa}}}=& {\frac{6-\kappa}{4\kappa}}(- \til W_t'(\lambda_t)^2 \cot_2(Y_{u(t)} )^2 + \cot_2(X_t)^2)dt- {\frac{6-\kappa}{4\kappa}}(\til W_t'(\lambda_t)^2-1)dt. \label{W'q'}
\end{align}
Define
\BGE N_t=\til W_t'(\lambda_t)^{\frac{6-\kappa}{2\kappa}}\til W_t'(q_t)^{\frac{6-\kappa}{2\kappa}}\sin_2(Y_{u(t)})^{2/\kappa}\sin_2(X_t)^{-2/\kappa}.\label{Nt'}\EDE
Combining (\ref{sinX'},\ref{sinY'},\ref{W'lambda'},\ref{W'q'}) and using It\^o's formula, we get
\begin{align}
\frac{dN_t}{N_t}=& -\cot_2(X_t)\frac{dB^\tau_{t-\tau}}{\sqrt\kappa}+\til W_t'(\lambda_t)\cot_2(Y_{u(t)})\frac{dB^\tau_{t-\tau}}{\sqrt\kappa}
+\frac{6-\kappa}{2} \frac{\til W_t''(\lambda_t)}{\til W_t'(\lambda_t)} \frac{dB^\tau_{t-\tau}}{\sqrt\kappa}\nonumber\\
& +\frac{\kappa -24}{12\kappa} (\til W_t'(\lambda_t)^2-1)dt+\frac{\cc}6 S(\til W_t)(\lambda_t)dt,\quad \tau\le t<\tau_U.\label{dNt'}
\end{align}

We need the following proposition, which follows easily from \cite[Lemma 4.4]{whole}.

\begin{Proposition}
  There is a positive continuous function $N(r)$ defined on $(0,\infty)$ that satisfies $N(r)=O(re^{-r})$ as $r\to\infty$, such that the following is true. Let $\ha U$ and $\ha V$ be doubly connected open neighborhoods of $\TT$ in $\{z\in\C:|z|>1\}$ with the same modulus  $r$. Let $\ha W:\ha U\conf\ha V$ be such that $\ha W(\TT)=\TT$. Then
  \BGE |\log \ha W'(x)|, |S(\ha W)(x)|\le N(r)\quad  x\in\R.\label{estW-ha}\EDE \label{Prop-whole}
\end{Proposition}

Let $a\in\R$ be such that $\{|z|=4e^{a}\}$ separates $0$ from $U^c\cup \{W^{-1}(\infty)\}$. For $t<a$, since $\ccap(K_t)=t$, we have $K_t\subset\{|z|\le 4e^t\}\subset \{|z|< 4 e^{a}\}$. Thus, $K_t\subset U$, and the modulus of the doubly connected domain between $K_t$ and $\{|z|=4e^{a}\}$ is at least $a-t$. Since the conformal image of this doubly connected domain under $g_t$ is an open neighborhood of $\TT$ in $U_t\cap\C$, and $W_t$ maps this domain conformally onto an open neighborhood of $\TT$ in $V_t\cap\C$, using Proposition \ref{Prop-whole}, we get
\BGE |\log \til W_t'(x)|, |S(\til W_t)(x)| =O(|a-t|e^{t-a}),\quad t\in(-\infty,a),\quad  x\in\R.\label{estW'}\EDE
For $t\in(-\infty,\tau_U)$, define
\BGE M_t=N_t\exp\Big(-\frac{\kappa -24}{12\kappa}\int_{-\infty}^t (\til W_s'(\lambda_s)^2-1)ds- \frac{\cc}6\int_{-\infty}^t  S(\til W_s)(\lambda_s)ds \Big).\label{Mt'}\EDE
From (\ref{estW'}) we know that the improper integrals inside the exponential function converge.
 From (\ref{dNt'}) we see that $(M_t)$ satisfies the SDE
\BGE  \frac{dM_t}{M_t}= -\cot_2(X_t)\frac{dB^\tau_{t-\tau}}{\sqrt\kappa}+\til W_t'(\lambda_t)\cot_2(Y_{u(t)})\frac{dB^\tau_{t-\tau}}{\sqrt\kappa}+\frac{6-\kappa}{2} \frac{\til W_t''(\lambda_t)}{\til W_t'(\lambda_t)} \frac{dB^\tau_{t-\tau}}{\sqrt\kappa},\quad \tau\le t<\tau_U.\label{dMt'}\EDE

Since $X_t=\lambda_t-q_t\in(0,2\pi)$ and $Y_{u(t)}=\sigma_{u(t)}-p_{u(t)}=\til W_t(\lambda_t)-\til W_t(q_t)$, from (\ref{estW'}) we get $\sin_2(Y_{u(t)})/\sin_2(X_t)\to 1$ as $t\to -\infty$. From (\ref{Nt'},\ref{Mt'}) we see that $M_{-\infty}:=\lim_{t\to -\infty} M_t= 1$.

Let $\rho$ be a Jordan curve, whose interior contains $0$, and whose exterior contains $U^c\cup \{W^{-1}(\infty)\}$. Let $\tau_\rho$ be the hitting time at $\rho$. Let $r_1=\min\{|z|:z\in\rho\}$ and $r_2=\max\{|z|:z\in\rho\}$. Then $0<r_1\le r_2<\infty$, and $\log(r_1/4)\le \tau_\rho\le \log(r_2)$.  There is another Jordan curve $\rho'$, whose interior contains $\rho$, and has the same property as $\rho$. Let $m$ be the modulus of the domain bounded by $\rho$ and $\rho'$. Then for $t\le \tau_\rho$, the modulus of the domain bounded by $\pa K_t$ and $\rho'$ is at least $m$. From Proposition \ref{Prop-whole} we see that
\BGE |\log \til W_t'(x)|, |S(\til W_t)(x)|\le N(m),\quad t\in(-\infty,\tau_\rho],\quad  x\in\R.\label{estW''}\EDE
Combining (\ref{estW''}) and (\ref{estW'}) with $a=\log(r_1/4)$, and using $\tau_\rho-a\le \log(r_2)-\log(r_1/4)$, we see that $(M_t)$, is uniformly bounded on $[-\infty,\tau_\rho]$. By choosing the $\tau$ in (\ref{dMt'}) to be any deterministic time less than $a$,
 we see that $M_t$, $-\infty\le t\le \tau_\rho$, is a uniformly bounded martingale. Thus, $\EE[M_{\tau_\rho}]=M_{-\infty}=1$. Weighting the underlying probability measure by $M_{\tau_\rho}$, we get a new probability measure. Suppose $\tau<\tau_\rho$. By Girsanov Theorem and (\ref{dMt'}), we find that
$$\ha B^\tau_t:=B^\tau_t-\frac 1{\sqrt\kappa} \int_0^t \til W_s'(\lambda_s)\cot_2(Y_{u(s)})-\cot_2(X_s)+\frac{6-\kappa}{2} \frac{\til W_s''(\lambda_s)}{\til W_s'(\lambda_s)}ds,\quad 0\le t\le \tau_\rho-\tau,$$
is a Brownian motion under the new probability measure. We may rewrite (\ref{dYtau'}) as
\BGE dY_{u(\tau+t)}=  \til W_{\tau+t}'(\lambda_{\tau+t})\sqrt\kappa d\ha B^\tau_{t}+ 2\til W_{\tau+t}'(\lambda_{\tau+t})^2  \cot_2(Y_{u(\tau+t)} )dt,\quad 0\le t\le \tau_\rho-\tau.\label{dYutau}\EDE

Since $L_{u(\tau_\rho)}=W(K_{\tau_\rho})$ intersects $W(\rho)$, we have $u(\tau_\rho)=\ccap(L_{u(\tau_\rho)})\ge \log(\dist(0,W(\rho))/4)$. By choosing $\tau=u^{-1}(b)$ for some $b\in(-\infty,\log(\dist(0,W(\rho))/4)]$, and using (\ref{u''},\ref{dYutau}), we see that there is a Brownian motion $\til B^b_s$ such that $Y_s$ satisfies the SDE
$$ dY_{b+s}=\sqrt\kappa d\til B^b_s+2\cot_2(Y_{b+s})ds, \quad 0\le s\le u(\tau_\rho)-b.$$
Since $Y_s=\sigma_s-p_s$ and $p_s'=-\cot_2(Y_s)$, and $e^{i\sigma_s}$ is the driving function for $\beta=W(\gamma)$, we see that $(e^{i\sigma_{s}};e^{ip_{s}})$ satisfy (\ref{dlambda'},\ref{dqt'}) for $b\le s\le u(\tau_\rho)$. Since this holds for any $b\le \log(\dist(0,W(\rho))/4)$, we see that $(e^{i\sigma_{ s}};e^{ip_{ s}})_{-\infty<s\le u(\tau_\rho)}$ is the driving process for a whole-plane SLE$_\kappa(2)$ curve stopped at $u(\tau_\rho)$, which is the hitting time at $W(\rho)$. Since $\beta=W(\gamma)$ is the whole-plane Loewner curve driven by $e^{i\sigma}$, we get the following lemma.

\begin{Lemma}
	 Let $\rho$ be a Jordan curve, whose interior contains $0$, and whose exterior contains $U^c\cup \{W^{-1}(\infty)\}$. Let $\tau_\rho$ and $\tau_{W(\rho)}$ be the hitting time at $\rho$ and $W(\rho)$, respectively. Then
	$$\K_{\tau_{W(\rho)}}(\nu^\#_{0\to\infty})= W(M_{\tau_\rho}\cdot \K_{\tau_{\rho}}(\nu^\#_{0\to\infty})),$$
	where $(M_t)$ is defined by (\ref{Nt'},\ref{Mt'}).
	Here we note that $M_{\tau_\rho}(\gamma)$ is determined by the driving process $(e^{i\lambda_t};e^{iq_t})_{t\le\tau_\rho}$, which in turn is determined by $\K_{\tau_\rho}(\gamma)$. \label{whole-conformal'}
\end{Lemma}

As a corollary, we obtain the following lemma about the absolute continuity between the laws of  whole-plane SLE$_\kappa(2)$ curves.

\begin{Lemma}
Let $w\in\C\sem\{0\}$.	Let $\rho$ be a Jordan curve in $\C$, whose interior contains $0$, and whose exterior contains $w$. Let $\tau$ be the hitting time at $\rho$. Then
$$\K_{\tau}(\nu^\#_{0\to w})(d\gamma_{\tau})=R_w(\gamma_{\tau})\cdot \K_{\tau}(\nu^\#_{0\to\infty})(d\gamma_{\tau}),$$
where, in terms of the whole-plane SLE$_\kappa(2)$ driving process $(e^{i\lambda_t};e^{iq_t})$ and the corresponding whole-plane Loewner maps $(g_t)$, $R_w(\gamma_{\tau})$ can be expressed by
$$R_w(\gamma_{\tau})=\frac{|w|^{2(2-d)}e^{(d-2)\tau}|g_{\tau}'(w)|^{2-d}(|g_{\tau}(w)|^2-1)^{\frac\kappa8+\frac8\kappa-2}}{|g_{\tau}(w)-e^{i\lambda_{\tau}}|^{\frac8\kappa-1}|g_{\tau}(w)-e^{i q_{\tau}}|^{\frac8\kappa-1}} .$$ \label{RN-whole'}
\end{Lemma}
\begin{proof}
	Let $W(z)=\frac{z}{w-z}$. Then $W:(\ha\C;0,w)\conf(\ha\C;0,\infty)$. Thus, $W^{-1}(\nu^\#_{0\to\infty})=\nu^\#_{0\to w}$. From Lemma \ref{whole-conformal'}, we know that $$ \K_{\tau_\rho}(\nu^\#_{0\to w})=M_{\tau_\rho}\cdot \K_{\tau_\rho}(\nu^\#_{0\to\infty}),$$
	where $M_{\tau}$ is the value of $(M_t)$ defined by (\ref{Nt'},\ref{Mt'}) for the above $W$ at the time $\tau$.
	
	We have $U=V=\ha C$. So for all $t<\tau_U$, $U_t=V_t=\D^*$, and $W_t:\D^*\conf\D^*$. From (\ref{circ'}) and that $h_s(\infty)=\infty$, we know that $W_t$ maps $g_t(w)$ to $\infty$. Thus, there is $C_t\in\TT$ such that
	$W_t(z)=C_t \frac{1-\lin{g_t(w)}z}{z-g_t(w)}$. So we have
	$$\til W_t'(x)=|W_t'(e^{ix})|=\frac{|g_t(w)|^2-1}{|g_t(w)-e^{ix}|^2},\quad x\in\R;$$
	 $$\frac{\sin_2(Y_t)}{\sin_2(X_t)}=\frac{|W_t(e^{i\lambda_t})-W_t(e^{iq_t})|}{|e^{i\lambda_t}-e^{iq_t}|}=\frac{|g_t(w)|^2-1}{|g_t(w)-e^{i \lambda_t}| |g_t(w)-e^{i q_t}|}.$$
	Combining the above formulas, we get
	\BGE N_t=\Big(\frac{|g_t(w)|^2-1}{|g_t(w)-e^{i \lambda_t}| |g_t(w)-e^{i q_t}|}\Big)^{\frac8\kappa-1}.\label{Nt''}\EDE
	Since $W_t$ is a M\"obius Transformation, we have $SW_t\equiv 0$. Since $e^i\circ \til W_t=W_t\circ e^i$, a straightforward calculation gives
	$$S(\til W_t)(\lambda_t)=-(\til W_t'(\lambda_t)^2-1)/2.$$
	Thus, from (\ref{W2-1'},\ref{cc}) we have
	\BGE \exp\Big(-\frac{\kappa -24}{12\kappa}\int_{-\infty}^t (\til W_s'(\lambda_s)^2-1)ds- \frac{\cc}6\int_{-\infty}^t  S(\til W_s)(\lambda_s)ds \Big)=(|w|e^{u(t)-t})^{1-\frac{\kappa}{8}}.\label{exp-int}\EDE
	Since $e^{u(t)}=h_{u(t)}'(\infty)$, using (\ref{circ'}) and the expressions of $W$ and $W_t$, we get $e^{u(t)}=\frac{|w||g_t'(w)|}{|g_t(w)|^2-1}$. Combining (\ref{Mt'},\ref{Nt''},\ref{exp-int}), we find that $M_\tau=R_w(\gamma_\tau) $, as desired.
\end{proof}

We use the following lemma to relate the integral of $S(\til W_t)(\lambda_t)$ in (\ref{Mt'}) with the normalized Brownian loop measure $\Lambda^*$ defined by (\ref{normalized-Brownian}).

\begin{Lemma}
	For any time $\tau<\tau_U$,
	$$\Lambda^*(\beta_{u(\tau)},V^c)-\Lambda^*(\gamma_\tau, U^c)=\frac 16 \int_{-\infty}^\tau S(\til W_t)(\lambda_t)dt+\frac 1{12}\int_{-\infty}^\tau (\til W_t'(\lambda_t)^2-1)dt,$$ \label{lemma-loop}
	where $\gamma_\tau$ and $\beta_{u(\tau)}$ are the parts of $\gamma$ and $\beta$ up to $\tau$ and $u(\tau)$, respectively.
\end{Lemma}
\begin{proof}
	We use the Brownian bubble analysis of Brownian loop measure. Let $\mu^{\bub}_{ {x_0}}$ denote the Brownian bubble measure in $\D^*$ rooted at $e^{ix_0}\in\TT$ as defined in \cite{loop}. From the decomposition theorem of Brownian loop measure and (\ref{normalized-Brownian-equality}), we know that
	\BGE \Lambda^*(\gamma_\tau, U^c)=\lim_{a\to -\infty} \int_{a}^\tau \mu^{\bub}_{\lambda_t}({\cal L}(\D^*\sem U_t))dt-\log(-a);\label{loop-bubbleU}\EDE
	$$ \Lambda^*(\beta_{u(\tau)},V^c)=\lim_{a\to -\infty} \int_{u(a)}^{u(\tau)} \mu^{\bub}_{\sigma_s}({\cal L}(\D^*\sem V_s))ds-\log(-u(a))$$
	\BGE =\lim_{a\to -\infty} \int_{a}^\tau  \til W_t'(\lambda_t)^2 \mu^{\bub}_{\sigma_{u(t)}}({\cal L}(\D^*\sem V_{u(t)}))dt-\log(-a),\label{loop-bubbleV}\EDE
where we used the facts that $\Lambda^*(\gamma_a,U^c)+\log(-a)\to 0$  and
 $u(a)-a\to \log(W'(0))$ as $a\to -\infty$. The former can be  derived using the argument in \cite{normalize}.

	If $U$ is a subdomain of $\D^*$ that contains a neighborhood of $\TT$ in $\D^*$, we let $P^{U}_{x_0}$ denote the Poisson kernel in $U$ with the pole at $e^{ix_0}$, and $\til P^U_{x_0}=P^U_{x_0}\circ e^i$. Especially, $P^{\D^*}_{x_0}(z)=\frac 1{2\pi} \Ree \frac{z+e^{ix_0}}{z-e^{ix_0}}$ and $\til P^{\D^*}_{x_0}(z)=\Imm \cot_2(z-x_0)$.
	From \cite{loop} we know
	\begin{align*}
	\mu^{\bub}_{{\lambda_t}}({\cal L}({\D^*\sem U_t})) =\lim_{U_t\ni z\to e^{i\lambda_t}}\frac{1}{|z-e^{i\lambda_t}|^2}\Big(1-\frac{P^{U_t}_{\lambda_t}(z)}{P^{\D^*}_{\lambda_t}(z)}\Big)
	=\lim_{\til U_t\ni z\to {\lambda_t}}\frac{1}{|z-{\lambda_t}|^2}\Big(1-\frac{\til P^{U_t}_{\lambda_t}(z)}{\til P^{\D^*}_{\lambda_t}(z)}\Big).
	\end{align*}
	Similarly, using (\ref{W(lambda)'}) and that $\til W_t:\til U_t\conf\til V_{u(t)}$, we get
	\begin{align*}
	\mu^{\bub}_{\sigma_{u(t)}}({\cal L}({\D^*\sem V_{u(t)}}))&= \lim_{\til V_{u(t)}\ni w\to {\sigma_{u(t)}}}\frac{1}{|w-{\sigma_{u(t)}}|^2}\Big(1-\frac{\til P^{V_{u(t)}}_{\sigma_{u(t)}}(w)}{\til P^{\D^*}_{\sigma_{u(t)}}(w)}\Big)\\
	&= \lim_{\til U_{t}\ni z\to {\lambda_t}} \frac{1}{|\til W_t(z) -\til W_t(\lambda_t)|^2}\Big(1-\frac{\til P^{V_{u(t)}}_{\sigma_{u(t)}}\circ \til W_t(z)}{\til P^{\D^*}_{\sigma_{u(t)}}\circ \til W_t(z)}\Big)\\
	&= \lim_{\til U_{t}\ni z\to {\lambda_t}} \frac{\til W_t'(\lambda_t)^{-2}}{|z-\lambda_t|^2}\Big(1-\frac{\til W_t'(\lambda_t)^{-1}\til P^{U_t}_{\lambda_t}(z) }{\til P^{\D^*}_{\sigma_{u(t)}}\circ \til W_t(z)}\Big).
	\end{align*}
	Combining the above two formulas, we get $$\til W_t'(\lambda_t)^2 \mu^{\bub}_{{\sigma_{u(t)}}}({\cal L}({\D^*\sem V_{u(t)}}))-\mu^{\bub}_{{\lambda_t}}({\cal L}({\D^*\sem U_t}))=\lim_{ z\to {\lambda_t}}\frac{1}{|z-{\lambda_t}|^2}\Big(\frac{\til P^{U_t}_{\lambda_t}(z)}{\til P^{\D^*}_{\lambda_t}(z)}-\frac{\til W_t'(\lambda_t)^{-1}\til P^{U_t}_{\lambda_t}(z) }{\til P^{\D^*}_{\sigma_{u(t)}}\circ \til W_t(z)}\Big)$$
	$$=\frac 16 S(\til W_t)(\lambda_t)+\frac 1{12} (\til W_t'(\lambda_t)^2-1),$$
	where the latter equality follows from some tedious but straightforward computation involving power series expansions. This together with (\ref{loop-bubbleU},\ref{loop-bubbleV}) completes the proof of Lemma \ref{lemma-loop}
\end{proof}

\section{SLE Loop Measures in $\ha\C$}\label{Section-C}
 We first construct rooted SLE loop measures $\mu^1_z$, $z\in\ha\C$, in $\ha\C$. The superscript $1$ means that the curve has one root, and the subscript $z$ means that the root is $z$.

\begin{Theorem} [Rooted loops]
Let
$ G_{\C}(w)=   |w|^{-2(2-d)}$. We have the following.
\begin{enumerate} [label=(\roman*)]
\item  For each $z\in\C$, there is a unique $\sigma$-finite measure $\mu^1_z$, which is supported by non-degenerate loops in $\ha\C$ rooted (start and end) at $z$ which possess  Minkowski content measure (in $\C$) that is parametrizable, and satisfies
\BGE \mu^1_z(d\gamma)\otimes {\cal M}_\gamma(dw)=\nu^\#_{z\rt w}(d\gamma)\overleftarrow{\otimes} G_{\C}(w-z)\cdot \mA^2(dw).\label{decomposition-whole-loop}\EDE
Moreover, $\mu^1_z$ satisfies the reversibility, and may be expressed by
\BGE \mu^1_z=\Cont(\cdot)^{-1}\cdot \int_{\C\sem \{z\}} \nu^\#_{z\rt w} G_{\C}(w-z)\mA^2(dw).\label{defofloop}\EDE
 \item  For every $z\in\C$, $\mu^1_z$ satisfies the following CMP. Let $T_z$ be the time that the loop returns to $z$. Then for any nontrivial stopping time $\tau$, we have
\BGE \K_\tau(\mu ^1_z|_{\{\tau<T_z\}})(d\gamma_\tau)\oplus \mu^\#_{\ha\C(\gamma_\tau;z);(\gamma_\tau)_{\tip}\to z}(d\gamma^\tau)=\mu^1_z|_{\{\tau<T_z\}},\label{CMP-chordal-loop}\EDE
where implicitly stated in the formula is that $\K_\tau(\mu^1_z|_{\{\tau<T_z\}})$ is supported by $\Gamma(\ha\C;z)$.
\item Suppose the law of  a random curve $\gamma$ is $\mu^1_0$. Let $\gamma$ be parametrized by its Minkowski content measure such that $\gamma(0)=0$. Let $a\in\R$ be a fixed deterministic number. Then the law of the random curve ${\cal T}_a(\gamma)$ defined by ${\cal T}_a(\gamma)(t)=\gamma(a+t)-\gamma(a)$ is also $\mu^1_0$.
\item Let  $J(z)=-1/z$, and $ \mu^1_\infty=J(\mu^1_0)$.
Then $\mu^1_\infty$ is  supported by loops in $\ha\C$ rooted at $\infty$, which possesses Minkowski content measure (in $\C$) that is parametrizable for the loop without $\infty$, and satisfies
\BGE \mu^1_\infty(d\gamma)\otimes {\cal M}_\gamma(dw)=\nu^\#_{\infty\rt w}(d\gamma)\overleftarrow{\otimes} \mA^2(dw).\label{decomposition-loop-infty}\EDE
Moreover, for any bounded set $S\subset \C$, $\mu^1_\infty$-a.s.\ $\lin\Cont(\gamma\cap S)<\infty$.
\item For each $z\in\ha\C$, the measures $\mu^1_z$ satisfies M\"obius covariance as follows. If $F$ is a  M\"obius transformation that fixes $z$, then $F(\mu^1_z)=|F'(z)|^{2-d} \mu^1_z$. In the case $z=\infty$, this means that $F(z)=az+b$ for some $a,b\in\C$ with $a\ne 0$, and $F(\mu^1_\infty)=|a|^{d-2} \mu^1_\infty$.
\item For any $r>0$ and $z\in\C$, $\mu^1_z(\{\gamma:\diam(\gamma)>r\})$ and $\mu^1_z(\{\gamma:\lin{\Cont}(\gamma)>r\})$ are finite. Moreover,
there are constants $C_1,C_2\in (0,\infty)$ such that $\mu^1_z(\{\gamma:\diam(\gamma)>r\})=C_1r^{d-2}$ and $\mu^1_z(\{\gamma:\lin{\Cont}(\gamma)>r\})=C_2 r^{(d-2)/d}$ for any $z\in\C$ and $r>0$.
 \item  For $z\in\C$, if a measure $\mu'$ supported by non-degenerate loops rooted at $z$ satisfies (ii) and that $\mu'(\{\gamma:\diam(\gamma)>r\})<\infty$ for every $r>0$, then $\mu'=c\mu^1_z$ for some $c\in[0,\infty)$.
\end{enumerate}
\label{Thm-loop-measure}
\end{Theorem}

The following theorem is about unrooted SLE loop measure. By an unrooted loop we mean an equivalence class of continuous functions defined on $\TT$, where $\gamma_1$ and $\gamma_2$ are equivalent if there is a orientation-preserving auto-homeomorphism $\phi$ of $\TT$ such that $\gamma_2=\gamma_1\circ \phi$. We may view the two-sided whole-plane SLE$_\kappa$ measure $\nu^\#_{z\rt w}$ as a measure on unrooted loops. By reversibility of two-sided whole-plane SLE$_\kappa$, we get $\nu^\#_{z\rt w}=\nu^\#_{w\rt z}$.

\begin{Theorem} [Unrooted loops] Let
	$ G_{\C}(w)=   |w|^{-2(2-d)}$. Define the measure $\mu^0$ on unrooted loops by
\BGE \mu^0=\Cont(\cdot)^{-2} \cdot \int_{\C}\!\int_{\C} \nu^\#_{z\rt w} G_{\C}(w-z)\mA^2(dw)\mA^2(dz).\label{loop-unrooted}\EDE
Then $\mu^0$ is a $\sigma$-finite measure that satisfies reversibility and the following properties.
  \begin{enumerate}  [label=(\roman*)]
    \item We have the equalities
\BGE \mu^0(d\gamma)\otimes {\cal M}_\gamma(dz)= \mu^1_z(d\gamma) \overleftarrow{\otimes} \mA^2(dz);\label{decomposition-unrooted}\EDE
\BGE \mu^0(d\gamma)\otimes ({\cal M}_\gamma)^2(dz\otimes dw)= \nu^\#_{z\rt w}(d\gamma) \overleftarrow{\otimes} G_{\C}(w-z)\cdot(\mA^2)^2(dz\otimes dw).\label{decomposition-unrooted2}\EDE
    \item  For any  M\"obius transformation $F$, $F(\mu^0)=\mu^0$.
  \end{enumerate}
\label{Thm-loop-measure-unrooted}
\end{Theorem}

\begin{Remark}
	The CMP of rooted SLE$_\kappa$ loop measures allows us to apply the SLE-based results and arguments to study SLE loop measures. In the next section, we will combine the generalized restriction property of chordal SLE with this CMP to   define SLE loop measures in multiply connected domains and general Riemann surfaces.
	
	 Another application of the CMP is to study the  multi-point Green's function for the rooted SLE loop measure:
	$$G_{z_0}(z_1,\dots,z_n):=\lim_{r_1,\dots,r_n\downarrow 0} \prod_{j=1}^n r_j^{d-2} \mu^1_{z_0}\{\gamma:\gamma\cap B(z_j;r_j)\ne\emptyset,1\le j\le n \},$$
	where $z_0,z_1,\dots,z_n$ are distinct points in $\C$.
	Using the CMP together with the results of \cite{Green} on multi-point Green's function for chordal SLE, it is not difficult to prove the existence and get up-to-constant sharp bounds for the Green's function here.
\end{Remark}

\begin{Remark}
	For $\kappa\ge 8$, we may construct a probability measure $\mu^\#_0$ on  loops rooted at $0$ that satisfies the CMP in Theorem \ref{Thm-loop-measure} (ii). For the construction, one may consider a whole-plane SLE$_\kappa(\kappa-6)$ curve started from $0$. Since $\kappa-6\ge\frac{\kappa}{2}-2$, $0$ is never separated by the curve from $\infty$.
	At any nontrivial stopping time $\tau$, conditional on the past of the curve, the rest of the curve is a radial SLE$_\kappa(\kappa-6)$ curve with $0$ being the force point. From \cite{SW} we know that this is a chordal SLE$_\kappa$ curve in the remaining domain aiming at $0$, but stopped at reaching $\infty$. Thus, we may construct a random curve with law $\mu^\#_0$ by continuing a whole-plane SLE$_\kappa(\kappa-6)$ curve with a chordal SLE$_\kappa$ curve from $\infty$ to $0$. The measure $\mu^\#_\infty:=J(\mu^\#_0)$ ($J(z)=1/z$) is invariant under translation and scaling; and for $\mu^\#_\infty$-a.s.\ $\gamma$, $\gamma$ visits every point in $\C$, and can be parametrized by the Lebesgue measure $\mA^2$. This measure agrees with the law of the space-filling SLE$_\kappa$ curve from $\infty$ to $\infty$ constructed in \cite{MS4}.
	The space-filling SLE$_\kappa$ from $\infty$ to $\infty$ was also defined for $\kappa\in(4,8)$ in \cite{MS4}. But that curve does not locally look like an ordinary SLE$_\kappa$ curve.
\label{>8}
\end{Remark}

\begin{Remark}
	Theorem \ref{Thm-loop-measure} (without (vii)) and Theorem \ref{Thm-loop-measure-unrooted} also hold for $\kappa=0$, and the proofs are quite simple. Here we note that a two-sided whole-plane SLE$_0$ curve from $z$ to $z$ passing through $w$ is a random circle in $\ha\C$ passing through $z$ and $w$ such that the angle of the curve at $z$ or $w$ is uniform in $[0,2\pi)$. The rooted SLE$_0$ loop measure $\mu^1_0$ turns out to be supported by circles passing through $0$, which are radially symmetric, and the distance of the center of the circle from $0$ follows the law of $\frac 1{x^2}\cdot{\bf 1}_{(0,\infty)}\cdot \mA(dx)$. The measure $\mu^1_\infty$ rooted at $\infty$ is supported by straight lines, which is invariant under rotation or translation.
\end{Remark}

\begin{proof} [Proof of Theorem \ref{Thm-loop-measure}]
(i) It suffices to consider the case $z=0$ since $\mu^1_z$ can be expressed by $z+\mu^1_0$.
Let $\gamma_\tau(t)$, $-\infty\le t\le \tau$, be a whole-plane Loewner curve started from $0$ with driving function $e^{i\lambda_t}$, $-\infty<t\le \tau$. Note that $(\gamma_\tau)_{\tip}=\gamma_\tau(\tau)$. Let $g_t$ and $\til g_t$ be the corresponding Loewner maps and covering Loewner maps. Suppose $\gamma_\tau\in \Gamma(\ha\C;0;\infty)$.
Then $g_\tau:(\ha\C(\gamma_\tau;0);\infty,\gamma_\tau(\tau),0)\conf(\D^*;\infty,e^{i\lambda_\tau},e^{iq_\tau})$ for some $q_\tau\in(\lambda_\tau-2\pi,\lambda_\tau)$.
Recall that we have the chordal SLE$_\kappa$ measure $\mu^\#_{\ha\C(\gamma_\tau;0);\gamma_\tau(\tau)\to 0}$ and the two-sided radial SLE$_\kappa$ measure $\nu^\#_{\ha\C(\gamma_\tau;0);\gamma_\tau(\tau)\to w\to 0}$ for each $w\in \ha\C({\gamma_\tau};0)$. Since these measures are all determined by $\gamma_\tau$, we now write $\mu^\#_{\gamma_\tau}$ and $\nu^\#_{\gamma_\tau;w}$, respectively, for them.
We write $G_{\gamma_\tau}(w)$ for the Green's function $G_{\ha\C(\gamma_\tau;0);\gamma_\tau(\tau)\to 0}(w)$. Let $K$ be a compact subset of $\C\sem\{0\}$ such that $K\cap\gamma_\tau=\emptyset$.
From Proposition \ref{decomposition-Thm}, we have
\BGE \mu^\#_{\gamma_\tau}(d\gamma)\otimes {\cal M}_{\gamma\cap K}(dw)=\nu^\#_{\gamma_\tau;w}(d\gamma)\overleftarrow{\otimes} {\bf 1}_K G_{\gamma_\tau}\cdot \mA^2(dw).\label{decomposition-whole}\EDE

We now compute $G_{\gamma_\tau}(w)$ for $w\in\ha\C(\gamma_\tau;0)$.
Let $\phi(z)=i\frac{z-e^{i\lambda_{\tau}}}{z-e^{iq_\tau}}$. Then $\phi:(\D^*;e^{i\lambda_\tau},e^{iq_\tau})\conf(\HH;0,\infty)$.
Since $g_\tau:(\ha\C(\gamma_\tau;0);\gamma_\tau(\tau),0)\conf(\D^*;e^{i\lambda_\tau},e^{iq_\tau})$,
by (\ref{Green-H}) and (\ref{Green}), we get
\begin{align}
G_{\gamma_\tau}(w)&=|\phi'(g_\tau(w))|^{2-d}|g_\tau'(w)|^{2-d}G_{\HH}(\phi\circ g_\tau(w))\nonumber\\
&= \frac{\ha c|g_\tau'(w)|^{2-d}|e^{i\lambda_{\tau}}-e^{iq_\tau}|^{\frac 8\kappa-1}}{|g_\tau(w)-\lambda_\tau|^{\frac 8\kappa-1}|g_\tau(w)-e^{iq_\tau}|^{\frac 8\kappa-1}}\cdot \Big(\frac{|g_\tau(w)|^2-1}{2}\Big)^{\frac 8\kappa+\frac\kappa 8-2}
 \label{G-whole}
\end{align}
Let $R_w(\gamma_\tau)$ be as in Lemma \ref{RN-whole'}.
Let
\BGE Q(\gamma_\tau)=2^{\frac \kappa 8+\frac 8 \kappa -2} \ha c^{-1}|e^{i\lambda_{\tau}}-e^{iq_\tau}|^{1-\frac 8\kappa} e^{(d-2)\tau}.\label{R-gamma}\EDE
 From  the above formulas, we get
\BGE Q(\gamma_\tau)G_{\gamma_\tau}(w)=R_w({\gamma_\tau})G_{\C}(w).\label{QGRG}\EDE
From (\ref{decomposition-whole}) and (\ref{QGRG}), we get
\BGE Q(\gamma_\tau)\mu^\#_{\gamma_\tau }(d\gamma^\tau)\otimes {\cal M}_{\gamma\cap K}(dw)=R_w({\gamma_\tau}) \nu^\#_{\gamma_\tau ;w}(d\gamma^\tau)\overleftarrow{\otimes}  {\bf 1}_K G_{\C} \cdot \mA^2(dw).\label{decomposition-whole'}\EDE

Suppose that $\tau$ is a nontrivial stopping time. Recall that ${\cal K}_\tau$ is the killing map at time $\tau$. Define
$$ \Gamma_{\tau}=\{\gamma:\tau(\gamma)<T_0(\gamma),\K_\tau(\gamma) \in \Gamma(\C;0;\infty)\}.$$
We view both sides of (\ref{decomposition-whole'}) as kernels from $\gamma_\tau\in \Gamma(\ha\C;0;\infty)$ to the space of curve-point pairs.

Let $K$ be a fixed compact subset of $\C\sem\{0\}$, and $\Gamma_{\tau;K}=\Gamma_\tau\cap\{\gamma:K\subset \ha \C({\K_\tau(\gamma)};0)\}$.
Then the measure ${\cal K}_\tau({\bf 1}_{\Gamma_{\tau;K}}\cdot\nu^\#_{0\to \infty})(d\gamma_\tau)$ is supported by $\Gamma(\ha\C;0;\infty)$, on which $\mu^\#_{\gamma_\tau }$ and $\nu^\#_{\gamma_\tau ;w}$ are well defined if $w\in K$. Acting $ {\cal K}_\tau({\bf 1}_{\Gamma_{\tau;K}}\cdot\nu^\#_{0\to \infty}) (d \gamma_\tau)\otimes$ on the left of both sides of (\ref{decomposition-whole'}), we get two equal measures on the space of curve-curve-point triples $(\gamma_\tau,\gamma^\tau,w)$ such that $w\in\gamma^\tau$, and $\gamma_\tau\oplus\gamma^\tau$ can be defined. On the lefthand side, we get the measure
\begin{align*}
  & {\cal K}_{\tau;K}({\bf 1}_{\Gamma_{\tau;K}}\cdot\nu^\#_{0\to \infty}) (d\gamma_\tau)\otimes [ Q(\gamma_\tau)\mu^\#_{\gamma_\tau }(d\gamma^\tau)\otimes {\cal M}_{\gamma^\tau\cap K}(dw)]\\
  =&[Q   \cdot {\cal K}_\tau({\bf 1}_{\Gamma_{\tau;K}}\cdot\nu^\#_{0\to \infty})(d\gamma_\tau)\otimes \mu^\#_{\gamma_\tau }(d\gamma^\tau)]\otimes {\cal M}_{\gamma^\tau\cap K}(dw).
\end{align*}
On the righthand side, we get the measure
\begin{align*}
  &   {\cal K}_\tau({\bf 1}_{\Gamma_{\tau;K}}\cdot\nu^\#_{0\to \infty})(d\gamma_\tau)\otimes[R_w({\gamma_\tau}) \nu^\#_{\gamma_\tau ;w}(d\gamma^\tau)\overleftarrow{\otimes} {\bf 1}_K G_{\C} \cdot \mA^2(dw)] \\
  =& [R_w \cdot{\cal K}_\tau({\bf 1}_{\Gamma_{\tau;K}}\cdot\nu^\#_{0\to \infty})(d\gamma_\tau)\otimes  \nu^\#_{\gamma_\tau ;w}(d\gamma^\tau)]\overleftarrow{\otimes}  {\bf 1}_K G_{\C} \cdot \mA^2(dw)\\
  =& [ {\cal K}_\tau({\bf 1}_{\Gamma_{\tau;K}}\cdot\nu^\#_{0\to w})(d\gamma_\tau)\otimes  \nu^\#_{\gamma_\tau ;w}(d\gamma^\tau)]\overleftarrow{\otimes}  {\bf 1}_K G_{\C} \cdot \mA^2(dw),
\end{align*}
where in the last step we used Lemma \ref{RN-whole'}.

Applying the map $(\gamma_\tau,\gamma^\tau,w)\mapsto (\gamma_\tau\oplus \gamma^\tau,w)$ to the above two measures, and using the fact that ${\cal M}_{(\gamma_\tau\oplus \gamma^\tau)\cap K}={\cal M}_{\gamma^\tau\cap K}$ when $K\cap \gamma_\tau=\emptyset$, we get
\begin{align}
  &[Q  \cdot {\cal K}_\tau({\bf 1}_{\Gamma_{\tau;K}}\cdot\nu^\#_{0\to \infty})(d\gamma_\tau)\oplus \mu^\#_{\gamma_\tau }(d\gamma^\tau)](d\gamma)\otimes {\cal M}_{\gamma\cap K}(dw)\nonumber\\
  =& [ {\cal K}_\tau({\bf 1}_{\Gamma_{\tau;K}}\cdot\nu^\#_{0\to w})(d\gamma_\tau)\oplus  \nu^\#_{\gamma_\tau ;w}(d\gamma^\tau)]\overleftarrow{\otimes}  {\bf 1}_K G_{\C} \cdot \mA^2(dw)\nonumber\\
  =& {\bf 1}_{\Gamma_{\tau;K}}\cdot\nu^\#_{0\rt w}(d\gamma)\overleftarrow{\otimes}  {\bf 1}_K G_{\C} \cdot \mA^2(dw),
  \label{decomposition-whole-triple}
\end{align}
where in the last step we used   the CMP formula (\ref{CMP-whole}).

Define
\BGE \mu_{\tau;K}=Q \cdot {\cal K}_\tau({\bf 1}_{\Gamma_{\tau;K}}\cdot\nu^\#_{0\to \infty})(d\gamma_\tau)\oplus \mu^\#_{\gamma_\tau }(d\gamma^\tau) . \label{mu-tau-K}\EDE
Using (\ref{decomposition-whole-triple}), we get
  \BGE \mu_{\tau;K}(d\gamma)\otimes {\cal M}_{\gamma\cap K}(dw)
={\bf 1}_{\Gamma_{\tau;K}}\cdot\nu^\#_{0\rt w}(d\gamma)\overleftarrow{\otimes} {\bf 1}_K G_{\C} \cdot \mA^2(dw).\label{decomposition-whole-J}\EDE

The total mass of the righthand side of (\ref{decomposition-whole-J}) is bounded above by
$\int_K G_{\C}(z)\mA^2(dz)$, which is finite. So both sides of (\ref{decomposition-whole-J}) are finite measures. Thus, $\mu_{\tau;K}$-a.s., ${\Cont}_d({\cdot \cap K})<\infty$. By looking at the marginal measure of the first component (the curve), we find that
\BGE {\Cont}_d({\cdot \cap K})\cdot \mu_{\tau;K} = \int_K {\bf 1}_{\Gamma_{\tau;K}}\cdot \nu^\#_{0\rt w} G_{\C}(w) \mA^2(dw)=:\ha\mu_{K;\tau}.\label{defofhamuK}\EDE
Thus, $\ha\mu_{K;\tau}$ is supported by  $\Gamma_K:=\{\gamma: \lin{\Cont}_d(\gamma\cap K)>0\}$.
Define
\BGE \mu_{K;\tau} = \lin{\Cont}_d({\cdot\cap K})^{-1}\cdot\ha\mu_{K;\tau}.\label{mu-K-tau}\EDE
By (\ref{defofhamuK},\ref{mu-K-tau}), we get
\BGE \mu_{\tau;K}|_{\Gamma_K}=\mu_{K;\tau}. \label{mu-restr}\EDE
We now define
\BGE \mu_{\tau} =Q \cdot {\cal K}_\tau({\bf 1}_{\Gamma_\tau}\cdot\nu^\#_{0\to \infty})(d\gamma_\tau)\oplus \mu^\#_{\gamma_\tau }(d\gamma^\tau);  \label{mu-tau}\EDE
  \BGE \ha\mu_K = \int_K  \nu^\#_{0\rt w} G_{\C}(w) \mA^2(dw). \label{mu-K-ha}\EDE
Then $\mu_\tau$ is supported by $\Gamma_{\tau}$.
From (\ref{mu-tau-K},\ref{defofhamuK})  we get
\BGE \mu_\tau|_{\Gamma_{\tau;K}}=\mu_{\tau;K};\label{mu-tau-K-rest}\EDE
\BGE \ha\mu_K|_{\Gamma_{\tau;K}}=\ha\mu_{K;\tau}.\label{mu-K-tau-rest-ha}\EDE

For $n\in\N$, let $\tau_n$ be the first time that the curve reaches the circle $\{|z|=1/n\}$.
Then
\BGE \Gamma_{\tau_n;K} =\Gamma_{\tau_n},\quad \mbox{if }\dist(0,K)>1/n.\label{Ksubset}\EDE
Let $n>1/\dist(0,K)$. From (\ref{mu-K-ha}) we see that $\ha\mu_K$ is supported by $\Gamma_{\tau_n}$.
Define
\BGE \mu_K=\lin{\Cont}({\cdot\cap K})^{-1}\cdot\ha\mu_{K}.\label{mu-K}\EDE
Then 
for any nontrivial stopping time $\tau$,
\BGE  \mu_K |_{\Gamma_{\tau;K}}=  \mu_{K;\tau}.\label{mu-K-tau-rest}\EDE
Since $\ha\mu_K$ is supported by $\Gamma_{\tau_n}$,
from (\ref{mu-K-tau-rest-ha},\ref{Ksubset}) we see that $\ha\mu_{K;\tau_n}=\ha\mu_K $.
So we have $ \mu_K =  \mu_{K;\tau_n}$. 
Since $\mu_{\tau_n}$ is supported by $\Gamma_{\tau_n}$, from (\ref{mu-tau-K-rest},\ref{Ksubset}) we get $\mu_{\tau_n}=\mu_{\tau_n;K}$.
Combining these formulas with (\ref{mu-restr}), we get
\BGE \mu_{\tau_n}|_{\Gamma_K}=\mu_K.\label{mu-tau-Gamma-K}\EDE

Let $K_1\subset K_2$ be two compact subsets of $\C\sem\{0\}$. Let $n>1/\dist(0,K_2)$. Then (\ref{mu-tau-Gamma-K}) holds for $K=K_1$ or $K_2$. Since $\Gamma_{K_1}\subset \Gamma_{K_2}$, we get
$$\mu_{K_2}|_{\Gamma_{K_1}}=\mu_{K_1}.$$
So we may define a $\sigma$-finite measure $\mu^1_0$ supported by $\bigcup_n \Gamma_{\{1/n\le |w|\le n\}}=\bigcup_{K\subset \C\sem\{0\}} \Gamma_K $ such that
\BGE \mu^1_0|_{\Gamma_K}=\mu_K,\quad \mbox{for any compact }K\subset\C\sem\{0\}.\label{mu-rest-K}\EDE
By Lemma \ref{Mink-whole} and (\ref{mu-K-ha},\ref{mu-K}), each $\mu_K$ is supported by non-degenerate loops rooted at $0$ which possess Minkowski content measure that is parametrizable. So $\mu^1_0$ also satisfies these properties.

Let  $K\subset\C\sem\{0\}$ be compact, and  $\tau$ be a nontrivial stopping time. From (\ref{mu-restr},\ref{mu-tau-K-rest},\ref{mu-K-tau-rest},\ref{mu-rest-K}) we have
$$\mu^1_0|_{\Gamma_{\tau;K}\cap\Gamma_K}=\mu_K|_{\Gamma_{\tau;K}} =\mu_{K;\tau}=\mu_{\tau;K}|_{\Gamma_K}=\mu_\tau|_{\Gamma_{\tau;K}\cap\Gamma_K}.$$
Let $\Xi$ denote the set of closure of domains that lie in $\C\sem\{0\}$ whose boundary consists of a disjoint union of finitely many polygonal curves whose vertices have rational coordinates. Then $\Xi$ is countable. From the above displayed formula, we see that $\mu^1_0$ and $\mu_\tau$ agree on
$$\til\Gamma_\tau:=\bigcup_{K\in\Xi} (\Gamma_{\tau;K}\cap\Gamma_K)\subset\Gamma_\tau.$$

Given $\gamma_\tau$, by Lemmas \ref{Minkowski-SLE} and \ref{conformal-content}, $\mu^\#_{\gamma_\tau}=\mu^\#_{\ha\C(\gamma_\tau;0);(\gamma_\tau)_{\tip}\to 0}$ is supported by
$$\bigcup_{K\in\Xi,K\subset \ha\C(\gamma_\tau;0)} \{\gamma^\tau: \lin{\Cont}(\gamma^\tau\cap K)>0\}=\bigcup_{K\in\Xi} \{\gamma^\tau: \gamma_\tau\oplus \gamma^\tau\in\Gamma_K,K\subset \ha\C(\gamma_\tau;0)\}.$$
From (\ref{mu-tau}) we see that $\mu_\tau$ is supported by $\til\Gamma_\tau$.

Fix any $w\in\C\sem\{0\}$.
Suppose $\gamma$ has the law of $\nu^\#_{0\rt w}$. Let $T_w$ be the hitting time at $w$. On the event $\Gamma_\tau$, let $\gamma_\tau$ and $\gamma^\tau$ be the parts of $\gamma$ before $\tau$ and after $\tau$, respectively. From the CMP of two-sided whole-plane SLE$_\kappa$, conditional on $\gamma_\tau$ and $\Gamma_\tau$, if $\tau<T_w$, $\gamma^\tau$ is a two-sided radial SLE$_\kappa$ curve in  $\ha\C(\gamma_\tau;0)$; and if $T_w\le\tau<T_0$, then $\gamma^\tau$ is a chordal SLE$_\kappa$ curve in $\ha\C(\gamma_\tau;0)$. Following the argument in the last paragraph and using Lemmas \ref{Minkowski-SLE}, \ref{conformal-content} and \ref{Minkowski-radial}, we find that $\nu^\#_{0\rt w}|_{\Gamma_\tau}$ is supported by $\til\Gamma_\tau$. From (\ref{mu-K-ha},\ref{mu-K}) we know that $\mu_K|_{\Gamma_\tau}$ is supported by $\til\Gamma_\tau$ for every compact $K\subset\C\sem \{0\}$. Since $\mu^1_0$ is supported by $\bigcup_K \Gamma_K$, from (\ref{mu-rest-K}) we see that $\mu^1_0|_{\Gamma_\tau}$ is supported by $\til\Gamma_\tau$. Since $\mu^1_0|_{\Gamma_\tau}$ and $\mu_\tau$ agree on $\til\Gamma_\tau$, and are both supported by $\til\Gamma_\tau$, we get
\BGE \mu^1_0|_{\Gamma_\tau}=\mu_\tau.\label{mu-rest-tau}\EDE

Let  $K\subset\C\sem\{0\}$ be compact, and $n>1/\dist(0,K)$. Taking $\tau=\tau_n$ in (\ref{decomposition-whole-J}) and using (\ref{Ksubset}), we get
\BGE (\mu^1_0(d\gamma)\otimes {\cal M}_\gamma(dw))|_{\Gamma_{\tau_n}\times K}=(\nu^\#_{0\rt w}(d\gamma)\overleftarrow{\otimes} G_{\C}\cdot \mA^2(dw))|_{\Gamma_{\tau_n}\times K}.\label{replace}\EDE
From the CMP formula (\ref{CMP-whole}), we know that, for each $w\in K$,   $\nu^\#_{0\rt w}$ vanishes on $\{\tau_n<\infty\}\sem \Gamma_{\tau_n}$.
From (\ref{mu-K-ha},\ref{mu-K},\ref{mu-rest-K}), we see that $\mu^1_0$ also vanishes on $\{\tau_n<\infty\}\sem \Gamma_{\tau_n}$. Thus, (\ref{replace}) holds with $\Gamma_{\tau_n}$ replaced by $\{\tau_n<\infty\}$.
Since both $\mu^1_0(d\gamma)\otimes {\cal M}_\gamma(dw)$ and $\nu^\#_{0\rt w}(d\gamma)\overleftarrow{\otimes} G_{\C}\cdot \mA^2(dw)$ are supported by
$$ \bigcup_{n>m} (\{\gamma:\tau_n(\gamma)<\infty\}\times \{z:1/m\le |z|\le m\}),$$
we obtain (\ref{decomposition-whole-loop}) with $z=0$. By looking at the marginal measure in curves, we obtain (\ref{defofloop}) with $z=0$, which immediately implies the uniqueness of $\mu^1_0$. The reversibility of $\mu^1_0$ follows from (\ref{defofloop}) and the reversibility of $\nu^\#_{0\rt w}$.

(ii) It suffices to consider the case $z=0$. From (\ref{mu-tau},\ref{mu-rest-tau}) we see that $\K_\tau({\bf 1}_{\Gamma_\tau}\mu^1_0)=Q\cdot\K_\tau({\bf 1}_{\Gamma_\tau}\nu^\#_{0\to \infty})$, and
$$ \mu^1_0|_{\Gamma_\tau}=\K_\tau({\bf 1}_{\Gamma_\tau}\mu)(d\gamma_\tau)\oplus \mu^\#_{\ha\C(\gamma_\tau;0);(\gamma_\tau)_{\tip}\to 0}(d\gamma^\tau).$$
This formula is different from (\ref{CMP-chordal-loop}) because  $\Gamma_\tau$ is a subset of $\{\tau<T_0\}$. However, if $\tau=\tau_n$, then the measures on both sides vanish on $\{\tau_n<T_0\}\sem \Gamma_{\tau_n}$. So we can conclude that (\ref{CMP-chordal-loop}) holds for $\tau=\tau_n$. Now we consider a general nontrivial stopping time $\tau$. We have $\tau>\inf_n \tau_n$. Fix any $n\in\N$. Since (\ref{CMP-chordal-loop}) holds for $\tau_n$, we get
$$ \mu^1_0|_{\Gamma_{\tau_n}\cap\{\tau_n<\tau\}}=\K_{\tau_n}({\bf 1}_{\Gamma_{\tau_n}\cap\{\tau_n<\tau\}}\mu)(d\gamma_{\tau_n})\oplus \mu^\#_{\ha\C(\gamma_{\tau_n};0);(\gamma_{\tau_n})_{\tip}\to 0}(d\gamma^{\tau_n}).$$
Applying   the CMP formula (\ref{CMP-chordal}) to the chordal SLE$_\kappa$ measure $\mu^\#_{\ha\C(\gamma_{\tau_n};0);(\gamma_{\tau_n})_{\tip}\to 0}$ and the stopping time $\tau-\tau_n $ on the event $\{\tau_n<\tau\}$, with $T_0^{\tau_n}:=T_0-\tau_n$, we get
\begin{align*}
  & \mu^1_0|_{\{\tau_n<\tau<T_0\}}=(\mu|_{\Gamma_{\tau_n}\cap\{\tau_n<\tau\}})|_{\{\tau-\tau_n<T_0^{\tau_n}\}}\\
  =& \K_{\tau_n}({\bf 1}_{\Gamma_{\tau_n}\cap\{\tau_n<\tau\}}\mu^1_0)(d\gamma_{\tau_n})\oplus {\bf 1}_{\tau-\tau_n<T_0^{\tau_n}} \mu^\#_{\ha\C(\gamma_{\tau_n};0);(\gamma_{\tau_n})_{\tip}\to 0}(d\gamma^{\tau_n})\\
  =& \K_{\tau_n}({\bf 1}_{\Gamma_{\tau_n}\cap\{\tau_n<\tau\}}\mu^1_0)(d\gamma_{\tau_n})\oplus \K_{\tau-\tau_n}({\bf 1}_{\{\tau-\tau_n<T_0^{\tau_n}\}} \mu^\#_{\ha\C(\gamma_{\tau_n};0);(\gamma_{\tau_n})_{\tip}\to 0})(d\gamma^{\tau_n}_\tau)
  \\ &\oplus  \mu^\#_{\ha\C(\gamma_{\tau_n}\oplus \gamma^{\tau_n}_\tau;0);(\gamma_{\tau_n}\oplus \gamma^{\tau_n}_\tau)_{\tip}\to 0}(d\gamma^{\tau}).
\end{align*}
 Thus, we get
 $$\mu^1_0|_{\{\tau_n<\tau<T_0\}}=\K_{\tau}({\bf 1}_{\{\tau_n<\tau<T_0\}}\mu^1_0)(d\gamma_{\tau})\oplus   \mu^\#_{\ha\C(\gamma_{\tau };0);(\gamma_{\tau} )_{\tip}\to 0}(d\gamma^{\tau}).$$
Since $\{\tau<T_0\}=\bigcup_n \{\tau_n<\tau<T_0\}$, from the above formula we get (\ref{CMP-chordal-loop}) with $z=0$.

(iii) Fix $a\in\R$. Since ${\cal T}_a(\gamma)$ has the same Minkowski content as $\gamma$, it suffices to prove that the statement holds with $\mu^1_0$ replaced by $\ha\mu^1_0:=\lin\Cont\cdot \mu^1_0=\int_\C G_{\C}(w)\nu^\#_{0\rt w}\mA^2(dw)$. Now suppose $\gamma$ has the law of $\ha\mu^1_0$, and is parametrized by its Minkowski content measure with $\gamma(0)=0$.

Let $\theta$ be a random variable uniformly distributed on $(0,1)$ and independent of $\gamma$. Let  $\beta={\cal T}_{\theta\Cont(\gamma)}(\gamma)$. Then $\beta$ is also parametrized by its Minkowski content measure periodically with $\beta(0)=0$, and $\Cont(\beta)=\Cont(\gamma)$.  Since $\gamma$ is parametrized by its Minkowski content measure, by (i), the law of $(\gamma,\gamma(\theta\Cont(\gamma)))$ is $$
\ha\mu^1_0(d\gamma)\otimes {\cal M}_\gamma(dw)/\Cont(\gamma)=
\nu^\#_{0\rt w}(d\gamma)\overleftarrow{\otimes} G_{\C}(w)\mA^2(dw).$$
For every $w\in\C\sem\{0\}$, by the reversibility of two-sided whole-plane SLE, if $\til\gamma$ has the law of $\nu^\#_{0\rt w}$ and is parametrized by its Minkowski content measure such that $\til\gamma(0)=0$, then there a.s.\ exists a unique $s\in(0,\Cont(\til\gamma))$ such that $\til\gamma(s)=w$, and ${\cal T}_s(\til\gamma)$ has the law of $\nu^\#_{0\rt -w}$ with ${\cal T}_s(\til\gamma)(-s)=-w$. Since $G_{\C}(-w)=G_{\C}(w)$, we see that $(\beta,\beta(-\theta\Cont(\beta)))$ has the same law as $(\gamma,\gamma(\theta\Cont(\gamma)))$. This means that $\beta$ has the same law as $\gamma$, and is independent of $\theta$.  By periodicity, we have ${\cal T}_a(\gamma)={\cal T}_{a-\theta\Cont(\beta)} (\beta)={\cal T}_{\theta'\Cont(\beta)}(\beta)$,
where $\theta'\in[0,1)$ is such that $a/\Cont(\beta)-\theta -\theta'\in\Z$. Since $\theta$ is uniformly distributed on $(0,1)$ and independent of $\beta$, so is $\theta'$. From the argument above, ${\cal T}_a(\gamma)={\cal T}_{\theta'\Cont(\beta)}(\beta)$ has the same law as $\beta$, which in turn has the same law as $\gamma$. This finishes the proof.

(iv)  Applying the map $J\otimes J$ to both sides of (\ref{decomposition-whole-loop}) and using  Lemma \ref{conformal-content}, we get (\ref{decomposition-loop-infty}) and conclude that $\mu^1_\infty$ is supported by loops rooted at $\infty$, which possesses Minkowski content measure (in $\C$) that is parametrizable for the loop without $\infty$. Let $K=\lin S$. Then $K$ is a compact set. Computing the total mass of the measures on both sides of (\ref{decomposition-loop-infty}) restricted to $w\in K$, we get $\int \Cont(\gamma\cap K) \mu^1_\infty(d\gamma)=\mA^2(K)<\infty$. So we have $\mu^1_\infty$-a.s.\ $\lin\Cont(\gamma\cap S)\le \Cont(\gamma\cap K)<\infty$.

(v) Let $F(z)=az+b$ be a polynomial of degree $1$.  Applying $F\otimes F$ to both sides of (\ref{decomposition-loop-infty}), and using  Lemma \ref{conformal-content}, we get
$$F(\mu^1_\infty)(d\gamma)\otimes a^{-d} {\cal M}_\gamma(dw)=\nu^\#_{\infty\rt w}(d\gamma)\overleftarrow{\otimes}a^{-2} \mA^2(dw)=a^{-2}  \mu^1_\infty(d\gamma)\otimes {\cal M}_\gamma(dw).$$
Let $K$ be a compact subset of $\C$ and $\Gamma_K=\{\gamma:\lin{\Cont}(\gamma\cap K)>0\}$. Restricting both sides of the above formula to $w\in K$, and looking at the marginal measures of $\gamma$, we get $F(\mu^1_\infty)|_{\Gamma_K}=a^{d-2}\mu^1_\infty |_{\Gamma_K}$. Since $\mu^1_\infty$-a.s.\ $\Cont(\gamma)>0$, we see that $\mu^1_\infty$ is supported by $\bigcup_K \Gamma_K$, and so does $F(\mu^1_\infty)$. Thus, $F(\mu^1_\infty) =a^{d-2}\mu^1_\infty  $, i.e., (v) holds for $z=\infty$. Applying the inverse map $J$ and translations $w\to w+z$, we see that (v) holds for any $z\in\C$.

(vi)  By the translation invariance, the scaling property (v) and Lemma \ref{conformal-content}, it suffices to prove the first sentence of (vi) for $z=0$. We first show $\mu^1_0(\{\gamma:\diam(\gamma)>r\})<\infty$ for any $r>0$. For a compact set $S\subset\C$, we use $K_S$ to denote the interior hull generated by $S$, i.e., $\ha\C\sem K_S$ is the connected component of $\ha\C\sem S$ that contain $\infty$. Since $e^{\ccap(K_\gamma)}\asymp \diam(K_\gamma)=\diam(\gamma)$, from the scaling property, it suffices to show that $\mu^1_0(\{\gamma:\ccap(K_\gamma)>0\})<\infty$. We use $\gamma_t$ to denote the part of $\gamma$ up to $t$. Let $\tau_0$ be the first time that the curve returns to $0$ or disconnects $0$ from $\infty$. We have $\mu^1_0$-a.s.\ $K_\gamma=K_{\gamma_{\tau_0}}$ since from the CMP of $\mu^1_0$, the part of $\gamma$ after $\tau_0$ grows inside $K_{\gamma_{\tau_0}}$.
Let $\sigma_0$ denote the first $t$ such that $\ccap(K_{\gamma_t})=0$.
Then $\ccap(K_\gamma)>0$ is equivalent to $\sigma_0<\tau_0$. Applying (\ref{mu-tau},\ref{mu-rest-tau}) with $\tau=\tau_0\wedge \sigma_0$ and using that $\mu^1_0$-a.s.\ $\Gamma_\tau=\{\tau<\tau_0\}=\{\sigma_0<\tau_0\}$, we get
$\K_{\sigma_0}(\mu^1_0|_{\{\sigma_0<\tau_0\}})=Q(\gamma_{\sigma_0})\cdot \K_{\sigma_0}(\nu^\#_{0\to \infty})$.
Thus, $\mu^1_0(\{\sigma_0<\tau_0\})=\EE_{\nu^\#_{0\to \infty}}[Q(\gamma_{\sigma_0})]$. It remains to show that the expectation is finite. Suppose $\gamma$ follows the law of $\nu^\#_{0\to \infty}$, i.e., is a whole-plane SLE$_\kappa(2)$ curve from $0$ to $\infty$. Let $(e^{i\lambda_t};e^{iq_t})_{t\in\R}$ be the driving process for $\gamma$. Then $(X_t:=\lambda_t-q_t)_{t\in\R}$ is a stationary diffusion process that satisfies
(\ref{dXtau'}).
 By \cite[Equations (56), (62)]{Law-real}, the law of $X_0$ is absolutely continuous w.r.t.\ $\mA|_{(0,\pi)}$, and the density is proportional to $\sin_2(x)^{8/\kappa}$. By (\ref{R-gamma}) we get
$$\EE_{\nu^\#_{0\rt \infty}}[Q(\gamma_{\sigma_0})]=\frac{2^{d-2}}{\ha c}\cdot \frac{\int_0^{2\pi} \sin_2(x) dx}{\int_0^{2\pi} \sin_2(x)^{8/\kappa }dx}<\infty.$$

Next, we show that $\mu^1_0(\{\gamma:\lin{\Cont}(\gamma)>r\})<\infty$ for any $r>0$. From (\ref{decomposition-whole-loop}, we know that
$$\int \Cont(\gamma\cap \lin\D)\mu^1_0(d\gamma)=\int_{\lin\D} G_{\C}(w)\mA^2(dw)=\int_{\lin\D} |w|^{-2(2-d)}\mA^2(dw)<\infty.$$
Thus, $\mu^1_0(\{\gamma:\lin{\Cont}(\gamma\cap\lin\D)>r\})<\infty$ for any $r>0$. Since for curves started from $0$,
$$\{\gamma:\lin{\Cont}(\gamma)>r\}\subset \{\gamma:\lin{\Cont}(\gamma\cap\lin\D)>r\}\cup \{\gamma:\diam(\gamma)> 1\},$$
and $\mu^1_0(\{\gamma:\diam(\gamma)>1\})<\infty$, we get $\mu^1_0(\{\gamma:\lin{\Cont}(\gamma)>r\})<\infty$ for any $r>0$.

(vii) We may assume that $z=0$. Suppose $\mu'$ satisfies the assumption for $z=0$. Fix $r>s>0$, a compact set $K\subset \C$ with $\dist(0,K)>r$. Let $\tau_s$ and $\tau_r$ be the first time that the curve reaches $\{|z|=s\}$ and $\{|z|=r\}$, respectively. We use the notation in the proof of (i).
From the assumption, we have $\mu'(\Gamma_{\tau_s})<\infty$.
 Suppose $\gamma$ is parametrized by whole-plane capacity up to $\tau_r$.
Let $\ha \mu'_{K} =  \Cont(\cdot\cap K)\cdot \mu'$.
Using (\ref{CMP-chordal-loop}) and Proposition \ref{decomposition-Thm}  we get
$$\ha\mu_{K}'=\int_K \K_{\tau_s}(\mu'|_{\Gamma_{{\tau_s}}})(d\gamma_{\tau_s})\oplus G_{\gamma_{\tau_s}}(w) \cdot \nu^\#_{\gamma_{\tau_s};w}(d\gamma^{\tau_s}_{\tau_r}) \mA^2(dw).$$
Thus, the total mass of $\ha \mu'_{ K}$ equals $\int \int_K  G_{\gamma_{\tau_s}}(w) \mA^2(dw) \K_{\tau_s}(\mu'|_{\Gamma_{{\tau_s}}})(d\gamma_{\tau_s})$. From (\ref{G-whole}) we see that $G_{\gamma_{\tau_s}}(w)$ is uniformly bounded in both $\gamma_{\tau_s}$ and $w\in K$. Thus, from the finiteness of $\mu'|_{\Gamma_{\tau_s}}$ we can conclude that $\ha\mu'_{K}$ is a finite measure. Since the first arm of a two-sided radial SLE$_\kappa$ curve is a radial SLE$_\kappa(2)$ curve, using a martingale in \cite{SW}, we get \BGE\K_{{\tau_r}}(\nu^\#_{\gamma_{\tau_s};w}|_{\Gamma_{{\tau_r}}})(d\gamma^{\tau_s}_{\tau_r})=\frac{R_w({\gamma_{\tau_s}\oplus\gamma^{\tau_s}_{\tau_r}})}{R_w({\gamma_{\tau_s} })}\cdot \K_{{\tau_r}}(\nu^\#_{\gamma_{\tau_s};\infty})(d\gamma^{\tau_s}_{\tau_r}), \quad w\in K.\label{CMP-radial}\EDE
A simple way to see that this formula is correct without complicated computation is to apply  Lemma \ref{RN-whole'} to the times ${\tau_s}$ and $\tau_r$ and use the CMP for whole-plane SLE$_\kappa(2)$ measures $\nu^\#_{0\to w}$ and $\nu^\#_{0\to \infty}$. In fact, by doing that, we see that (\ref{CMP-radial}) at least holds for $\K_{\tau_s}(\nu^\#_{0\to\infty})$-a.s.\ every $\gamma_{\tau_s}$.
Using (\ref{QGRG}) and the above two displayed formulas, we get
\BGE \K_{\tau_r}(\ha\mu'_{K})=  \K_{\tau_s}(\mu'|_{\Gamma_{{\tau_s}}})(d\gamma_{\tau_s})\oplus  \int_K  G_{\gamma_{\tau_s}\oplus\gamma_{\tau_r}^{\tau_s}}(w)\mA^2(dw) \frac{Q(\gamma_{\tau_s}\oplus\gamma_{\tau_r}^{\tau_s})}{Q(\gamma_{\tau_s})} \cdot \K_{{\tau_r}}(\nu^\#_{\gamma_{\tau_s};\infty})(d\gamma^{\tau_s}_{\tau_r}).\label{hamutauKoplus}\EDE
Define a new measure $\nu'_{r;K}$ by
$$\nu'_{r;K} (d\gamma_{\tau_r})= \Big(\int_K  Q(\gamma_{\tau_r})G_{\gamma_{\tau_r}}(w)\mA^2(dw)  )\Big)^{-1}\cdot  \K_{\tau_r}(\ha\mu'_{K})(d\gamma_{\tau_r}).$$
Since $\ha\mu'_{K}$ is a finite measure, from  (\ref{G-whole},\ref{R-gamma}) we see that $\nu'_{r;K}$ is also finite.
From (\ref{hamutauKoplus}) we see that
$$ \K_{\tau_s}(\nu'_{r;K})(d\gamma_{\tau_s})=\frac 1{Q(\gamma_{\tau_s})} \cdot \K_{\tau_s}(\mu'|_{\Gamma_{{\tau_s}}})(d\gamma_{\tau_s});$$
$$\nu'_{r;K} =\K_{\tau_s}(\nu'_{r;K})(d\gamma_{\tau_s})\oplus \K_{{\tau_r}}(\nu^\#_{\gamma_{\tau_s};\infty})(d\gamma^{\tau_s}_{\tau_r}).$$
We observe that $\nu'_{r;K}$ satisfies the CMP for $\nu^\#_{0\rt \infty}$ up to $\tau_r$. Since $\nu'_{r;K}$ is supported by non-degenerate curves started from $0$, and is finite, we conclude that there is $c_{r;K}\in[0,\infty)$ such that $\nu'_{r;K}=c_{r;K}\K_{\tau_r}( \nu^\#_{0\rt \infty})=c_{r;K}\K_{\tau_r}( \nu^\#_{0\to \infty})$. By the definitions of $\nu'_{r;K}$ and $\ha\mu'_K$, we get
$$\K_{\tau_r}(\Cont(\cdot\cap K)\cdot \mu')=c_{r;K}\int_K  Q(\gamma_{\tau_r})G_{\gamma_{\tau_r}}(w)\mA^2(dw) \cdot \K_{\tau_r}( \nu^\#_{0\to \infty}).$$
Using (\ref{QGRG},\ref{mu-K-ha},\ref{mu-K},\ref{mu-rest-K}) and Lemma \ref{RN-whole'}, we get
$$\K_{\tau_r}(\Cont(\cdot\cap K)\cdot \mu')=c_{r;K} \int_K \K_{\tau_r}( \nu^\#_{0\to w}) G_{\C}(w)\mA^2(dw)$$
$$=c_{r;K} \K_{\tau_r}(\ha \mu_K)=c_{r;K}\K_{\tau_r}(\Cont(\cdot\cap K)\cdot \mu_K)=c_{r;K}\K_{\tau_r}(\Cont(\cdot\cap K)\cdot \mu^1_0).$$
Since the total mass of the measures on both sides do not depend on $r$, we see that $c_{r;K}$ depends only on $K$, and so write it as $c_K$.
Since both $\mu'$ and $\mu^1_0$ satisfy (\ref{CMP-chordal-loop}), from Proposition \ref{decomposition-Thm} we see that the expectation of $\Cont(\gamma\cap K)$ conditional on $\K_{\tau_r}(\gamma)$ w.r.t.\ either $\mu'$ or $\mu^1_0$ is equal to $\int_K G_{\K_{\tau_r}(\gamma)}(w)\mA^2(dw)$, which is positive and finite. So from the above displayed formula, we get
$$\K_{\tau_r} (\mu'|_{\Gamma_{\tau_r}})=c_K \K_{\tau_r} (\mu^1_0|_{\Gamma_{\tau_r}}).$$
Thus, $c_K$ also does not depend on $K$, and we may write it as $c$. Applying (\ref{CMP-chordal-loop}) again, we get
$\mu'|_{\Gamma_{\tau_r}}=c\mu^1_0|_{\Gamma_{\tau_r}}$.
Since both $\mu'$ and $\mu^1_0$ are supported by non-degenerate loops rooted at $0$, by letting $r,s\to 0^+$, we conclude that $\mu'=c\mu^1_0$.
\end{proof}

\begin{Remark}
We record the following fact for future references. From the proof of Theorem \ref{Thm-loop-measure} (i), we see that, if $\rho$ is any Jordan curve in $\C$ surrounding $0$, and $\tau_\rho$ is the hitting time at $\rho$, then
$ \K_{\tau_\rho}(\mu^1_0|_{\{\tau_\rho<\infty\}})=Q\cdot \K_{\tau_\rho}(\nu^\#_{0\to\infty })$,
and the Radon-Nikodym derivative $Q$ may be expressed by
$$ Q(\gamma_{\tau_\rho})=2^{d-2} \ha c^{-1}|\sin_2(\lambda_{\tau_\rho}-q_{\tau_\rho})|^{1-\frac 8\kappa} e^{(d-2)\tau_\rho}, $$
if $(e^{i\lambda_t};e^{iq_t})$ is the driving process for the whole-plane SLE$_\kappa(2)$ curve. In the proof, we only considered the case $\rho=\{|z|=r\}$, but the above formula holds for general $\rho$. Thus, $\mu^1_0|_{\{\tau_\rho<\infty\}}$ may be constructed by first weighting the law of a whole-plane SLE$_\kappa(2)$ curve stopped at $\tau_\rho$ by $Q$, and then continue with a chordal SLE$_\kappa$ curve from the tip to $0$ in the remaining domain.
\label{record}
\end{Remark}

\begin{Corollary}
Suppose that $\ha\gamma_0$ is a Minkowski content parametrization of a two-sided whole-plane SLE$_\kappa$ curve  from $\infty$ to $\infty$ passing through $0$ such that $\ha\gamma_0(0)=0$. Then $\ha\gamma_0$ is a self-similar process of index $\frac 1d$ defined on $\R$ with stationary increments. 	\label{sssi}
\end{Corollary}
\begin{proof}
	We view $\nu^\#_{\infty\rt 0}$ as a measure on unparametrized curves. Let $\ha\nu^\#_{\infty\rt 0}$ denote the law of the random parametrized curve $\ha\gamma_0$ in the statement.
	The self-similarity of $\ha\gamma_0$ follows easily from the scaling invariance of $\nu^\#_{\infty\rt 0}$  and the scaling covariance of the Minkowski content measure (Proposition \ref{conformal-content} applied to a scaling map). Since the Minkowski contents of both arms of $\ha\gamma_0$ are positive, by the self-similarity, the definition interval of $\ha\gamma_0$ has to be $\R$.
	
	Now we prove that $\ha\gamma_0$ has stationary increments. Because of the self-similarity of $\ha\gamma_0$, it suffices to show that $\ha\nu^\#_{\infty\rt 0}$ is invariant under the map ${\cal T}_1:\ha\gamma\mapsto \ha\gamma(\cdot+1)-\ha\gamma(1)$.
	
	Let $\Gamma$ denote the space of unparametrized  curves $\gamma$, which possesses Minkowski content measure that is parametrizable for $\gamma$, such that the definition domain for any Minkowski content parametrization of $\gamma$ is $\R$.
	For each $\gamma\in \Gamma$, define ${\cal T}_\gamma:\gamma\to\gamma$ such that if $\ha\gamma$ is a Minkowski content parametrization of $\gamma$, then for $z\in\gamma$, ${\cal T}_\gamma(z)=\ha\gamma(\tau_z(\ha\gamma)+1)$, where $\tau_z(\ha\gamma)$ is the first time that $\ha\gamma$ reaches $z$. Note that the definition does not depend on the choice of $\ha\gamma$. Since $\ha\gamma$ induces an isomorphism modulo zero between $(\R,\mA)$ and $(\gamma,{\cal M}_\gamma)$ (Remark \ref{Remark-param}), and $\mA$ is invariant under translation, we see that ${\cal M}_\gamma$ is invariant under ${\cal T}_\gamma$. Thus, $\mu^1_\infty(d\gamma)\otimes {\cal M}_\gamma(dz)$ is invariant under the map ${\cal T}_*:(\gamma,z)\mapsto (\gamma,{\cal T}_\gamma(z))$. By Theorem \ref{Thm-loop-measure} (iv), $\nu^\#_{\infty\rt z}(d\gamma)\otimes \mA^2(dz)$ is also invariant under ${\cal T}_*$.
			
	Define the map ${\cal R}_\Gamma(\gamma,z)=(z+\gamma,z)$ on  $\Gamma\times\C$. Since $\nu^\#_{\infty\rt z}=z+\nu^\#_{\infty\rt 0}$, we have $\nu^\#_{\infty\rt z}(d\gamma)\otimes \mA^2(dz)={\cal R}_\Gamma(\nu^\#_{\infty\rt 0}\otimes \mA^2)$. Thus, $\nu^\#_{\infty\rt 0}\otimes \mA^2$ is invariant under the map ${\cal R}_\Gamma^{-1}\circ {\cal T}_*\circ {\cal R}_\Gamma$.
	
	Let $\Gamma_0$ be the set of $\gamma\in\Gamma$ such that $0\in\gamma$ and $0$ is not a double point of $\gamma$. By scaling invariance, $\mu^\#_{\infty\rt 0}$ is supported by $\Gamma_0$. For every $\gamma\in\Gamma_0$, there is a unique Minkowski content parametrization of $\gamma$, denoted by ${\cal P}(\gamma)$ such that ${\cal P}(\gamma)(0)=0$. Then $\ha\nu^\#_{\infty\rt 0}={\cal P}(\nu^\#_{\infty\rt 0})$.
	Define  ${\cal R}_{\C}(\gamma,z)=(\gamma,z-{\cal P}(\gamma)(1))$ on $\Gamma_0\times \C$. By the translation invariance of $\mA^2$, $\nu^\#_{\infty\rt 0}\otimes \mA^2$ is also invariant under ${\cal R}_{\C}$. Thus, $\nu^\#_{\infty\rt 0}\otimes \mA^2$ is invariant under ${\cal R}_\Gamma^{-1}\circ {\cal T}_*\circ {\cal R}_\Gamma\circ{\cal R}_{\C}$.
	
	 Let $\gamma\in\Gamma_0$ and $z\in\ C$. Then $z$ is not a double point of $z+\gamma$, and $z+{\cal P}(\gamma)$ is a Minkowski content parametrization of $z+\gamma$ such that $z+{\cal P}(\gamma)(0)=z$. Thus,
	$$ {\cal T}_*\circ {\cal R}_\Gamma(\gamma,z)={\cal T}_*(z+\gamma,z)=(z+\gamma,{\cal T}_{z+\gamma}(z))=(z+\gamma,z+{\cal P}(\gamma)(1)).$$
	So we have
	$ {\cal R}_\Gamma^{-1}\circ {\cal T}_{\cal P}\circ {\cal R}_\Gamma\circ{\cal R}_{\C}(\gamma,z)=(\gamma-{\cal P}(\gamma)(1),z)$. 
	Therefore, $\nu^\#_{\infty\rt 0}$ is invariant under  $\gamma\mapsto \gamma-{\cal P}(\gamma)(1)$. So for $\nu^\#_{\infty\rt 0}$-a.s.\ $\gamma$, $ \gamma-{\cal P}(\gamma)(1)\in\Gamma_0$. Note that when $ \gamma-{\cal P}(\gamma)(1)\in\Gamma_0$, with $\ha\gamma:={\cal P}(\gamma)$, ${\cal T}_1(\ha\gamma)=\ha\gamma(\cdot+1)-\ha\gamma(1)$ is the Minkowski content parametrization of $\gamma-\ha\gamma(1)$ that satisfies ${\cal T}_1(\ha\gamma)(0)=0$, which implies that ${\cal P}(\gamma-{\cal P}(\gamma)(1))={\cal T}_1({\cal P}(\gamma))$. Since $\nu^\#_{\infty\rt 0}$ is invariant under  $\gamma\mapsto \gamma-{\cal P}(\gamma)(1)$, we get that $\ha\nu^\#_{\infty\rt 0}={\cal P}(\nu^\#_{\infty\rt 0})$ is invariant under   ${\cal T}_1$, as desired.
\end{proof}

\begin{Remark}
  In the subsequent paper \cite{Holder}, it is proved that the $\gamma$ in Corollary \ref{sssi} is locally $\alpha$-H\"older continuous for any $\alpha<1/d$, and for any deterministic closed $A\subset\R$, $\dim_H(\gamma(A))=d\cdot\dim_H(A)$, where $\dim_H$ stands for Hausdorff dimension. 
\end{Remark}

\begin{Remark}
Corollary \ref{sssi} also holds for $\kappa\ge 8$, if we replace  the two-sided SLE$_\kappa$ curve from $\infty$ to $\infty$ passing through $0$  with the SLE$_\kappa$ loop rooted at $\infty$ (with law $\mu^\#_\infty$) as described in Remark \ref{>8}, and let $d=2$ so that the ($d$-dimensional) Minkowski content agrees with the Lebesuge measure $\mA^2$. This is \cite[Lemma 2.3]{mating}. We now  give an alternative proof by modifying the above proof. The self-similarity is obvious. For the stationarity of increments, we define $\Gamma$ to be the space of space-filling curves from $\infty$ to $\infty$ that is parametrizable by $\mA^2$, and define ${\cal T}_\gamma:\C\to\C$ for each $\gamma\in\Gamma$ as in the above proof. The same argument shows that $\mA^2$ is invariant under ${\cal T}_\gamma$. Thus, $\mu^\#_\infty\otimes \mA^2$ is invariant under ${\cal T}_*:(\gamma,z)\mapsto (\gamma,{\cal T}_\gamma(z))$.
Since $\mu^\#_\infty$ is invariant under translation, $\mu^\#_\infty\otimes \mA^2$ is also invariant under ${\cal R}_\Gamma:(z,\gamma)\mapsto (z+\gamma,z)$. Define $\Gamma_0$, ${\cal P}$, and ${\cal R}_{\C}$ as in the above proof. By the scaling invariance, $\mu^\#_\infty$ is supported by $\Gamma_0$. By translation invariance of $\mA^2$, $\mu^\#_\infty\otimes \mA^2$ is also invariant under ${\cal R}_{\C}$. Thus, $\mu^\#_\infty\otimes \mA^2$ is invariant under the composition ${\cal R}_\Gamma^{-1}\circ {\cal T}_{\cal P}\circ {\cal R}_\Gamma\circ{\cal R}_{\C}:(\gamma,z)\mapsto (\gamma-{\cal P}(\gamma)(1),z)$. So $\mu^\#_\infty$ is invariant under $\gamma\mapsto \gamma-{\cal P}(\gamma)(1)$. When $\gamma-{\cal P}(\gamma)(1)\in\Gamma_0$, we have ${\cal P}(\gamma-{\cal P}(\gamma)(1))={\cal T}_1({\cal P}(\gamma))$. Thus,
$\ha\mu^\#_\infty:={\cal P}(\mu^\#_\infty)$ is invariant under  ${\cal T}_1$. So the increments are stationary.
\end{Remark}

\begin{proof} [Proof of Theorem \ref{Thm-loop-measure-unrooted}]
(i) From (\ref{defofloop},\ref{loop-unrooted}) we see that $\mu^0=\Cont(\cdot)^{-1} \cdot\int_{\C} \mu^1_z \mA^2(dz)$ and  satisfies reversibility.
Integrating (\ref{decomposition-whole-loop}) against the measure $\mA^2(dz)$ and using the above formula and the definition of $\mu^1_z$, we get
$$ \Cont(\gamma)\cdot \mu^0(d\gamma)\otimes {\cal M}_\gamma(dw)=\Cont(\gamma)\cdot \mu^1_{w}(d\gamma)\overleftarrow{\otimes} G_{\C}\cdot \mA^2(dw),$$
which immediately implies (\ref{decomposition-unrooted}) since both sides are supported by loops with positive Minkowski content. Combining (\ref{decomposition-unrooted}) with (\ref{decomposition-whole-loop}), we get (\ref{decomposition-unrooted2}).

(ii) Let $F$ be a M\"obius transformation. Applying the map $F\otimes F$ to both sides of (\ref{decomposition-unrooted}), we get two equal measures. On the left, using  Lemma \ref{conformal-content}, we get
$$F(\mu^0)(d\gamma)\otimes F({\cal M}_{F^{-1}(\gamma)})(dz)=F(\mu^0)(d\gamma)\otimes|F'(F^{-1}(z))|^{-d}\cdot {\cal M}_{\gamma}(dz).$$
On the right, using Theorem \ref{Thm-loop-measure} (iv) and (\ref{decomposition-unrooted}), we get
$$F(\mu^1_{F^{-1}(z)})(d\gamma)\overleftarrow{\otimes} F(\mA^2)(dz)=|F'(F^{-1}(z))|^{2-d}\mu^1_z\overleftarrow{\otimes}|F'(F^{-1}(z))|^{-2}\cdot \mA^2(dz)$$
$$=|F'(F^{-1}(z))|^{-d}\cdot(\mu^1_z\overleftarrow{\otimes} \mA^2(dz))=\mu^0(d\gamma)\otimes |F'(F^{-1}(z))|^{-d}\cdot {\cal M}_{\gamma}(dz).$$
Since  both $\mu^0$ and $F(\mu^0)$ are supported by loops with positive Minkowski content, by looking at the marginal measures in loops, we get $F(\mu^0)=\mu^0$.
\end{proof}

\section{SLE Loop Measures in Riemann Surfaces}\label{Section-S}
First, we use Brownian loop measure (c.f.\ \cite{loop}), the approach used in \cite{Law-mult}, and the normalized Brownian loop measure introduced in \cite{normalize} to define SLE loops in subdomains of $\ha\C$.
We are going to prove the following theorem.

\begin{Theorem} [Loops in a subdomain of $\ha\C$]
  Let $\mu^1_z$ and $\mu^0$ be as in Theorems \ref{Thm-loop-measure} and \ref{Thm-loop-measure-unrooted}. Let $D$ be a subdomain of $\ha\C$.  For $z\in D$, define
  $$\mu^1_{D;z}={\bf 1}_{\{\cdot\subset D\}}e^{\cc\Lambda^*(\cdot, D^c)}\cdot \mu^1_z, \quad
  \mu^0_D={\bf 1}_{\{\cdot\subset D\}}e^{\cc\Lambda^*(\cdot,D^c)}\cdot \mu^0.$$
  Then $\mu^1_{D;z}$ and $\mu^0_D$ satisfy the following conformal covariance and invariance, respectively.
 If $W$ maps a domain $U\subset\ha\C$ conformally onto a domain $V\subset\ha\C$, then
  \BGE W(\mu^1_{U;z})=|W'(z)|^{2-d} \mu^1_{V;W(z)},\quad \forall z\in U\sem\{\infty,W^{-1}(\infty)\};\label{mu1Uz}\EDE
  \BGE W(\mu^0_{U})=\mu^0_{V}.\label{mu0U}\EDE
 \label{Thm-subdomain}
\end{Theorem}

 Using (\ref{normalized-Brownian-equality}), we easily get the following generalized restriction property:  if $D_1\subset D_2$ are nonpolar domains, and $z\in D_1$, then
   $$\mu^1_{D_1;z}={\bf 1}_{\{\cdot\subset D_1\}}e^{\cc\mu^{\lloop}{\cal L}_{D_2}(\cdot,D_2\sem D_1)}\cdot \mu^1_{D_2;z};$$
  \BGE \mu^0_{D_1}={\bf 1}_{\{\cdot\subset D\}}e^{\cc\mu^{\lloop} {\cal L}_{D_2}(\cdot,D_2\sem D_1)}\cdot \mu^0_{D_1}.\label{G-restr}\EDE

 Now we show how Theorem \ref{Thm-subdomain} can be used to define unrooted SLE$_\kappa$ loop measure in some Riemann surfaces, such that the loop measures satisfy the generalized restriction property and conformal invariance. Let $S$ be a Riemann surface. The Brownian loop measure on $S$ was defined in \cite{Wer-loop}, which satisfies conformal invariance and the restriction property. We use $\mu^{\lloop}_S$ to denote this measure. We say that $S$ is of type I if (\ref{finite}) holds  for disjoint closed subsets $V_1,V_2$ of $S$, one of which is compact.

The definition of unrooted SLE$_\kappa$ loop on a type I Riemann surface $S$ is as follows. Let ${\cal S}$ denote the set of subdomains of $S$, which are conformally equivalent to some subdomain of $\ha\C$. For every $D\in\cal S$, we may find $E\subset \ha\C$ and   $f:E\conf D$. Then we define $\mu^0_D=f(\mu^0_E)$. From Theorem \ref{Thm-subdomain}, the value of $\mu^0_D$ does not depend on the choices of $E$ and $f$. Moreover, from (\ref{G-restr})  we get the generalized restriction property
\BGE \mu^0_{D_1}={\bf 1}_{\{\cdot\subset D_1\}} e^{\cc\mu_S^{\lloop}({\cal L}_{D_2} (\cdot, D_2\sem D_1)) }\cdot \mu^0_{D_2},\quad \forall D_1\subset D_2\in\cal S.\label{consistencty}\EDE
Using (\ref{consistencty}), we may define a measure $\mu^0_S$ on the space of (unrooted) loops in $S$, which is supported by the union of $\{\cdot\subset D\}$ over $D\in\cal S$, such that
\BGE \mu^0_D={\bf 1}_{\{\cdot\subset D\}} e^{\cc \mu_S^{\lloop}({\cal L}_S(\cdot, S\sem D))}\cdot \mu^0_S,\quad \forall D\in\cal S.\label{mu0S}\EDE
In fact,  (\ref{mu0S}) requires that $\mu^0_S|_{\{\cdot\subset D\}}=e^{-\cc \mu_S^{\lloop}({\cal L}_S(\cdot, S\sem D))}\mu^0_{D}$, where we use $\mu_S^{\lloop}({\cal L}_S(\cdot, S\sem D))<\infty$.
Let $\mu^S_D$ denote the measure on the right hand side. From (\ref{consistencty}) and the fact that ${\cal L}_S(\cdot, S\sem D_1)$ is the disjoint union of ${\cal L}_S(\cdot, S\sem D_2)$ and ${\cal L}_{D_2} (\cdot, D_2\sem D_1)$, we get the consistency criterion: $\mu^S_{D_1}=\mu^S_{D_2}|_{\{\cdot\subset D_1\}}$ if $D_1\subset D_2\in\cal S$. Thus, $\mu^0_S$ exists and is unique. We call $\mu^0_S$ the unrooted SLE$_\kappa$ loop measure in $S$. It clearly satisfies the conformal invariance and the generalized restriction property.

 We say that a Riemann surface $S$ is of type II if  (\ref{finite}) does not hold, but the normalization method in \cite{normalize} works. This means that, for any nonpolar closed subset $K$ of $S$, $S\sem K$ is of type I, and
 the limit $\Lambda^*_S(V_1,V_2)$ in (\ref{normalized-Brownian-S}) converges for disjoint closed subsets $V_1,V_2$ of $S$, one of which is compact, and does not depend on the choice of $z_0\in S$. We may also define unrooted SLE$_\kappa$ loop on a type II Riemann surface. The above approach still works except that we now use $\Lambda^*_S(\cdot,S\sem D)$ to replace the $\mu_S^{\lloop}({\cal L}_S(\cdot, S\sem D))$ in (\ref{mu0S}).

We expect that (\cite{future})  every subsurface $D$ of a compact Riemann surface $S$ is of type I or type II depending on whether $S\sem D$ can be reached by a Brownian motion on $S$. Therefore, unrooted SLE$_\kappa$ loop measure can be defined on any Riemann surface that can be embedded into a compact Riemann surface.

\begin{Remark}
there may be other ways to define SLE loops on Riemann surfaces, such as using Werner's SLE$_{8/3}$ loop measure in place of the normalized or unnormalized Brownian loop measure. St\'ephane Benoist dis some work on classifying all possible definitions of conformally invariant loop measures (\cite{Benoist}).
\end{Remark}


\begin{Remark}
If $\kappa=8/3$, we have the strong restriction property: $\mu^0_{S'}= \mu^0_{S}|_{\{\cdot\subset S'\}}$ because $\cc=0$. This measure is supported by simple loops, and so agrees with the loop measure constructed by Werner in \cite{Wer-loop} up to a positive multiplicative constant. Since $\cc=0$ when $\kappa=6$, the SLE$_6$ unrooted loop measure also satisfies the strong restriction property.
\end{Remark}

\begin{Remark}
If $\kappa=2$, and $D$ is a doubly connected domain, then $\mu^0_D$ restricted to the family $\Gamma$ of the loops in $D$ that separate the two boundary components of $D$  is a finite measure. The normalized probability measure $\mu^\#_D:=\mu^0_D|_\Gamma/|\mu^0_D|_\Gamma|$ should agree with the measure constructed in \cite{KK} as the scaling limit of the unicycle of a conditional uniform CRST.
\end{Remark}

\begin{Remark}
  If some assumption holds, we also have the CMP of rooted SLE$_\kappa$ loop measure in a subdomain of $\ha\C$. We use the measure $\mu^D_{U;a\to b}$ defined in (\ref{chordal-BL}). If it is a finite measure, then we may normalize it to get a probability measure, which is denoted by $\mu^\#_{U;a\to b}$. This is the case, e.g., if $\kappa \in(0,8/3]\cup [6,8)$.  From Proposition \ref{prop-mult} we know that $\mu^\#_{U;a\to b}$ satisfies conformal invariance. From the CMP for the rooted  SLE loop measure in $\ha\C$, we get the following CMP:
$$\K_\tau(\mu ^1_{U;z}|_{\{\tau<T_z\}})(d\gamma_\tau)\oplus \mu^\#_{\ha\C(\gamma_\tau;z)\cap U;(\gamma_\tau)_{\tip}\to z}(d\gamma^\tau)=\mu^1_{U;z}|_{\{\tau<T_z\}},$$
if $\tau$ is a nontrivial stopping time, and if $\mu^\#_{\ha\C(\gamma_\tau;z)\cap U;(\gamma_\tau)_{\tip}\to z}$ is well defined.
\end{Remark}

\begin{proof} [Proof of Theorem \ref{Thm-subdomain}.]
We first prove (\ref{mu1Uz}). We may assume that $z=0$ and $W(0)=0$.
Let $\rho$ be a Jordan curve in $\C$ that separates $0$ from $(\ha\C\sem U)\cup\{W^{-1}(\infty)\}$. Then $W(\rho)$ is a Jordan curve in $\C$ that separates $0$ from $(\ha\C\sem V)\cup\{W(\infty)\}$. Let $\tau_\rho$ and $\tau_{W(\rho)}$ be the hitting times at $\rho$ and $W(\rho)$, respectively.  From Remark \ref{record}, we see that
$$\K_{\tau_\rho}(\mu^1_0|_{\{\cdot\cap \rho\ne\emptyset\}})=Q\cdot \K_{\tau_\rho}(\nu^\#_{0\to\infty});$$
$$\K_{\tau_{W(\rho)}}(\mu^1_0|_{\{\cdot\cap W(\rho)\ne\emptyset\}})=Q\cdot \K_{\tau_{W(\rho)}}(\nu^\#_{0\to\infty}).$$
Moreover, the Radon-Nikodym derivatives $Q$ may be expressed by the following. Suppose that $\gamma$ is a whole-plane SLE$_\kappa(2)$ curve with driving process $(e^{i\lambda_t};e^{iq_t})$.
With the symbols in Section \ref{whole-kappa-rho} (e.g., $X_t=\lambda_t-q_t$, $Y_s=\sigma_s-p_s$, $\sigma_{u(t)}=\til W_t(\lambda_t)$, $p_{u(t)}=\til W_t(q_t)$), we have  $\tau_{W(\rho)}(W(\gamma))=u(\tau_\rho(\gamma))$, and
$$Q(\gamma_{\tau_\rho})=2^{d-2}\ha c^{-1}|\sin_2(X_{\tau_\rho})|^{1-\frac 8\kappa} e^{(\frac \kappa 8-1)\tau_\rho};$$
$$Q(W(\gamma_{\tau_\rho}))=2^{d-2}\ha c^{-1}|\sin_2(Y_{u(\tau_\rho)})|^{1-\frac 8\kappa} e^{(\frac \kappa 8-1)u(\tau_\rho)}.$$
Using (\ref{W2-1'},\ref{Nt'},\ref{Mt'}) and Lemma \ref{lemma-loop}, we may express the $M_{\tau_\rho}$ in Lemma \ref{whole-conformal'} as
\BGE
  M_{\tau_\rho} =|W'(0)|^{\frac\kappa 8-1}|\til W_{\tau_\rho}'(\lambda_{\tau_\rho})\til W_{\tau_\rho}'(q_{\tau_\rho})|^{\frac{6-\kappa}{2\kappa}}\cdot \frac{e^{\cc \Lambda^*(\gamma_{\tau_\rho},U^c)}}{e^{\cc  \Lambda^*(W(\gamma_{\tau_\rho}),V^c)}}\cdot\Big|\frac{\sin_2(Y_{u(\tau_\rho)})} {\sin_2(X_{\tau_\rho})}\Big|^{1-\frac 6\kappa}. \label{M-Q}
\EDE

For a curve $\gamma$  started from $0$ that intersects $\rho$, we use $\gamma_{\tau_\rho}$ and $\gamma^{\tau_\rho}$ to denote the parts of $\gamma$ before $\tau_\rho$ and after $\tau_\rho$, respectively. Since $\rho$ separates $0$ from $U^c$, $\gamma\cap U^c\ne\emptyset$ iff $\gamma^{\tau_\rho}\cap U^c\ne\emptyset$. Recall that $K_{\tau_\rho}$ is the interior hull generated by $\gamma_{\tau_\rho}$.
Suppose that  $\gamma^{\tau_\rho}$ lies in the closure of $\ha\C\sem K_{\tau_\rho}$. If a loop is disjoint from $\gamma_{\tau_\rho}$ and intersects both $\gamma^{\tau_\rho}$ and $U^c$, then the loop is contained in  $\ha\C\sem K_{\tau_\rho}$. Thus, ${\cal L}(\gamma,U^c)$ is the disjoint union of ${\cal L}(\gamma_{\tau_\rho},U^c)$ and ${\cal L}_{\ha\C\sem K_{\tau_\rho}}(\gamma^{\tau_\rho},U^c)$. Using the above facts, the formula (\ref{normalized-Brownian-equality}), the CMP for $\mu^1_0$ at  $\tau_\rho$ (note that $\ha\C(\gamma_{\tau_\rho};0)=\ha\C\sem K_{\tau_\rho}$), and Remark \ref{record},
with the notation in (\ref{chordal-BL}), we have the expression:
 $$ \mu^1_{U;0}|_{\{\cdot\cap \rho\ne\emptyset\}}=Qe^{\cc \Lambda^*(\cdot,U^c)}\cdot \K_{\tau_\rho}(\nu^\#_{0\to\infty})(d\gamma_{\tau_\rho})\oplus  \mu^{\ha\C\sem K_{\tau_\rho}}_{ U\sem K_{\tau_\rho}; (\gamma_{\tau_\rho})_{\tip}\to 0}(d\gamma^{\tau_\rho}).$$
 Similarly,
 $$ \mu^1_{V;0}|_{\{\cdot\cap W(\rho)\ne\emptyset\}}=Qe^{\cc \Lambda^*(\cdot,V^c)}\cdot \K_{\tau_{W(\rho)}}(\nu^\#_{0\to\infty})(d\beta_{\tau_{W(\rho)}})\oplus  \mu^{\ha\C\sem L_{\tau_{W(\rho)}}}_{V\sem L_{\tau_{W(\rho)}}; (\beta_{\tau_{W(\rho)}})_{\tip}\to 0}(d\beta^{\tau_{W(\rho)}}).$$
 From Lemma \ref{whole-conformal'} and (\ref{M-Q}), we find that the $W$-image of the measure
 $$ |W_{\tau_\rho}'(e^{i\lambda_{\tau_\rho}})  W_{\tau_\rho}'(e^{iq_{\tau_\rho}})|^{\frac{6-\kappa}{2\kappa}}\Big|\frac{\sin_2(Y_{u(\tau_\rho)})} {\sin_2(X_{\tau_\rho})}\Big|^{1-\frac 6\kappa}\cdot Qe^{\cc \Lambda^*(\cdot,U^c)}\cdot \K_{\tau_\rho}(\nu^\#_{0\to\infty})$$
 is
 $$|W'(0)|^{1-\frac \kappa 8} Qe^{\cc \Lambda^*(\cdot,V^c)}\cdot \K_{\tau_{W(\rho)}}(\nu^\#_{0\to\infty}).$$
From Lemma \ref{lem-mult}, we see that the $W$-image of the kernel
$$|W_{\tau_\rho}'(e^{i\lambda_{\tau_\rho}})  W_{\tau_\rho}'(e^{iq_{\tau_\rho}})|^{-\frac{6-\kappa}{2\kappa}}\Big|\frac{\sin_2(Y_{u(\tau_\rho)})} {\sin_2(X_{\tau_\rho})}\Big|^{\frac 6\kappa-1}\cdot  \mu^{\ha\C\sem K_{\tau_\rho}}_{ U\sem K_{\tau_\rho}; (\gamma_{\tau_\rho})_{\tip}\to 0}$$
is
$$\mu^{\ha\C\sem L_{\tau_{W(\rho)}}}_{V\sem L_{\tau_{W(\rho)}}; (\beta_{\tau_{W(\rho)}})_{\tip}\to 0}.$$
Combining the above six displayed formulas, we get
$$W( \mu^1_{U;0}|_{\{\cdot\cap \rho\ne\emptyset\}})=|W'(0)|^{1-\frac \kappa 8}\mu^1_{V;0}|_{\{\cdot\cap W(\rho)\ne\emptyset\}}.$$
Since $\mu^1_{U;0}$ and $\mu^1_{V;0}$ are supported by non-degenerate loops rooted at $0$, by choosing $\rho=\{|z|=1/n\}$ and letting $n\to\infty$, we finish the proof of (\ref{mu1Uz}).

Finally, we prove (\ref{mu0U}). From (\ref{decomposition-unrooted}) we get
$$ \mu^0_U(d\gamma)\otimes {\cal M}_\gamma(dz)= \mu^1_{U;z}(d\gamma) \overleftarrow{\otimes} {\bf 1}_U\cdot\mA^2(dz).$$
Applying the map $W\otimes W$ to  both sides, and using Lemma \ref{conformal-content} and (\ref{mu1Uz}), we conclude that
$$W(\mu^0_U)(d\gamma)\otimes {\cal M}_\gamma(dz)=\mu^0_V(d\gamma)\otimes {\cal M}_\gamma(dz).$$
Let $L$ be a compact subset of $V$, and $\Gamma_L=\{\gamma: \lin{\Cont}_d(\gamma\cap L)>0\}$.
Restricting both sides of the above formula to $z\in L$, and looking at the marginal measures of the first component, we get $W(\mu^0_U)|_{\Gamma_L}=\mu^0_V |_{\Gamma_L}$. Since $\mu^0_V$-a.s.\ the Minkowski content measure for $\gamma$ is strictly positive, we see that $\mu^0_V$ is supported by $\bigcup_{L\subset V} \Gamma_L$, and so does $W(\mu^0_U)$. This implies that (\ref{mu0U}) holds and completes the proof of Theorem \ref{Thm-subdomain}.
\end{proof}

\section{SLE Bubble Measures} \label{Section-B}
In this section, we will construct SLE$_\kappa$ loop measures, for $\kappa\in(0,8)$, rooted at boundary, which we also call SLE$_\kappa$ bubble measures, and study their basic properties. The SLE$_\kappa$ bubble measures were first constructed in \cite{LSW-8/3} for $\kappa=8/3$ using the restriction property of SLE$_{8/3}$, and later in \cite{CLE} for $\kappa\in(8/3,4]$ in order to construct CLE.

The argument in this section is similar to the construction of  SLE  loop measures in $\ha\C$. We will need the degenerate two-sided radial SLE process. To motivate the definition, let's consider a two-sided radial SLE$_\kappa$ curve in $\HH$ from  $a\in\R$ to $b\in\R$ through $w\in\HH$. From \cite{SW} the curve up to hitting $w$ or separating $b$ from $\infty$ is a chordal SLE$_\kappa(2,\kappa-8)$ curve started from $a$ with force points $(b,w)$ (modulo a time change). We now define a degenerate two-sided radial SLE$_\kappa$ curve in $\HH$ from  $a^-$ to $a^+$ through $w\in\HH$. Roughly speaking, it is the limit of the above curve when $b\to a^+$. More specifically, the degenerate two-sided radial SLE$_\kappa$ curve is defined by first running a  chordal SLE$_\kappa(2,\kappa-8)$ curve $\beta$ started from $a$ with force points $(a^+,w)$ up to a nontrivial stopping time $\tau$ before $w$ is reached, and then continuing it with a two-sided radial SLE$_\kappa$ curve in the remaining domain from $\beta(\tau)$  to $a^+$ through $w$. The definition does not depend on the choice of the stopping time $\tau$. Similarly, we may define degenerate two-sided radial SLE$_\kappa$ curve in $\HH$ from  $a^+$ to $a^-$ through $w$. We have the obvious reversibility property: the time-reversal of a degenerate two-sided radial SLE$_\kappa$ curve in $\HH$ from  $a^-$ to $a^+$ through $w$ is a degenerate two-sided radial SLE$_\kappa$ curve in $\HH$ from  $a^+$ to $a^-$ through $w$. Moreover, conditional on any arm (between $w$ and $a^+$ or $a^-$) of the curve, the other arm is a chordal SLE$_\kappa$ curve. From Lemmas \ref{conformal-content} and \ref{Minkowski-radial} we see that the above degenerate two-sided radial SLE$_\kappa$ curve possesses Minkowski content measure in $\C\sem \{a\}$, which is parametrizable for the loop without $a$.

From the definition, we see that a degenerate two-sided radial SLE$_\kappa$ curve satisfies CMP. We now use the language of kernels to describe this CMP. For $a\in\R$, let $\Gamma(\HH;a^+)$ denote the set of curves $\gamma$ in $\lin\HH$ (modulo a time change) from $a$ to another point $\gamma_{\tip}\in\lin\HH$, such that there is a unique connected component of $\ha\HH\sem \gamma$ whose boundary contains $(a,a+\eps)$ for some $\eps>0$ and has two distinct prime ends determined by  $a^+$ and $\gamma_{\tip}$. Let $\HH(\gamma;a^+)$ denote this connected component. For $\gamma$ in this space, the chordal SLE$_\kappa$ measure $\mu^\#_{\HH(\gamma;a^+);\gamma_{\tip}\to a^+}$ is well defined, and the map from $\gamma$ to this measure is a kernel.  For $a\in\R$ and  $w\in\HH $, let $\Gamma(\HH;a^+;w)$ denote the set of $\gamma\in\Gamma(\HH;a^+)$ such that $w\in \HH(\gamma;a^+)$. For $\gamma$ in this space,  the two-sided radial SLE$_\kappa$ measure $\nu^\#_{\HH(\gamma;a^+);\gamma_{\tip}\to w\to a^+}$ is well defined, and the map from $\gamma$ to this measure is a kernel. Let $\nu^\#_{\HH;a^-\to w;a^+}$ denote the law of a chordal SLE$_\kappa(2,\kappa-8)$ curve started from $a$ with force points $(a^+,w)$. Let $\nu^\#_{\HH;a^-_+\rt w}$ denote the law of a degenerate two-sided radial SLE$_\kappa$ curve (modulo a time change) from $a^-$ to $a^+$ through $w$. The CMP of the degenerate two-sided radial SLE can now be stated as follows. If $\tau$ is a nontrivial stopping time, then
\BGE \K_\tau(\nu^\#_{\HH;a^-\to  w;a^+})|_{\{\tau<T_{w}\}}(d\gamma_\tau)\oplus \nu^\#_{\HH( \gamma_\tau;a^+);(\gamma_\tau)_{\tip}\to w \to a^+}(d\gamma^\tau)=\nu^\#_{\HH;a^-_+\rt w}|_{\{\tau<T_{w}\}},\label{CMP-degenerate}\EDE
where implicitly in the formula is that $\K_\tau(\nu^\#_{\HH;a^-\to w;a^+})|_{\{\tau<T_{w}\}}$ is supported by $\Gamma(\HH;a^+;w)$.

We may similarly define $\Gamma(\HH;a^-)$, $\HH(\gamma;a^-)$, and $\Gamma(\HH;a^-;w)$. Then (\ref{CMP-degenerate})
holds with $a^+$, $a^-$ and $a^-_+$ replaced with $a^-$, $a^+$ and $a^+_-$, respectively.

For $a\in\R$, we use $\nu^\#_{\HH;a^-\to\infty;a^+}$ to denote the law of a chordal SLE$_\kappa(2)$ curve started from $a$ with force point $a^+$. The following proposition described the Radon-Nikodym derivatives between $\nu^\#_{\HH;a^-\to w;a^+}$ and $\nu^\#_{\HH;a^-\to\infty;a^+}$ stopped at certain stopping times, which follows immediately from \cite[Theorem 6]{SW}.

\begin{Proposition}
	Let $a\in\R$, $w\in\HH$, and let $\tau_w$ be the first time that the curve visits $w$ or disconnects $w$ from $\infty$. Then for any stopping time $\tau$,
	$$\K_\tau({\bf 1}_{\{\tau<\tau_w\}}\cdot \nu^\#_{\HH;a^-\to w;a^+})(d\gamma_\tau)=R_w({\gamma_\tau})\cdot \K_\tau({\bf 1}_{\{\tau<\tau_w\}}\cdot \nu^\#_{\HH;a^-\to\infty;a^+})(d\gamma_\tau),$$
	where $R_w({\gamma_\tau})$ is given by the following. Let  $\gamma_\tau$ be parametrized by half-plane capacity, and let $\lambda_t$ and $g_t$, $0\le t\le \tau$, be the chordal Loewner driving function and maps for $\gamma_\tau$ (see, e.g., Appendix \ref{Appendix}). Then
	\BGE R_w({\gamma_\tau})=|g_\tau'(w)|^{\frac{8-\kappa}8}\Big(\frac{|g_\tau(w)-\lambda_\tau|}{|w-a|}\Big)^{\frac{\kappa-8}\kappa} \Big(\frac{|g_\tau(w)-g_\tau(a^+)|}{|w-a|}\Big)^{\frac{\kappa-8}\kappa} \Big(\frac{\Imm g_\tau(w)}{\Imm w}\Big)^{\frac{(\kappa-8)^2}{8\kappa}}.\label{Q_tau}\EDE
\end{Proposition}

\begin{Theorem}
	Let
	$ G_{\HH}(w)= |w|^{\frac2\kappa(\kappa-8)}(\Imm w)^{\frac{(\kappa-8)^2}{8\kappa}}$. Then the following are true.
	\begin{enumerate} [label=(\roman*)]
		\item  For every $a\in\R$, there is a unique $\sigma$-finite measure $\mu^1_{\HH;a^-_+}$, which is supported by non-degenerate loops in $\lin\HH$ rooted at $a$ which possess Minkowski content measure in $\C\sem\{0\}$ that is parametrizable for the loop without $a$, and satisfies
		\BGE \mu^1_{\HH;a^-_+}(d\gamma)\otimes {\cal M}_{\gamma;\C\sem\{0\}}(dw)=\nu^\#_{\HH;a^-_+\rt w}(d\gamma)\overleftarrow{\otimes} G_{\HH}(w-a)\cdot \mA^2(dw),\quad a\in\R.\label{decomposition-loop-d}\EDE
		Moreover, the time-reversal of $\mu^1_{\HH;a^-_+}$ is $\mu^1_{\HH;a^+_-}$, which satisfies the same property as  $\mu^1_{\HH;a^-_+}$ except that (\ref{decomposition-loop-d}) is modified with $a^+$ and $a^-$ swapped.
		\item  For every $a\in\R$, $\mu^1_{\HH;a^-_+}$ satisfies the following CMP: if $\tau$ is a nontrivial stopping time, then
		\BGE \K_\tau(\mu ^1_{\HH;a^-_+}|_{\{\tau<T_a\}})(d\gamma_\tau)\oplus \mu^\#_{\HH(\gamma_\tau;a^+);(\gamma_\tau)_{\tip}\to a^+}(d\gamma^\tau)=\mu^1_{\HH;a^-_+}|_{\{\tau<T_a\}}, \label{CMP-chordal-loop-d}\EDE
		where implicitly stated in the formula is that $\K_\tau(\mu^1_{\HH;a^-_+})|_{\{\tau<T_a\}}$ is supported by $\Gamma(\HH;a^+)$.
		\item  Let  $J(z)=-1/z$, and
		$ \mu^1_{\HH;\infty^+_-}=J(\mu^1_{\HH;0^-_+})$.
		Then $\mu^1_{\HH;\infty^+_-}$ is  supported by loops in $\lin\HH$ rooted at $\infty$, which possesses Minkowski content measure (in $\C$), and satisfies
		\BGE \mu^1_{\HH;\infty^+_-}(d\gamma)\otimes {\cal M}_\gamma(dw)=\nu^\#_{\HH;\infty^+_- \rt w}(d\gamma)\overleftarrow{\otimes} (\Imm w)^{\frac{(\kappa-8)^2}{8\kappa}} \mA^2(dw),\label{decomposition-loop-infty-d}\EDE
		where we define $\nu^\#_{\HH;\infty^+_-\rt w}=J(\nu^\#_{\HH;0^-_+\rt J(w)})$.		Moreover, for any bounded set $S\subset \C$, $\mu^1_{\HH;\infty^+_-}$-a.s.\ $\lin\Cont(\gamma\cap S)<\infty$.
		\item If $F$ is a  M\"obius automorphism of $\HH$, then $F(\mu^1_{\HH;x^-_+})=|F'(x)|^{\frac8\kappa -1} \mu^1_{\HH;F(x)^-_+}$ for any $x\in\R$ such that $F(x)\in\R$. If $F(z)=az+b$ for some $a,b\in\R$ with $a>0$, then  $F(\mu^1_{\HH;\infty^+_-})=|a|^{1-\frac 8\kappa}\mu^1_{\HH;\infty^+_-}$.
		\item For any $r>0$ and $a\in\R$, $\mu^1_{\HH;a^-_+}(\{\gamma:\diam(\gamma)>r\})$is finite. Moreover,
		there is a constant $C\in (0,\infty)$ such that $\mu^1_{\HH;a^-_+}(\{\gamma:\diam(\gamma)>r\})=Cr^{1-\frac 8\kappa}$  for any $a\in\R$ and $r>0$.
		\item  For $a\in\R$, if a measure $\mu'$ supported by non-degenerate loops in $\lin\HH$ rooted at $a$ satisfies (ii) and that $\mu'(\{\gamma:\diam(\gamma)>r\})<\infty$ for every $r>0$, then $\mu'=c\mu^1_{\HH;a^-_+}$ for some $c\in[0,\infty)$.
	\end{enumerate}
	\label{Thm-loop-measure-d}
\end{Theorem}

\begin{Remark}
For $\kappa\ge 8$, it is easy to construct an SLE$_\kappa$ bubble measure $\mu^\#_{\HH;a^-_+}$ that satisfies the CMP as in Theorem \ref{Thm-loop-measure-d} (ii). This is similar to Remark \ref{>8}. To construct a random curve with law $\mu^\#_{\HH;a^-_+}$, we start a chordal SLE$_\kappa(\kappa-6)$ curve in $\HH$ from $a$ to $\infty$ with force point $a^+$, and after the curve reaches $\infty$, we continue it with a chordal SLE$_\kappa$ curve from $\infty$ to $0$ growing in the remaining domain. \label{Remark-8}
\end{Remark}

\begin{proof}[Proof of Theorem \ref{Thm-loop-measure-d}] This proof is very similar to the proof of Theorem \ref{Thm-loop-measure}. We will point out the main difference and omit the parts that are similar.
	
	(i) It suffices to consider the case $a=0$. Let $\gamma_\tau(t)$, $0\le t\le \tau$, be a chordal Loewner curve started from $0$ with driving function $\lambda_t$, $0\le t\le \tau$. Let $g_t$ be the corresponding Loewner maps. Suppose $\gamma_\tau\cap (0,\infty)=\emptyset$.
	Then $g_\tau:(\HH(\gamma_\tau;0^+);\infty,(\gamma_\tau)_{\tip}=\gamma_\tau(\tau),0^+)\conf(\HH;\infty,\lambda_\tau,q_\tau)$ for some $q_\tau>\lambda_\tau$.
	We have the chordal SLE$_\kappa$ measure $\mu^\#_{\HH(\gamma_\tau;0^+);\gamma_\tau(\tau)\to 0^+}$ and the two-sided radial SLE$_\kappa$ measure $\nu^\#_{\HH(\gamma_\tau;0^+);\gamma_\tau(\tau)\to w\to 0^+}$ for each $w\in \HH(\gamma_\tau;0^+)$. Since these measures are all determined by $\gamma_\tau$, we now write $\mu^\#_{\gamma_\tau}$ and $\nu^\#_{\gamma_\tau;w}$, respectively, for them.
	We write $G_{\gamma_\tau}(w)$ for the Green's function $G_{\HH(\gamma_\tau;0^+);\gamma_\tau(\tau)\to 0^+}(w)$. Let $K$ be a compact subset of $\lin\HH\sem\{0\}$ such that $K\cap \gamma_\tau=\emptyset$.
	From Proposition \ref{decomposition-Thm}, we have
	\BGE \mu^\#_{\gamma_\tau}(d\gamma)\otimes {\cal M}_{\gamma\cap K}(dw)=\nu^\#_{\gamma_\tau;w}(d\gamma)\overleftarrow{\otimes} ({\bf 1}_K G_{\gamma_\tau}\cdot \mA^2)(dw).\label{decomposition-degenerate}\EDE
	
	We now compute $G_{\gamma_\tau}(w)=G_{\HH(\gamma_\tau;0^+);\gamma_\tau(\tau)\to 0}(w)$ for $w\in\HH(\gamma_\tau;0^+)$.
	  Let $\phi(z)=\frac{z-{\lambda_\tau}}{q_\tau-z}$.
	Then $\phi:(\HH;\lambda_{\tau},q_\tau)\conf(\HH;0,\infty)$.  Recall that $g_\tau:(\HH(\gamma_\tau;0^+);  \gamma_\tau(\tau),0^+)\conf(\HH;\ \lambda_\tau,q_\tau)$.
	By (\ref{Green-H}) and (\ref{Green}), we get
	\begin{align}
	G_{\gamma_\tau}(w)&=\ha c\,|g_\tau'(w)|^{2-d}|\phi'(g_\tau(w))|^{2-d} |\phi(g_\tau(w))|^{1-\frac 8\kappa} (\Imm\phi(g_\tau(w)))^{\frac \kappa 8+\frac 8\kappa -2} \nonumber \\
	&=\ha c|g_\tau'(w)|^{2-d}\cdot \frac{|q_\tau-\lambda_\tau|^{\frac 8\kappa-1}}{|g_\tau(w)-q_\tau|^{\frac 8\kappa -1}|g_\tau(w)-\lambda_\tau|^{\frac 8\kappa -1}}\cdot  (\Imm g_\tau(w))^{{\frac{\kappa}{8}+\frac 8\kappa -2}}. \label{G-degenerate}
	\end{align}
	
	Let $G_\HH(w)$ be as in the statement and
	\BGE Q(\gamma_\tau)= \ha c^{-1}|q_\tau-\lambda_\tau|^{1-\frac 8 \kappa}.\label{R-gamma-d}\EDE
	Using (\ref{Q_tau}) and (\ref{decomposition-d'}), we find that
	\BGE Q(\gamma_\tau)G_{\gamma_\tau}(w)=R_w({\gamma_\tau})G_{\HH}(w).\label{QGRG-d}\EDE
	From (\ref{decomposition-degenerate}) and (\ref{QGRG-d}), we get
	\BGE Q(\gamma_\tau)\mu^\#_{\gamma_\tau }(d\gamma^\tau)\otimes {\cal M}_{\gamma\cap K}(dw)=(R_w({\gamma_\tau}) \nu^\#_{\gamma_\tau ;w})(d\gamma^\tau)\overleftarrow{\otimes} ( {\bf 1}_K G_{\HH} \cdot \mA^2)(dw).\label{decomposition-d'}\EDE
	Note that the above two formulas are similar to (\ref{QGRG},\ref{decomposition-whole'}).
	
	For any stopping time $\tau$, define
	$$\Gamma_{\tau}=\{\gamma:\tau(\gamma)<T_0(\gamma),\K_\tau(\gamma) \in \Gamma(\HH;0^+),\gamma([0,\tau])\cap(0,\infty)=\emptyset\}.$$
	We view both sides of (\ref{decomposition-d'}) as kernels from $\gamma_\tau\in \Gamma_\tau$ to the space of curve-point pairs. Let $K$ be a fixed compact subset of $\lin\HH\sem\{0\}$. Let $\Gamma_{\tau;K}=\Gamma_\tau\cap\{\gamma:K\subset \HH({\K_\tau(\gamma)};0^+)\}$.
	Acting $ {\cal K}_\tau({\bf 1}_{\Gamma_{\tau;K}}\cdot\nu^\#_{\HH;0^-\to\infty;0^+}) (d \gamma_\tau)\otimes$ on the left of both sides of (\ref{decomposition-d'}), we get an equality of two measures on the space of curve-curve-point triples $(\gamma_\tau,\gamma^\tau,w)$, i.e.,
	\begin{align*}
	&[Q   \cdot {\cal K}_\tau({\bf 1}_{\Gamma_{\tau;K}}\cdot\nu^\#_{\HH;0^-\to\infty;0^+})(d\gamma_\tau)\otimes \mu^\#_{\gamma_\tau }(d\gamma^\tau)]\otimes {\cal M}_{\gamma^\tau\cap K}(dw)\\
	=& [ {\cal K}_\tau({\bf 1}_{\Gamma_{\tau;K}}\cdot\nu^\#_{\HH;0^-\to w;0^+})(d\gamma_\tau)\otimes  \nu^\#_{\gamma_\tau ;w}(d\gamma^\tau)]\overleftarrow{\otimes} ( {\bf 1}_K G_{\HH} \cdot \mA^2)(dw).
	\end{align*}
	
	The rest of the proof of (i) is almost the same as the part of the proof of Theorem \ref{Thm-loop-measure} (i) starting from the paragraph containing (\ref{decomposition-whole-triple}) except for the following modifications: we use  $\lin\HH\sem\{0\}$, $G_{\HH}$, $\nu^\#_{\HH;0^-\to\infty;0^+}$, $\nu^\#_{\HH;0^-\to w;0^+}$ $\nu^\#_{\HH;0^-_+\rt w}$, $\mu^1_{\HH;0^-_+}$, $\mu^1_{\HH;\infty^+_-}$, $\Gamma(\HH;0^+)$, $\HH(\cdot;0^+)$, and ${\cal M}_{\cdot;\C\sem\{0\}}$ to replace $\C\sem \{0\}$, $G_{\C}$, $\nu^\#_{0\to\infty}$, $\nu^\#_{0\to w}$, $\nu^\#_{0\rt w}$, $\mu^1_0$, $\mu^1_\infty$, $\Gamma(\C;0)$, $\ha\C(\cdot;0)$, and ${\cal M}_{\cdot}$, respectively.
	
	We need to prove the uniqueness of $\mu^1_{\HH;0^-_+}$ without a formula similar to (\ref{defofloop}). Suppose $\mu$ satisfies the properties of  $\mu^1_{\HH;0^-_+}$. Let $K$ be a compact subset of $\lin\HH\sem\{0\}$. Let $r\in(0,\dist(0,K))$ and $\tau_r$ be the first time that the curve reaches $\{|z|=r\}$. Restricting (\ref{decomposition-loop-d}) for $\mu$ to $\gamma\in \Gamma_{\tau_r}$ and $w\in K$, we get
	$$\mu|_{\Gamma_{\tau_r}}(d\gamma)\otimes {\cal M}_{\gamma\cap K}(dw)= \nu^\#_{\HH;0^-_+\rt w}\overleftarrow{\otimes} {\bf 1}_K G_{\HH}\cdot\mA^2(dw).$$
	Since $\mu$-a.s., ${\cal M}_{\gamma\cap K}$ is a finite measure, from the above formula, we get
	$$\mu|_{  \Gamma_K}=\Cont(\cdot)^{-1}\cdot \int_K \nu^\#_{\HH;0^-_+\rt w} G_{\HH}(w)\mA^2(dw)=\mu^1_{\HH;0^-_+}|_{\Gamma_K},$$
	where $\Gamma_K:=\{\gamma: \lin{\Cont}_d(\gamma\cap K)>0\}\subset \Gamma_{\tau_r}$. From the assumption we see that both $\mu$ and $\mu^1_{\HH;0^-_+}$ are supported by $\bigcup_K \Gamma_K$. So they must agree. Finally, the reversibility of $\mu^1_{\HH;0^-_+}$ follows from (\ref{decomposition-loop-d}), the reversibility of $\nu^\#_{\HH;0^-_+\rt w}$, and the uniqueness of $\mu^1_{\HH;0^-_+}$.
	
	(ii, iii, iv) The  proofs of (ii, iii, iv) are almost the same as  the proofs of Theorem \ref{Thm-loop-measure} (ii, iv, v), respectively, except for the modifications described near the end of the proof of (i).
	
	(v) By the translation invariance and the scaling property (iv), it suffices to prove the first sentence of (v) for $a=0$ and $r=1$. We will use the chordal Loewner equation (see Appendix \ref{Appendix}).
	Let $\gamma$ be a chordal SLE$_\kappa(2)$ curve started from $0$ with force point at $0^+$. Let $\gamma$ be parametrized by half-plane capacity, and $\lambda$ be its chordal Loewner driving function. Let $K_t$ and $g_t$ be the chordal Loewner hulls and maps, respectively, driven by $\lambda$. Recall that $\HH\sem K_t$ is the unbounded connected component of $\HH\sem \gamma([0,t])$, $g_t(\HH\sem K_t;\infty)\conf(\HH;\infty)$, satisfies $g_t'(\infty)=1$, and maps $\gamma(t)$ to $\lambda_t$. Let $K_t^{\doub}=\lin{K_t}\cup\{z\in-\HH:\lin z\in K_t\}$. By Schwarz reflection principle, $g_t$ extends to  $g_t:\C\sem K_t^{\doub}\conf\C\sem [q^-_t,q^+_t]$ for some $q^-_t\le q^+_t\in\R$. Since $\gamma(t)\in\pa K_t$, we get $q^-_t\le \lambda_t\le q^+_t$.
	Since $g_t'(\infty)=1$, we have $\ccap(K_t^{\doub})=\ccap([q^-_t,q^+_t])$. Thus, $\diam(\gamma([0,t]))\le \diam(K_t^{\doub})\le q^+_t-q^-_t$. By chordal Loewner equation (\ref{chordal-eqn}) and the definition of SLE$_\kappa(2)$ process (note that $(q^+_t)$ is the force point process), we see that  $(\lambda_t)$ and   $(q^\pm_t)$ satisfy the SDE
	\begin{align*}
	d\lambda_t&=\sqrt\kappa dB_t+\frac{2}{\lambda_t-q^+_t}\,dt;\\
	dq^\pm_t&=\frac{2}{q^\pm_t-\lambda_t}\,dt,
	\end{align*}
	for some Brownian motion $B_t$.
	So we have $d\log(q^+_t-q^-_t)=\frac{2}{(q^+_t-\lambda_t)(\lambda_t-q^-_t)}>0$.
	Let $$ Y_t=\frac{q^+_t-\lambda_t}{q^+_t-q^-_t}-\frac{\lambda_t-q^-_t}{q^+_t-q^-_t}\in[-1,1].$$
	Then $Y_t+1=\frac{2(q^+_t-\lambda_t)}{q^+_t-q^-_t}$ and $Y_t-1=\frac{2(q^-_t-\lambda_t)}{q^+_t-q^-_t}$.
	By It\^o's formula, $Y_t$ satisfies the SDE
	$$d Y_t=\frac{-2\sqrt\kappa}{q^+_t-q^-_t}\,dB_t+\frac{6}{(q^+_t-\lambda_t)(q^+_t-q^-_t)}\,dt -\frac{2}{(\lambda_t-q^-_t)(q^+_t-q^-_t)}\,dt-\frac{2 Y_t}{(q^+_t-\lambda_t)(\lambda_t-q^-_t)}\,dt.$$
	Let $u(t)=\frac \kappa 2 \log(q^+_t-q^-_t)$. Then $u$ is absolutely continuous and strictly increasing, and maps $(0,\infty)$ onto $(-\infty,\infty)$. Moreover, $u'(t)=\frac{\kappa}{(q^+_t-\lambda_t)(\lambda_t-q^-_t)}$ whenever $q^-_t<\lambda(t)<q^+_t$, which holds for almost every $t>0$. Let $v=u^{-1}$, $\ha Y_t=Y_{v(t)}$. By a straightforward computation, we find that $\ha Y_t$, $-\infty<t<\infty$, satisfies the SDE
	$$d\ha Y_t=-\sqrt{1-\ha Y_t^2}d\ha B_t-\frac 2\kappa (\ha Y_t+1)dt-\frac 4\kappa (\ha Y_t-1)dt.$$
	This agrees with the SDE in \cite[Remark 3 after Corollary 8.5]{tip} with $\delta_+=\frac 8\kappa$ and $\delta_-=\frac{16}\kappa$. Thus, for each fixed deterministic $t\in\R$, the law of $\ha Y_t$ has a density w.r.t.\ ${\bf 1}_{[-1,1]}\cdot\mA$, which is proportional to $(1-x)^{\frac4\kappa -1}(1+x)^{\frac 8\kappa -1}$. Let $\tau=v(0)$. Then $\tau$ is the first time $t$ such that $q^+_t-q^-_t=1$. So we get $1+Y_{\tau}=2(q^+_{\tau}-\lambda_{\tau})$. Thus, the law of $q^+_{\tau}-\lambda_{\tau}$ is proportional to ${\bf 1}_{[0,1]}(x)x^{\frac 8\kappa -1}(1-x)^{\frac4\kappa -1}\cdot \mA$. From the construction of $\mu^1_{\HH;0^-_+}$ and (\ref{R-gamma-d}) we see that
	$$\K_\tau(\mu^1_{\HH;0^-_+}|_{\Gamma_\tau})=\ha c^{-1}(q^+_\tau-\lambda_\tau)^{1-\frac 8\kappa}\cdot \K_\tau(\nu^\#_{\HH;0^-\to\infty;0^+}|_{\Gamma_\tau}).$$
	Thus, $$\mu^1_{\HH;0^-_+}(\Gamma_\tau)=\frac{\int_0^1 (1-x)^{\frac4\kappa -1}dx}{\int_0^1x^{\frac 8\kappa -1}(1-x)^{\frac4\kappa -1}dx }<\infty.$$
	Let $\tau_{[0,\infty)}$ be the first $t>0$ such that $\gamma(t)\in[0,\infty)$. The above formula implies that, for any
	$$\mu^1_{\HH;0^-_+}\Big[\sup_{0\le t\le \tau_{[0,\infty)}} q^+_t-q^-_t>1\Big]\le\mu^1_{\HH;0^-_+}(\Gamma_\tau)<\infty.$$
	Since $\diam(K_t)\le q^+_t-q^-_t$ for $t\le \tau_{[0,\infty)}$, we have $\mu^1_{\HH;0^-_+}[\diam(K_{\tau_{[0,\infty)}})>1]<\infty$. Since $\gamma$ either ends at $\tau_{[0,\infty)}$ (when $\kappa\in(0,4]$) or grows inside $K_{\tau_{[0,\infty)}}$ after $\tau_{[0,\infty)}$ (when $\kappa\in(4,8)$), we get $\diam(\gamma)=\diam(K_{\tau_{[0,\infty)}})$. Thus, $\mu^1_{\HH;0^-_+}[\diam(\gamma)>1]<\infty$.
	
	(vi) The  proofs of (vi) is almost the same as  the proof of Theorem \ref{Thm-loop-measure} (vii) except for the modifications described near the end of the proof of (i).
\end{proof}

\begin{Theorem}  Let   $\mu^1_{\HH;a^-_+}$ be as in the previous theorem.  Let $D\subset \HH$ be an open neighborhood of $\R\cup\{\infty\}$ in $\HH$.  Define
	$$\mu^1_{D;a^-_+}={\bf 1}_{\{\cdot\subset D\}}e^{\cc\mu^{\lloop}({\cal L}_{\HH}(\cdot,\HH\sem D))}\cdot \mu^1_{\HH;a^-_+},\quad a\in\R.$$
	Then $\mu^1_{D;a^-_+}$ satisfies the following conformal covariance.
If $W$ maps  $D$ conformally onto another domain $E$ with the same properties as $D$, and maps $a\in\R$ to $b\in\R$, then
		$$W(\mu^1_{D;a^-_+})=|W'(a)|^{\frac 8\kappa -1} \mu^1_{E;b^-_+}.$$ \label{Thm-subdomain-bubble}
\end{Theorem}
\begin{proof}
  The proof is similar to that of Theorem \ref{Thm-subdomain} except that here we use  Lemma \ref{lem-mult'},  (\ref{S-loop-chordal}), and  Lemma \ref{chordal-conformal} below to replace Lemmas \ref{lem-mult}, \ref{lemma-loop}, and \ref{whole-conformal'} in the proof of Theorem \ref{Thm-subdomain}, and the role of (\ref{normalized-Brownian-equality}) is played by an equality of Brownian loop measures without normalization.
\end{proof}

\begin{Lemma}
  Let $a>a'>0$ be such that the circle $\{|z|=a\}$ separates $0$ from $\HH\sem U$. Let $\rho=\{|z|={a'}\}$. Let $\tau_\rho$ and $\tau_{W(\rho)}$ be the hitting time at $\rho$ and $W(\rho)$, respectively. Then
  $$\K_{\tau_{W(\rho)}}(\nu^\#_{\HH;0^-\to\infty;0^+})= W(M_{\tau_\rho}\cdot \K_{\tau_{\rho}}(\nu^\#_{\HH;0^-\to\infty;0^+})),$$
  where $(M_t)$ is a local martingale defined as follows.

  Suppose that $\gamma$ has the law $\nu^\#_{\HH;0^-\to\infty;0^+}$, i.e., is a chordal SLE$_\kappa(2)$ curve started from $0$ with force point $0^+$. Let $\lambda_t$ and $q_t$ be its driving function and force point process, respectively, and let $X_t=\lambda_t-q_t$. Let $(K_t)$  be the  chordal Loewner hulls  driven by $\lambda$. Let $L_t=W(K_t)$. Suppose that $g_{K_t}:(\HH\sem K_t;\infty)\conf(\HH;\infty)$ and $g_{L_t}:(\HH\sem L_t;\infty)\conf (\HH;\infty)$ behave like $z+o(1)$ as $z\to \infty$. Let $W_t=g_{L_t}\circ W\circ g_{K_t}^{-1}$, $\sigma_t=W_t(\lambda_t)$, $p_t=W_t(q_t)$, and $Y_t=\sigma_t-p_t$. Then $$M_t(\gamma)=W_t'(\lambda_t)^{\frac{6-\kappa}{2\kappa}}W_t'(q_t)^{\frac{6-\kappa}{2\kappa}}\Big|\frac{Y_t}{X_t}\Big|^{\frac 2\kappa} \exp\Big(-\frac{\cc}6 \int_0^t S(W_s)(\lambda_s)ds\Big).$$
  \label{chordal-conformal}
\end{Lemma}
\begin{proof}
  This follows from a chordal version of the argument in the proof of Lemma \ref{whole-conformal'}, which uses chordal Loewner equations, and is similar to the one used in the proof of Proposition \ref{prop-mult}.  Here we use $Y_t$ instead of the $Y_{u(t)}$ as in (\ref{Nt'}) because we did not do a time-change on $(L_t)$.
\end{proof}


\begin{Remark}
For $\kappa\in(4,8)$, there is another way to construct the SLE$_\kappa$ bubble measure. The construction uses two-sided chordal SLE. Roughly speaking, a two-sided chordal SLE$_\kappa$ curve is a chordal SLE$_\kappa$ conditioned to pass through a fixed boundary point.
For $a\ne x\in\R$, the degenerate two-sided chordal SLE$_\kappa$ curve in $\HH$ from $a^-$ to $a^+$ passing through $x$ can be defined as the limit as $b\to a^+$ of a two-sided chordal SLE$_\kappa$ curve in $\HH$ from $a$ to $b$ passing through $x$.
The degenerate two-sided chordal SLE$_\kappa$ curve satisfies the reversibility as a two-sided whole-plane SLE$_\kappa$ curve does. \cite[Theorem 6.1]{decomposition} states that if we integrate the law of two-sided chordal SLE$_\kappa$ curves in $\HH$ from $0$ to $\infty$ passing through different $x\in\R$ against the measure ${\bf 1}_U\cdot\mA(dx)$, where $U$ is a compact subset of $\R\sem\{0\}$, we get a law, which is absolutely continuous w.r.t.\ that of a chordal SLE$_\kappa$ in $\HH$ from $0$ to $\infty$, and the Radon-Nikodym derivative may be described as the $(2-\frac 8\kappa)$-dimensional Minkowski content of the intersection of the curve with $U$. Here we use Lawler's result on the existence of the Minkowski content of the intersection of SLE$_\kappa$ curve with the domain boundary \cite{Law-real}, which was conjectured in \cite{AlbertsSheffield} and later solved  We may derive a theorem that is similar to Theorem \ref{Thm-loop-measure} except for the following modifications: the measure $\nu^\#_{z\rt w}$ should be replaced by $\nu^\#_{\HH;x^-_+\rt y^+_-}$, the law of a degenerate two-sided chordal SLE$_\kappa$ curve in $\HH$ from $x^-$ to $x^+$ passing through $y$; the function $G_{\C}(w-z)$ should be replaced by $G_{\HH}(y-x):=|x-y|^{-2(\frac 8\kappa -1)}$; the measure $\mA^2(dw)$ should be replaced by $\mA(dy)$; the $d$-dimensional Minkowski content $\Cont(\cdot)$ and Minkowski content measure ${\cal M}_\gamma$ of $\gamma$ should be replaced by the $(2-\frac 8\kappa)$-dimensional Minkowski content $\Cont_{2-\frac 8\kappa}(\cdot\cap \R)$ and Minkowski content measure  ${\cal M}^{(2-\frac 8\kappa) }_{\gamma\cap\R}$ of $\gamma\cap\R$; the measure  ${\cal M}^{(2-\frac 8\kappa) }_{\gamma\cap\R}$ is not parametrizable for the curve, so here we do not have a statement similar to Theorem \ref{Thm-loop-measure}  (iii); and  the exponents $d-2$ and $(d-2)/d$ in (vi) should be replaced by $1-\frac8\kappa$ and $(1-\frac8\kappa)/(2-\frac8\kappa)$, respectively. The statements on the CMP and uniqueness in this theorem and Theorem \ref{Thm-loop-measure-d} ensures that the bubble measure constructed in the two theorems are equal up to a multiplicative constant because of the uniqueness. Moreover, following the proof of Theorem \ref{Thm-loop-measure-unrooted}, we may construct an unrooted SLE$_\kappa$ bubble measure, which is invariant under M\"obius automorphisms of $\HH$. Following the proof of (\ref{mu0U}) in Theorem \ref{Thm-subdomain}, we can prove that this unrooted loop measure satisfies the generalized restriction property without the factor $|W'(a)|^{\frac 8\kappa -1}$ as in Theorem \ref{Thm-subdomain-bubble}. Then we may follow the argument after Theorem \ref{Thm-subdomain} to define unrooted SLE$_\kappa$ measure $\mu_{S;C}$ in any Riemann surface $S$ with a boundary component $C$, which is conformally invariant and satisfies the generalized restriction property.
\end{Remark}
\appendixpage
\addappheadtotoc
\appendix
\section{Chordal SLE in Multiply Connected Domains} \label{Appendix}
In the appendix, we review the definition of chordal SLE in multiply connected domains for $\kappa\in(0,8)$. First, we review hulls, Loewner chains and chordal Loewner equations, which define chordal SLE in simply connected domains. The reader is referred to \cite{Law1} for details.

A subset $K$ of $\HH$ is called an $\HH$-hull if it is bounded, relatively closed in $\HH$, and $\HH\sem K$ is simply connected. For every $\HH$-hull $K$, there is are a unique $c\ge 0$ and a unique $g_K:\HH\sem K\conf \HH$ such that $g_K(z)=z+\frac{c}z+O(\frac 1{z^2})$ as $z\to\infty$. The number $c$ is called the $\HH$-capacity of $K$, and is denoted by $\hcap(K)$.

 If $K_1\subset K_2$ are two $\HH$-hulls, we define $K_2/K_1=g_{K_1}(K_2\sem K_1)$. Then $K_2/ K_1$ is also an $\HH$-hull, and we have $\hcap(K_2/K_1)=\hcap(K_2)-\hcap(K_1)$.

 The following proposition is essentially Lemma 2.8 in \cite{LSW1}.
\begin{Proposition}
Let $W$ be a conformal map defined on a neighborhood of $x_0\in\R$ such that an open real interval containing $x_0$ is mapped into $\R$. Then
$$\lim_{H\to z_0} \frac{\hcap(W(H))}{\hcap(H)}=|W'(z_0)|^2,$$
where $H\to z_0$ means that $\diam (H\cup\{z_0\})\to 0$ with  $H$ being a nonempty $\HH$-hull. \label{chordal-hull}
\end{Proposition}

Let $T\in(0,\infty]$ and $\lambda\in C([0,T),\R)$. The chordal Loewner equation driven by $\lambda$ is
\BGE \pa_t g_t(z)=\frac{2}{g_t(z)-\lambda_t},\quad 0\le t<T;\quad  g_0(z)=z.\label{chordal-eqn}\EDE
For each $z\in\C$, let $\tau_z$ be such that the maximal interval for $t\mapsto g_t(z)$ is $[0,\tau_z)$. Let $K_t=\{z\in\HH:\tau_z\le t\}$, i.e., the set of $z\in\HH$ such that $g_t(z)$ is not defined. Then $g_t$ and $K_t$, $0\le t<T$, are called the chordal Loewner maps and hulls driven by $\lambda$. It is known that $(K_t)$ is an increasing family of $\HH$-hulls with $\hcap(K_t)=2t$ and $g_t=g_{K_t}$ for $0\le t<T$. At $t=0$, $K_0=\emptyset$ and $g_0=\id$.

We say that $\lambda$ generates a chordal Loewner curve $\gamma$ if
$$\gamma(t):=\lim_{\HH\ni z\to \lambda(t)} g_t^{-1}(z)\in\lin{\HH}$$
exists for $0\le t<T$, and $\gamma$ is a continuous curve. We call such $\gamma$ the chordal Loewner curve driven by $\lambda$. If the such $\gamma$ exists, then for each $t$, $\HH\sem K_t$ is the unbounded component of $\HH\sem \gamma([0,t])$, and $g_t^{-1}$ extends continuously from $\HH$ to $\HH\cup\R$.
Since $\hcap(K_t)=2t$ for all $t$, we say that $\gamma$ is parametrized by half-plane capacity.

Another way to characterize the chordal Loewner hulls $(K_t)$ is using the notation of $\HH$-Loewner chain. A family of $\HH$-hulls: $K_t$, $0\le t<T$, is called an $\HH$-Loewner chain if
\begin{enumerate}
  \item $K_0=\emptyset$ and $K_{t_1}\subsetneqq K_{t_2}$ whenever $0\le t_1<t_2<T$;
  \item for any fixed $a\in[0,T)$, the diameter of $K_{t+\eps}/ K_t$ tends to $0$ as $\eps\to 0^+$, uniformly in $t\in[0,a]$.
\end{enumerate}
An $\HH$-Loewner chain $(K_t)$ is said to be normalized if $\hcap(K_t)=2t$  for each $t$. The following proposition is a result in \cite{LSW1}.
\begin{Proposition}
Let $T\in(0,\infty]$.  The following are equivalent.
\begin{enumerate}
    \item [(i)] $K_t$, $0\le t<T$, are chordal  Loewner hulls driven by some $\lambda\in C([0,T))$.
    \item [(ii)] $K_t$, $0\le t<T$, is a normalized $\HH$-Loewner chain.
\end{enumerate}
If either of the above holds,   we have
\BGE  \{ \lambda(t)\}=\bigcap_{\eps>0}\overline{K_{t+\eps}/K_t}, \quad 0\le t<T.  \EDE
If $K_t$, $0\le t<T$, is any $\HH$-Loewner chain, then the function $u(t):=\hcap(K_t)/2$, $0\le t<T$, is continuous and strictly increasing with $u(0)=0$, which implies that $K_{u^{-1}(t)}$, $0\le t<u(T)$, is a normalized $\HH$-Loewner chain. \label{Loewner-eqn-chain}
\end{Proposition}

For $\kappa>0$, chordal  SLE$_\kappa$ is defined by solving the chordal   Loewner equation with $\lambda(t)=\sqrt\kappa B(t)$, where $B(t)$ is a Brownian motion. The chordal Loewner curve $\gamma$ driven by this driving function a.s.\ exists, and satisfies $\lim_{t\to\infty}\gamma(t)=\infty$. So it is called a chordal SLE$_\kappa$ curve in $\HH$ from $0$ to $\infty$. It satisfies that, if $\kappa\in(0,4]$, $\gamma$ is simple, and $K_t=\gamma((0,t])$;  if $\kappa\ge 8$, $\gamma$ is space-filling, i.e., visits every points in $\lin\HH$; if $\kappa\in(4,8)$, $\gamma$ is neither simple nor space-filling, and every bounded subset of $\lin\HH$ is contained in  $K_t$ for some finite $t> 0$.

Via conformal maps, we may define SLE$_\kappa$ curve in any simply connected domain $D$ from one prime end $a$ to another prime end $b$. Recall that we use $\mu^\#_{D;a\to b}$ to denote the law of such curve (modulo a time change).

Now we review the definition of chordal SLE in multiply connected domains in \cite{Law-mult}. The laws of such SLE are no longer probability measures, but finite or $\sigma$-finite measures.
We will use the following notation. Suppose $D$ is a simply connected domain with two distinct prime ends $a$ and $b$. Let $U\subset D$ be an open neighborhood of both $a$ and $b$ in $D$. We define
\BGE \mu^D_{U;a\to b}={\bf 1}_{\{\cdot\cap (D \sem  U)=\emptyset\}}e^{\cc \mu^{lp}({\cal L}_{D}(\cdot, D\sem  U))}\cdot \mu^\#_{D;a\to b}.\label{chordal-BL}\EDE

\begin{Proposition}
  Let $  U$ and $  V$ be open neighborhoods of $\R\cup\{\infty\}$ in $\HH$. Suppose $  W:  U\conf V$ extends conformally across $\R\cup\{\infty\}$ such that $W(\R)=\R$ and $W(\infty)=\infty$.  Then for any $x\in\R$,
  $$\mu^{\HH}_{  V;  W(x)\to \infty}=|  W'(x)\cdot  W'(\infty) |^{-\frac{6-\kappa}{2\kappa}}   W(\mu^{\HH}_{  U;x\to \infty}),$$
 where $  W'(\infty):=(J\circ   W\circ J)'(0)$ with $J(z):=-1/z$.  \label{prop-mult}
\end{Proposition}
\begin{proof}
This proposition was proved in \cite[Section 4.1]{Law-mult} for $\kappa\in(0,4]$ by considering simply connected subdomains of $U$. In this proof, we assume that $\kappa\in(4,8)$. The proof is similar to those of Theorem \ref{Thm-subdomain} and Lemma \ref{lemma-loop}, and uses a standard argument that originated in \cite{LSW-8/3}. WLOG, we may assume that $x=0$ and $W(0)=0$. Let $P_a$ denote the multiplication map $z\mapsto az$. By conformal invariance of chordal SLE and Brownian loop measure, we know that $\mu^{\HH}_{ P_a(V);  0\to \infty}={P}_a(\mu^{\HH}_{  V;  0\to \infty})$ for any $a>0$. Since $(aW)'(0)\cdot (aW)'(\infty)=W'(0)\cdot  W'(\infty)$, we may assume that $W'(\infty)=1$ by replacing $W$ with $aW$ for some $a>0$.

Let $\gamma$ be a chordal SLE$_\kappa$ curve in $\HH$ from $0$ to $\infty$ with driving function $\lambda_t=\sqrt\kappa B_t$. Let $g_t$ and $K_t$, $0\le t<\infty$, be the chordal Loewner maps and hulls, respectively, driven by $\lambda$.

  Let $\tau_U$ be the first time that $\gamma$ exits $U$. Then $\beta(t):=W(\gamma(t))$ is well defined for $0\le t<\tau_U$. For each $0\le t<\tau_U$, let $L_t$ be the $\HH$-hull such that $\HH\sem L_t$ is the unbounded connected component of $\HH\sem \beta([0,t])$. If $K_t\subset U$, then $L_t=W(K_t)$. Since $\kappa\in(4,8)$, $K_t$ may swallow some relatively clospen subset of $\HH\sem U$ before the time $\tau_U$, and $W(K_t)$ is not defined at that time. Using the conformal invariance of extremal length, we can see that $(L_t)$ is an $\HH$-Loewner chain (even after $K_t$ intersects $\HH\sem U$). From Proposition \ref{Loewner-eqn-chain}, we may reparametrize the family $(L_t)$ using the function
  $u(t)=\hcap(L_t)/2$ to get a family of chordal Loewner hulls. Let $\sigma_s$, $0\le s<S:=u(\tau_U)$, be the driving function for the normalized $(L_s)$. Let $h_s$, $0\le s<S$, be the corresponding chordal Loewner maps.
  We also  reparametrize $\beta$ using $u$. Then $\beta$ is the chordal Loewner curve driven by $\sigma$, and $\beta_{u(t)}=W(\gamma(t))$, $0\le t<\tau_U$.

For $0\le t<\tau_U$, define $  U_t=g_t(  U\sem K_t)$, $  V_t=h_{u(t)}(  V\sem L_{u(t)})$, and $W_t=h_{u(t)}\circ W\circ g_t^{-1}$. Then $  U_t$ and $  V_t$ are open neighborhoods of $\R\cup\{\infty\}$ in $\HH$, $  W_t: U_t\conf V_t$, and satisfies that, if $z\in  U_t$ tends to $\R$ or $\infty$, then $  W_t$ tends to $\R$ or $\infty$, respectively. By Schwarz reflection principle, $  W_t$ extends conformally across $\R$, and maps $\R$ onto $\R$. Since $W,g_t,h_{u(t)}$ all fix $\infty$, and have derivative $1$ at $\infty$, $W_t$ also satisfies this property.

By the continuity of $ g_t$ and $ h_{u(t)}$ in $t$ and the maximal principle, we know that the extended  $  W_t$ is continuous in $t$ (and $z$). Fix $0\le t<\tau_U$. Let $\eps\in(0,\tau_U-t)$.  Now $K_{t+\eps}/K_t$ is an $\HH$-hull with $\HH$-capacity being $2\eps$; and $L_{u(t+\eps)}/ L_{u(t)}$ is an $\HH$-hull  with $\HH$-capacity being $2u(t+\eps)-2u(t)$. Since $  W_t(K_{t+\eps}/K_t)=L_{u(t+\eps)}/ L_{u(t)}$, using Propositions \ref{chordal-hull} and \ref{Loewner-eqn-chain}, we get
\BGE \sigma_{u(t)}=  W_t(\lambda_t),
 \label{W(lambda)-chordal}\EDE
and $u_+'(t)=   W_t'({\lambda_t})^2$.
Using the continuity of $W_t$ in $t$, we get
\BGE u'(t)=  W_t'(\lambda_t)^2.\label{u'-chordal}\EDE
 Thus, $  h_{u(t)}$ satisfies the equation
\BGE \pa_t h_{u(t)}(z)=  \frac{2   W_t'(\lambda_t)^2}{ h_{u(t)}(z)-\sigma_{u(t)}}.\label{pahut-chordal}\EDE

From the definition of $  W_t$, we get the equality
\BGE   W_t\circ  g_t(z)=h_{u(t)}\circ   W(z), \quad z \in   U\sem K_t.\label{circ-chordal}\EDE
Differentiating this equality w.r.t.\ $t$ and using (\ref{chordal-eqn},\ref{pahut-chordal}), we get
$$\pa_t  W_t( g_t(z))+\frac{2  W_t'( g_t(z))}{ g_t(z)-\lambda_t}=\frac{2  W_t'(\lambda_t)^2}{ h_{u(t)}\circ   W(z)-\sigma_{u(t)}},\quad z\in U\sem K_t.$$
Combining this formula with (\ref{W(lambda)-chordal},\ref{circ-chordal}) and replacing $g_t(z)$ with $w$, we get
\BGE \pa_t  W_t(w)=\frac{2   W_t'(\lambda_t)^2}{  W_t(w)-  W_t(\lambda_t)}-\frac{2  W_t'(w)}{w-\lambda_t},\quad w\in   U_t.
\label{patWt-chordal}\EDE
Letting $  U_t\ni w\to \lambda_t$ in (\ref{patWt-chordal}), we get
\BGE \pa_t  W_t(\lambda_t)=-3  W_t''(\lambda_t).\label{-3-chordal}\EDE
Differentiating (\ref{patWt-chordal}) w.r.t.\ $w$ and letting  $ U_t\ni w\to \lambda_t$, we get
\BGE \frac{\pa_t  W_t'(\lambda_t)}{  W_t'(\lambda_t)}=\frac 12\Big(\frac{  W_t''(\lambda_t)}{  W_t'(\lambda_t)}\Big)^2-\frac 43\frac{  W_t'''(\lambda_t)}{  W_t'(\lambda_t)}.\label{1243-chordal}\EDE

Combining (\ref{W(lambda)-chordal},\ref{-3-chordal}), and using It\^o's formula and that $\lambda_t=\sqrt\kappa B_t$, we see that $\sigma_{u(t)}$ satisfies the SDE
\begin{align}
   d\sigma_{u(t)}=    W_t'(\lambda_t)\sqrt\kappa dB_t    +\Big(\frac\kappa 2-3\Big)   W_t''(\lambda_t)dt.\label{dYtau-chordal}
\end{align}
Combining (\ref{1243-chordal}) with $\lambda_t=\sqrt{\kappa}B_t$ and using It\^o's formula, we get
\begin{align}
  \frac{d  W_t'(\lambda_t)}{  W_t'(\lambda_t)}=\frac{  W_t''(\lambda_t)}{  W_t'(\lambda_t)}\sqrt\kappa dB_{t}
  +\frac 12\Big(\frac{  W_t''(\lambda_t)}{  W_t'(\lambda_t)}\Big)^2dt+\Big(\frac \kappa 2-\frac 43\Big)\frac{  W_t'''(\lambda_t)}{  W_t'(\lambda_t)}\,dt . \label{dW'-chordal}
\end{align}

Let $(Sf)(z)=\frac{f'''(z)}{f'(z)}-\frac 32(\frac{f''(z)}{f'(z)})^2$ be the Schwarzian derivative of $f$. Using (\ref{dW'-chordal}) and It\^o's formula, we see that
\begin{align}
  \frac{d   W_t'(\lambda_t)^{\frac{6-\kappa}{2\kappa}}}{    W_t'(\lambda_t)^{\frac{6-\kappa}{2\kappa}}}=\frac{6-\kappa}{2 } \frac{  W_t''(\lambda_t)}{  W_t'(\lambda_t)} \frac{dB_t}{\sqrt\kappa}   +\frac{\cc}6 S( W_t)(\lambda_t)dt.\label{W'lambda-chordal}\
\end{align}
So we get the following positive continuous local martingale
\BGE M_t:=W_t'(\lambda_t)^{\frac{6-\kappa}{2\kappa}}\exp\Big(-\int_0^t \frac{\cc}6 S( W_s)(\lambda_s)ds\Big), \label{Mt-chordal}\EDE
which satisfies the SDE
\BGE  \frac{d M_t}{  M_t}= \frac{6-\kappa}{2 } \frac{  W_t''(\lambda_t)}{  W_t'(\lambda_t)} \frac{dB_t}{\sqrt\kappa},\quad 0\le t<\tau_U.\label{dMt-chordal}\EDE

We claim that the following equality holds: for any $0\le T<\tau_U$,
\BGE  \int_{0}^T \frac 16 S(W_t)(\lambda_t)dt=\mu^{\lloop}({\cal L}_{\HH}(\beta([0,u(T)],\HH\sem V))-\mu^{\lloop}({\cal L}_{\HH}(\gamma([0,T]), \HH\sem U)).\label{S-loop-chordal}\EDE
Note that this is similar to Lemma \ref{lemma-loop}.
To prove (\ref{S-loop-chordal}), we use the Brownian bubble analysis of Brownian loop measure. Let $\mu^{\bub}_{ {x_0}}$ denote the Brownian bubble measure in $\HH$ rooted at $x_0\in\R$ as defined in \cite{loop}, from which we know, for any $0\le T<\tau_U$,
\begin{align}
   \mu^{\lloop}({\cal L}_{\HH}(\gamma([0,T]), \HH\sem U))=&\int_{0}^T \mu^{\bub}_{\lambda_t}({\cal L}(\HH\sem U_t))dt;\label{loop-bubbleU-chordal}\\
    \mu^{\lloop}({\cal L}_{\HH}(\beta([0,u(T)],\HH\sem V))=&\int_{0}^{u(T)} \mu^{\bub}_{\sigma_s}({\cal L}(\HH\sem V_s))ds \nonumber \\
     =&\int_{0}^T   W_t'(\lambda_t)^2 \mu^{\bub}_{\sigma_{u(t)}}({\cal L}(\HH\sem V_{u(t)}))dt.\label{loop-bubbleV-chordal}
\end{align}

If $U^*$ is a subdomain of $\HH$ that contains a neighborhood of $\R\cup\{\infty\}$ in $\HH$, we let $P^{U^*}_{x_0}$ denote the Poisson kernel in $U^*$ with the pole at $x_0\in\R$. Especially, $P^{\HH}_{x_0}(z)=\Imm \frac{-1/\pi }{z-x_0}$. From \cite{loop} we know
\begin{align*}
  \mu^{\bub}_{{\lambda_t}}({\cal L}({\HH\sem U_t})) =\lim_{U_t\ni z\to {\lambda_t}}\frac{1}{|z-{\lambda_t}|^2}\Big(1-\frac{P^{U_t}_{\lambda_t}(z)}{P^{\HH}_{\lambda_t}(z)}\Big)
\end{align*}
Similarly, using (\ref{W(lambda)-chordal}) and that $W_t:U_t\conf V_{u(t)}$, we get
\begin{align*}
  \mu^{\bub}_{\sigma_{u(t)}}({\cal L}({\HH\sem V_{u(t)}}))&= \lim_{V_{u(t)}\ni w\to {\sigma_{u(t)}}}\frac{1}{|w-{\sigma_{u(t)}}|^2}\Big(1-\frac{P^{V_{u(t)}}_{\sigma_{u(t)}}(w)}{ P^{\HH}_{\sigma_{u(t)}}(w)}\Big)\\
  &= \lim_{U_{t}\ni z\to {\lambda_t}} \frac{1}{| W_t(z) - W_t(\lambda_t)|^2}\Big(1-\frac{ P^{V_{u(t)}}_{\sigma_{u(t)}}\circ W_t(z)}{P^{\HH}_{\sigma_{u(t)}}\circ  W_t(z)}\Big)\\
  &= \lim_{ U_{t}\ni z\to {\lambda_t}} \frac{ W_t'(\lambda_t)^{-2}}{|z-\lambda_t|^2}\Big(1-\frac{ W_t'(\lambda_t)^{-1} P^{U_t}_{\lambda_t}(z) }{P^{\HH}_{\sigma_{u(t)}}\circ  W_t(z)}\Big).
\end{align*}
Combining the above two formulas and using some tedious but straightforward computation involving power series expansions, we get $$ W_t'(\lambda_t)^2 \mu^{\bub}_{{\sigma_{u(t)}}}({\cal L}({\HH\sem V_{u(t)}}))-\mu^{\bub}_{{\lambda_t}}({\cal L}({\HH\sem U_t}))=\frac 16 S(W_t)(\lambda_t).$$
This together with (\ref{loop-bubbleU-chordal},\ref{loop-bubbleV-chordal}) completes the proof of (\ref{S-loop-chordal}).

Since $\gamma$ is continuous and tends to $\infty$, from (\ref{finite},\ref{S-loop-chordal}), we see that, on the event that $\gamma\cap (\HH\sem U)=\emptyset$, the improper integral $\int_0^\infty \frac 16 S(W_s)(\lambda_s)ds$ converges to $\mu^{\lloop}({\cal L}_{\HH}(\beta ,\HH\sem V))-\mu^{\lloop}({\cal L}_{\HH}(\gamma, \HH\sem U))$.

We claim that $\lim_{t\to \infty} W_t'(\lambda_t)=1$ on the event that $\gamma\cap \HH\sem U=\emptyset$. Since $\kappa\in(4,8)$, there is $t_0\in(0,\infty)$ such that $\HH\sem U\subset K_{t_0}$. Then for $t\ge t_0$, $U\sem K_t=\HH\sem K_t$, and so $U_t=\HH$. Similarly, $V_t=\HH$ for $t\ge t_0$. Thus, for $t\ge t_0$, $W_t:(\HH;\infty)\conf (\HH;\infty)$ and $W_t'(\infty)=1$,
which implies that  $W_t'(\lambda_t)=1$. So the claim is proved.

From the above we see that $M_\infty:=\lim_{t\to \infty} M_t=e^{\cc \mu^{\lloop}({\cal L}_{\HH}(\gamma, \HH\sem U))}/e^{\cc \mu^{\lloop}({\cal L}_{\HH}(W(\gamma) ,\HH\sem V))}$ on the event that $\gamma\cap (\HH\sem U)=\emptyset$. Thus, $M_t$, $0\le t<\infty$, is bounded on this event.

For $n\in\N$,  let $T_n$ be the first time that $\gamma$ hits $\HH\sem U$ or $M_t\ge n$, whichever happens first. Then $T_n$ is a stopping time, and $M_t$ up to $T_n$ is bounded by $n$. Thus, $\EE[M_{T_n}]=M_0=W'(0)^{\frac{6-\kappa}{2\kappa}}$. Weighting the underlying probability measure by $M_{T_n}/M_0$, we get a new probability measure. By Girsanov Theorem and (\ref{dMt-chordal}), we find that
$$ \ha B _t:=B_t-\frac 1{\sqrt\kappa} \int_0^t  \frac{6-\kappa}{2} \frac{  W_s''(\lambda_s)}{ W_s'(\lambda_s)}ds,\quad 0\le t< T_n,$$
is a Brownian motion under the new probability measure. From (\ref{dYtau-chordal}), we get
$$d\sigma_{u(t)}=    W_t'(\lambda_t)\sqrt\kappa d\ha B_t,\quad 0\le t< T_n.$$
From (\ref{u'-chordal}) we see that, under the new probability measure, $\sigma_s/\sqrt\kappa$, $0\le s< u(T_n)$, is a Brownian motion, and so $\beta_s$, $0\le s\le u(T_n)$, is a chordal SLE$_\kappa$ curve in $\HH$ from $0$ to $\infty$, stopped at $u(T_n)$.
let $E_n$ denote the event that $\gamma\cap (\HH\sem U)=\emptyset$ and $M_t\le n$ for $0\le t<\infty$; and let $F_n$ denote the event that $W^{-1}(\beta)\in E_n$. Then on the event $E_n$, $T_n=u(T_n)=\infty$ , and $M_{T_n}/M_0=M_\infty/W'(0)^{\frac{6-\kappa}{2\kappa}}$. From the above argument, we get
$${\bf 1}_{F_n}\cdot \mu^\#_{\HH;0\to\infty}=W(W'(0)^{-\frac{6-\kappa}{2\kappa}}e^{\cc \mu^{\lloop}({\cal L}_{\HH}(\cdot, \HH\sem U ))}/e^{\cc \mu^{\lloop}({\cal L}_{\HH}(W(\cdot),\HH\sem V))} {\bf 1}_{E_n}\cdot \mu^\#_{\HH;0\to\infty} ).$$
Since $\mu^\#_{\HH;0\to\infty}$-a.s.\ $\bigcup E_n=\{\cdot\cap \HH\sem U=\emptyset\}$ and $\bigcup F_n=\{\cdot\cap \HH\sem V=\emptyset\}$, the above formula holds with $E_n$ and $F_n$ replaced by $\{\cdot\cap \HH\sem U=\emptyset\}$ and $\{\cdot\cap \HH\sem V=\emptyset\}$, respectively. The proposition now follows from this formula since we assumed that $W'(\infty)=1$.
\end{proof}

\begin{Remark}
The above proof also works for $\kappa\in(0,4]$ except that  $\lim_{t\to \infty} W_t'(\lambda_t)=1$ on the event $\gamma\cap (\HH\sem U)=\emptyset$ requires a little bit more work to prove.
\end{Remark}

\begin{Lemma}
  Let $K$ and $L$ be two non-degenerate interior hulls. Let $U,V\subset \ha\C$  be open neighborhoods of $K$ and $L$, respectively. Suppose $W:(U;K)\conf (V;L)$. Let $a$ and $b$ be distinct prime ends of $\ha\C\sem K$. Then $W(a)$ and $W(b)$ are distinct prime ends of $\ha\C\sem L$.
Let $g_K:\ha\C\sem K\conf \D^*$ and $g_L:\ha\C\sem L\conf \D^*$.
 Suppose $g_K(a)=e^{i\lambda}$, $g_K(b)=e^{iq}$, $g_L(W(a))=e^{i\sigma}$, and $g_L(W(b))=e^{ip}$ for some $\lambda,q,\sigma,p\in\R$. Let $W_K=g_L\circ W\circ g_K^{-1}$. Extend $W_K$ conformally across $\TT$.
  Then we have
   $$\mu^{\ha C\sem L}_{V\sem L;W(a)\to W(b)}= \Big|\frac{\sin_2(\sigma-p)}{\sin_2(\lambda-q)}\Big|^{\frac 6\kappa -1}\cdot  |W_K'(e^{i\sigma})W_K'(e^{iq}) |^{-\frac{6-\kappa}{2\kappa}} \cdot W(\mu^{\ha\C\sem K}_{U\sem K;a\to b}).$$ \label{lem-mult}
\end{Lemma}
\begin{proof}
   Let $\phi(z)=i\frac{z+e^{iq}}{z-e^{iq}}$ and $\psi(z)=i\frac{z+e^{ip}}{z-e^{ip}}$.
   Then $\phi:(\D^*;e^{i\lambda},e^{iq})\conf(\HH;\cot_2(\lambda-q),\infty)$ and $\psi:(\D^*;e^{i\sigma},e^{ip})\conf(\HH;\cot_2(\sigma-p),\infty)$.
    Let $U_K=g_K(U\sem K)$ and $V_L=g_L(V\sem L)$. Then $U_K$ and $V_L$ are open neighborhoods of $\TT$ in $\D^*$,  $W_K:U_K\conf V_L$, and can be extended conformally across $\TT$. The extended $W_K$ maps $\TT$ onto $\TT$, and maps $e^{i\lambda}$ and $e^{iq}$ to $e^{i\sigma}$ and $e^{ip}$, respectively.
 Let $\ha U_K= \phi(U_K)$, $\ha V_L=\psi(V_L)$, and $\ha W_K=\psi\circ W_K\circ \phi^{-1}$. Then $\ha U_K$ and $\ha V_L$ are open neighborhoods of $\R\cup\{\infty\}$ in $\HH$,
 and $\ha W_K:(\ha U_K;\R,\cot_2(\lambda-q),\infty)\conf (\ha V_K;\R,\cot_2(\sigma-p),\infty)$.
 From Proposition \ref{prop-mult}, we have
   \begin{align*}
       \mu^{\HH}_{\ha V_L;\cot_2(\sigma-p)\to \infty}
     = |\ha W_K'(\cot_2(\lambda-q))\ha W'(\infty)|^{-\frac{6-\kappa}{2\kappa}} \ha W_K(\mu^{\HH}_{\ha U_K;\cot_2(\lambda-q)\to \infty}).
   \end{align*}
 We have $\phi\circ g_K:(\ha\C\sem K,U\sem K;a,b)\conf (\HH,\ha U_K;\cot_2(\lambda-q),\infty)$ and $\psi\circ g_L:(\ha\C\sem L,V\sem L;W(a),W(b))\conf (\HH,\ha V_L;\cot_2(\sigma-p),\infty)$.
   From the conformal invariance of chordal SLE and Brownian loop measure, we have
   \begin{align*}
     \phi\circ g_K(\mu^{\ha\C\sem K}_{U\sem K;a\to b})= \mu^{\HH}_{\ha U_K;\cot_2(\lambda-q)\to \infty},\quad
      \psi\circ g_L(\mu^{\ha\C\sem L}_{V\sem L;W(a)\to W(b)})     = \mu^{\HH}_{\ha V_L;\cot_2(\sigma-p)\to \infty}.
   \end{align*}
Combining the above displayed formulas and the fact that $\ha W_K=\psi\circ g_L\circ W\circ g_K^{-1}\circ \phi^{-1}$, we see that it suffices to prove that
$$ \Big|\frac{\sin_2(\sigma-p)}{\sin_2(\lambda-q)}\Big|^{-2}\cdot  |W_K'(e^{i\sigma})W_K'(e^{iq}) | = |\ha W_K'(\cot_2(\lambda-q))\ha W_K'(\infty)|  .$$
To see this, one may check that $|\phi'(e^{i\lambda})|=|\sin_2(\lambda-q)|^{-2}/2$, $|\psi'(e^{i\sigma})|=|\sin_2(\sigma-p)|^{-2}/2$; and with $J(z):=-1/z$, $|(J\circ \phi)'(e^{iq})|=|(J\circ \psi)'(e^{ip})|=1/2$.
\end{proof}

\begin{Lemma}
	Let $K$ and $L$ be two $\HH$-hulls. Let $U$ and $V$  be open neighborhoods of $ \R\cup\{\infty\}$ in $\HH$ such that $K\subset U$ and $L\subset V$. Suppose $W:(U;\R,\infty,K)\conf (V;\R,\infty,L)$.  Let $a$ and $b$ be distinct prime ends of $\HH\sem K$ that lie on $\pa K$. Then $W(a)$ and $W(b)$ are distinct prime ends of $\HH\sem L$ that lie on $\pa L$.
	Let $g_K:\HH\sem K\conf \HH$ and $g_L:\HH\sem L\conf \HH$. Suppose $g_K(a)={\lambda}$, $g_K(b)={q}$, $g_L(W(a))={\sigma}$, and $g_L(W(b))={p}$ for some $\lambda,q,\sigma,p\in\R$. Let $W_K=g_L\circ W\circ g_K^{-1}$. Extend $W_K$ conformally across $\R$.
	Then we have
	$$\mu^{\HH\sem L}_{V\sem L;W(a)\to W(b)}= \Big|\frac{ \sigma-p}{ \lambda-q }\Big|^{\frac 6 \kappa-1}\cdot  |W_K'({\sigma})W_K'({q}) |^{-\frac{6-\kappa}{2\kappa}} \cdot W(\mu^{\HH\sem K}_{U\sem K;a\to b}).$$ \label{lem-mult'}
\end{Lemma}
\begin{proof}
	The proof is similar to that of Lemma \ref{lem-mult} except that here we use the functions $\phi(z)=-\frac{z+q}{z-q}$ and $\psi(z)=-\frac{z+p}{z-p}$, which map $\HH$ conformally onto $\HH$.
\end{proof}


\begin{thebibliography}{00}
\bibitem{Ahl} Lars V.\ Ahlfors. {\it Conformal invariants: topics in geometric function theory}. McGraw-Hill Book Co., New York, 1973.
\bibitem{AlbertsSheffield} Tom Alberts, Scott Sheffield.  The Covariant Measure of SLE on the Boundary, {\it Probab. Theory Rel.}, {\bf 149}:331-371, 2011.
\bibitem{BF-multiply} Robert O.\ Bauer and Roland M.\ Friedrich. Stochastic Loewner evolution in multiply connected domains. {\it C.\ R.\ Acad.\ Sci.\ Paris Ser.\ I}, {\bf 339}(8): 579-584, 2004.
\bibitem{Bf} Vincent Beffara.  The dimension of SLE curves. {\it Annals of Probab.}, 36:1421-1452, 2008.
\bibitem{Dubedat-loop} St\'ephane Benoist and Julien Dub\'edat. An SLE$_2$ loop measure. {\it Ann.\ I.\ H.\ Poincare-Pr.}, 52(3):1406-1436, 2016.
\bibitem{Benoist} St\'ephane Benoist. Classifying conformally invariant loop measures. 	arXiv:1608.03950, 2016.
\bibitem{Benoist-Dubedat} St\'ephane Benoist and Julien Dub\'edat. Building SLE$_\kappa$ loop measures for $\kappa<4$. In preparation.
\bibitem{Fie} Laurence S.\ Field. Two-sided radial SLE and length-biased chordal SLE.   arXiv:1601.03374, 2016.
\bibitem{normalize} Laurence S.\ Field and Gregory F.\ Lawler. Reversed Radial SLE and the Brownian Loop Measure. {\it J.\ Stat.\ Phys.}, 150(6):1030-1062, 2013.
\bibitem{Talk} Laurence S.\ Field and Gregory F.\ Lawler. SLE loops rooted at an interior point. In preparation.
\bibitem{future} Laurence S.\ Field, Gregory F.\ Lawler and Dapeng Zhan. Brownian loop measures on Riemann surfaces. In preparation.
\bibitem{mating} Nina Holden and Xin Sun. SLE as a mating of trees in Euclidean geometry. arXiv:1610.05272, 2016.
\bibitem{Kal} Olav Kallenberg. {\it Foundations of Modern Probability}. Springer, 2011.
\bibitem{KK}  Adrien Kassel and Richard Kenyon. Random curves on surfaces induced from the laplacian determinant.
In preprint, arXiv:1211.6974, 2012.
\bibitem{KW} Antti Kemppainen and Wendelin Werner. The nested simple conformal loop ensembles in the Riemann sphere. {\it Probab.\ Theory Related Fields}, (165)3:835-866, 2016.
\bibitem{KS} Maxim Kontsevich and Yuri Suhov. On Malliavin measures, SLE, and CFT. {\it Proc. Steklov Inst. Math.}, 258:100-146, 2007.
\bibitem{Law-mult} Gregory F.\ Lawler. Defining SLE in multiply connected domains with the Brownian loop measure. In preprint, arXiv:1108.4364, 2011.
\bibitem{Law-real} Gregory F.\ Lawler. Minkowski content of the intersection of a Schramm-Loewner evolution (SLE)
curve with the real line. {\it J.\ Math.\ Soc.\ Japan}, 67(4):1631-1669, 2015.
\bibitem{Law1} Gregory F.\ Lawler. {\em Conformally Invariant Processes in the Plane}, Amer. Math. Soc, 2005.
\bibitem{LR} Gregory F.\ Lawler and Mohammad A.\ Rezaei. Minkowski content and natural parametrization for the Schramm-Loewner evolution. {\it Annals of Probab.}, 43(3):1082-1120, 2015.
\bibitem{LSW1} Gregory F.\ Lawler, Oded Schramm and Wendelin Werner.
Values of Brownian intersection exponents I: Half-plane exponents.
{\it Acta Math.}, 187(2):237-273, 2001.
\bibitem{LSW-8/3} Gregory F.\ Lawler, Oded Schramm and Wendelin Werner. Conformal restriction: the chordal case, {\it J.\ Amer.\ Math.\ Soc.}, 16(4): 917-955, 2003.
\bibitem{LS} Gregory F.\ Lawler and Scott Sheffield. A natural parametrization for the Schramm-Loewner evolution.  {\it Annals  of Probab.}, 39:1896-1937, 2011.
\bibitem{LERW-NP2} Gregory F.\  Lawler and  Fredrik Viklund. Convergence of loop-erased random walk in the natural parametrization. arXiv:1603.05203, 2016.
\bibitem{loop} Gregory F.\ Lawler and Wendelin Werner. The Brownian loop soup. {\it Probab.\ Theory Related Fields},
128(4):565-588, 2004.
\bibitem{LZ}  Gregory F.\ Lawler and Wang Zhou. SLE curves and natural parametrization. {\it Ann.\ Probab.}, 41(3A):1556-1584, 2013.
\bibitem{Malliavin} Paul Malliavin. The canonic diffusion above the diffeomorphism group of the circle. {\it C.\ R.\ Acad.\ Sci.\ Paris Ser.\ I}, 329(4):325-329, 1999.
\bibitem{MS1} Jason Miller and Scott Sheffield. Imaginary Geometry I: intersecting SLEs. {\it Prob. Theory Related Fields}, 164(3):553-705, 2016.
\bibitem{MS3} Jason Miller and Scott Sheffield. Imaginary Geometry III: reversibility of SLE$_\kappa$ for $\kappa\in (4, 8)$. {\it Annals of Math.}, 184(2):455-486, 2016.
\bibitem{MS4} Jason Miller and Scott Sheffield. Imaginary Geometry IV: interior rays, whole-plane reversibility, and space-filling trees. In preprint, arXiv:1302.4738, 2013.
\bibitem{RY} Daniel Revuz and Marc Yor. {\it Continuous Martingales
and Brownian Motion}. Springer, Berlin, 1991.
\bibitem{Green} Mohammad A.\ Rezaei and Dapeng Zhan. Green's function for chordal SLE curves. In preprint, arXiv:1607.03840, 2016.
\bibitem{RS} Steffen Rohde and Oded Schramm. Basic properties of SLE. {\it Ann.\  Math.}, 161:879-920, 2005.
\bibitem{Sch} Oded Schramm. Scaling limits of loop-erased random walks and uniform spanning trees, {\it Israel J. Math}. {\bf 118}, 221-288, 2000.
\bibitem{CLE} Scott Sheffield and Wendelin Werner. Conformal loop ensembles: the Markovian characterization and the loop-soup construction. {\it Ann. Math.}, 176(3):1827-1917, 2012.
\bibitem{SW} Oded Schramm and David B.\ Wilson. SLE coordinate changes. {\it New York Journal of Mathematics}, 11:659--669, 2005.
\bibitem{Wer-loop} Wendelin Werner. The conformally invariant measure on self-avoiding loops. {\it J.\ Am.\ Math.\ Soc.}, 21:137-169, 2008.
\bibitem{decomposition} Dapeng Zhan. Decomposition of Schramm-Loewner evolution along its curve. In preprint, arXiv:1509.05015.
\bibitem{Holder} Dapeng Zhan. Optimal H\"older Continuity and Dimension Properties for SLE with Minkowski Content Parametrization. In preprint, arXiv:1706.05603.
\bibitem{tip} Dapeng Zhan. Ergodicity of the tip of an SLE curve. {\it Prob. Theory Related Fields}, 164(1):333-360, 2016.
\bibitem{whole} Dapeng Zhan. Reversibility of whole-plane SLE. {\it Probab.\ Theory Rel.}, 161(3):561-618, 2015.
\bibitem{reversibility} Dapeng Zhan. Reversibility of chordal SLE. {\it Ann.\ Probab.}, 36(4):1472-1494, 2008.
\bibitem{Thesis} Dapeng Zhan. Random Loewner chains in Riemann surfaces. Ph.D dissertation. Caltech, 2004.
\end{thebibliography}
\end{document}